\documentclass{amsart}
\usepackage{amssymb,amsfonts,amsmath,graphicx}
\usepackage[alphabetic,y2k,lite]{amsrefs}
\usepackage{fullpage,cite,mystyle}
\usepackage{MnSymbol,float,braket}
\usepackage[mathletters]{ucs}
\usepackage[utf8x]{inputenc}
\usepackage{enumitem}

\newcommand{\wdtld}[1]{\widetilde{#1}}
\newcommand{\wdht}[1]{\widehat{#1}}
\newcommand{\hattld}[1]{\widehat{\widetilde{#1}}}
\newcommand{\TT}{\mathbf{T}}
\newcommand{\north}{{*_N}}
\newcommand{\south}{{*_S}}
\newcommand{\HH}{{\mathbb{H}^3}}
\newcommand{\PSL}{{PSL(2,\CC)}}

\newcommand{\isomHH}{{\textrm{Isom}^+(\HH)}}

\newcommand{\cT}{{\mathcal{T}}}

\newcommand{\geomlabel}{{\Omega^{\text{geo}}}}
\newcommand{\geomw}[1]{{w_{#1}^{\text{geo}}}}
\newcommand{\geomu}[1]{{u_{#1}^{\text{geo}}}}
\DeclareMathOperator{\vol}{vol}
\DeclareMathOperator{\argmax}{argmax}

\newcommand{\cev}[1]{\overset{\leftarrow}{#1}}

\newcommand{\matw}[1]{\begin{pmatrix}
	0 & #1 \\
	1 & 0
	\end{pmatrix}}
\newcommand{\matu}[1]{\begin{pmatrix}
	1 & #1 \\
	0 & 1
	\end{pmatrix}}
\newcommand{\matwil}[1]{(0 \; #1; 1 \; 0)}
\newcommand{\matuil}[1]{(1 \; #1; 0 \; 1)}

\author{Alice Kwon \textsuperscript{\textdagger}}
\author{Byungdo Park \textsuperscript{\ddag}}
\author{Ying Hong Tham \textsuperscript{*}}

\address{\textsuperscript{\textdagger} \small Alice Kwon, Department of Science, SUNY Maritime, 6 Pennyfield Avenue, Bronx, NY 10465, USA}
\email{akwon@sunymaritime.edu}

\address{\textsuperscript{\ddag} \small Byungdo Park,
	Department of Mathematics Education, Chungbuk National University, Cheongju 28644, Republic of Korea} 
\email{byungdo@cbnu.ac.kr}

\address{\textsuperscript{*} \small Ying Hong Tham
	Fachbereich Mathematik, Universit\"at Hamburg, Bundesstra{\ss}e 55, 20146 Hamburg, Germany} 
\email{ying.hong.tham@uni-hamburg.de}

\thanks{This work was supported by Deutsche Forschungsgemeinschaft via the Cluster of Excellence EXC 2121 ``Quantum Universe'' - 390833306 and the National Research Foundation of Korea (NRF) grant funded by the Korean government (MSIT) (No. 2020R1G1A1A01008746).}

\makeatletter
\@namedef{subjclassname@2020}{%
  \textup{2020} Mathematics Subject Classification}
\makeatother
\subjclass[2020]{Primary 57K32; Secondary 57M50, 57K10, 52C26}
\keywords{Augmented link, Circle packing, Hyperbolic knot theory, Geometric triangulation, Thurston's gluing equations, Thurston's completeness equations}

\begin{document}


\title{Generalization of the Thistlethwaite--Tsvietkova Method}

\maketitle

\begin{abstract}
Thurston's equations determine the hyperbolic structure
of a 3-manifold with a triangulation.
In \cite{TTmethod}, an alternative method was developed
for link complements in $S³$ depending on the link diagram,
where a set of labels are associated to the vertices and edges of the link diagram,
and one attempts to solve a set of equations on the labels.
Under certain conditions, there exists a solution to these equations
that corresponds to the complete hyperbolic structure,
but in general it is difficult to determine which one it is.
We generalize this method to 3-manifolds with a polyhedral decomposition,
and show that solutions to the equations correspond to $\PSL$-representations
of the fundamental group, and that the solution with the largest volume corresponds
to the complete hyperbolic structure.
We also consider different classes of complements of links,
in particular links in the thickened torus and fully augmented links.
For the latter, we establish a correspondence between solutions satisfying some criteria
and circle packings realizing the region graph associated to the fully augmented link.
\end{abstract}

%

\tableofcontents

\section{Introduction}
\label{s:intro}

In \cite{TTmethod}, Thistlethwaite and Tsvietkova describe an 
alternative method for calculating the hyperbolic structure on the complement
of a link in $S³$, which we refer to as the \emph{TT method}.
They produce a set of equations depending on the link diagram
without resorting to a triangulation of the link complement;
the unknowns in these equations are meant to capture and quantify
parallel transport in the complete hyperbolic structure.
We refer to solutions to these equations as
\emph{algebraic solutions (to the TT method)};
under certain conditions, there is a unique algebraic solution which corresponds
to the complete hyperbolic geometry on the link complement,
and we refer to it as the \emph{geometric labeling}.
A major challenge to the usability of the TT method is that
it is generally difficult to know which of the algebraic solutions
is the geometric solution.
The goal of this paper is to provide some techniques to overcome this challenge.

We generalize the TT method to cusped hyperbolic 3-manifolds.
This generalized method depends on a choice of an ideal polyhedral decomposition
of the 3-manifold, which is analogous to the dependence of the
original TT method on a choice of a link diagram.
Indeed, the original TT method is the application of our generalization
to the Menasco decomposition of the link complement based on a given link diagram.
This generalization also naturally specializes to (the complements of)
other types of links, in particular to links in the thickened torus
and a class of links called \emph{fully augmented links (FALs)}.

We then describe a sort of developing map for an algebraic solution $Ω$,
which we call the \emph{geometric realization of $Ω$}
(see \secref{s:alg-to-geom-reconstruction}).
In brief, the geometric realization is a map that assigns
a point at infinity of hyperbolic space $\HH$
to each end of the universal cover of $M$,
such that the labels in $Ω$ are equal to the \emph{geometric labels}
between corresponding points at infinity.
The geometric realization serves as an intermediate step that connects
algebraic solutions to $\PSL$-representations of the fundamental group of $M$,
recovering the direct construction found already in \cite{TTmethod}
(see also \cite{intercusp}).
Indeed, based on \prpref{p:alg-soln-repn-class},
we can view the TT method as
``picking a basis to study $\PSL$-representations''. More precisely,
depending on the choice of polyhedral decomposition of $M$,
it picks out a subset of the set of $\PSL$-representations and simultaneously
provides a set of coordinates with which to locate these representations.
The geometric realization also facilitates the discussion in other contexts,
in particular it is a convenient tool to use in defining and discussing volume
(see \secref{s:alg-vol}).

This brings us to the first main contribution of the paper:
we prove that, under certain conditions on the link diagram
(the polyhedral decomposition in the generalization),
the algebraic solution with the maximal volume is the geometric labeling.
We reproduce the theorem here:
\\

\noindent\textbf{Theorem \ref{t:max-vol}}\emph{
Let $M$ be a cusped hyperbolic 3-manifold,
and let $τ$ be an ideal polyhedral decomposition of $M$
such that every edge is geodesic-like.
Fix meridians $μ_i$ for each end of $M$,
determining a canonical parametrization of $M$.
Then the algebraic solution with the maximal volume is the geometric labeling:
\[
\geomlabel = \underset{Ω: \text{ alg.\,sol.}}{\argmax} \  \vol(Ω)
\]
}

The proof relies on the results of \cite{repnvol},
which concerns volumes of $\PSL$-representations of a 3-manifold.
In particular, \cite{repnvol} proves a similar statement but for
$\PSL$-representations.
We translate these results into results about algebraic solutions
via the use of the geometric realization.

Our second main contribution concerns the application of the (generalized)
TT method to fully-augmented links (FALs) in $S³$ or $T² \times (-1,1)$.
From \cite{lackenby}, it is known that to every FAL is associated
a circle packing realizing the FAL's \emph{region graph}.
Such a circle packing is unique, but if we loosen the definition of circle packing
(e.g. drop the univalence condition),
then many circle packings are possible.
We find that these new circle packings actually correspond to algebraic solutions
satisfying some criteria (mostly derived from symmetry considerations).
We restate the main results concerning this correspondence as follows:
\\

\noindent\textbf{Theorem \ref{t:FAL-circ-packing}, Corollary \ref{c:geometric}}
\emph{There is a 1-to-1 correspondence between
algebraic solutions to the TT method for a FAL (satisfying some criteria)
and circle packings realizing the region graph associated to the FAL.
Moreover, we give two sets of additional criteria on algebraic solutions
that correspond to certain properties on circle packings,
namely local univalence and locally order-preserving respectively.\\
\indent
As a consequence, the algebraic solution that satisfies all the additional criteria
corresponds to a circle packing that is
locally univalent and locally order-preserving, which implies univalence.
In other words, such an algebraic solution must be the geometric labeling.
}
\\

Let us provide an overview of the paper.
We provide some background in \secref{s:background},
where we recall the original TT method (\secref{s:recap})
as well as fully-augmented links (\secref{s:FAL}) and
their relation with circle packings (\secref{s:FAL-circ}).
Next, in \secref{s:TT-3mfld},
we generalize the TT method to 3-manifolds with toric end,
and develop tools to study algebraic solutions.
We establish the connection between algebraic solutions and $\PSL$-representations
via geometric realizations (\secref{s:alg-to-geom-reconstruction}),
and prove the volume-maximzing property of the geometric labeling (\thmref{t:max-vol}).
Then in \secref{s:application-links},
we discuss the application of the generalized TT method
to links in the thickened torus $T² \times (-1,1)$,
as well as to FALs in $S³$ and $T² \times (-1,1)$.
We show that the geometric labeling for FALs satisfy some criteria
(Sections \ref{s:TT-FAL-S3-eqn-simplified}, \ref{s:TT-FAL-T2-eqn-simplified}).
Conversely, we show that these criteria are sufficient conditions
for an algebraic solution to be the geometric labeling (\corref{c:geometric}),
which is established by developing the correspondence between
algebraic solutions and circle packings
(\secref{s:FAL-circ-packing}), \thmref{t:FAL-circ-packing}).
Finally, we work through some examples in \secref{s:examples}.

We thank the reviewer for many helpful comments,
in particular for bringing our attention to \cite{repnvol}.

\subsection{Conventions and Notations}
\label{s:conventions}
\ \\
We denote the upper-half space by $\HH = \{(x,y,z) \,|\, z > 0\} ⊆ \RR³$,
We also implicitly identify the $x$-$y$-plane with $\CC$,
so that we may write $\HH = \{(w = x + yi,z) \,|\, z > 0\} =
\CC \times \RR_{>0}$.
We identify $\PSL$ with the group of orientation-preserving isometries of $\HH$,
denote by $\isomHH$, by acting on the boundary $∂\HH = \CC ∪ \{∞\} \simeq
\mathbb{P}¹$, that is,
the action of a matrix $(a \; b; c \; d) :=
\begin{pmatrix}
a & b\\
c & d
\end{pmatrix}
∈ \PSL$ on $\HH$
is the isometry whose action on the boundary is the fractional transformation
$w \mapsto \frac{aw + b}{cw + d}$.

Horospheres in $\HH$ will be oriented such that
if we remove the corresponding horoball,
then the orientation is the outward orientation;
in particular, the horosphere $\{z = 1\} ⊆ \HH$
is oriented so that its projection to $\CC \times \{0\} ⊆ ∂\HH$
is orientation preserving.
Note that this convention is opposite to that of \cite{TTmethod}.
We think that this convention makes more sense as it agrees with the
outward orientation, which is a generally held convention;
it is also convenient to have the projection to the $xy$-plane be
orientation-preserving, and in face a ``parametrization''
(as in \defref{d:param-horosphere}).
To convert between our convention and theirs,
simply apply complex conjugation to the labels.



\section{Background}
\label{s:background}

\subsection{Recap of the TT method}
\label{s:recap}
\ \\
Let us briefly recall the methods in \cite{TTmethod},
with some minor differences in conventions,
as well as some new terminology for clarification.
Just as Thurston's gluing/completeness equations can be set up
for any triangulation of a 3-manifold,
the TT method can be set up for any link diagram regardless of hyperbolicity
or choice of link diagram.
We will first describe the setup for an arbitrary link $L$
with an arbitrary link diagram $D$.

Assign to each crossing $c$ of $D$ an unknown $w_c$,
called a \emph{crossing label},
and to each pair $(\vec{e},R)$, consisting of an oriented edge $\vec{e}$ of $D$
and a region $R$ of $D$ that meets $e$, an unknown $u_{\vec{e},R}$,
called an \emph{edge label},
such that $u_{\cev{e},R} = -u_{\vec{e},R}$.
We call such a collection of unknowns a \emph{labeling} of $D$.

These unknowns are meant to record certain local data about
the (expected) complete hyperbolic geometry on the link complement.
Let $T_i ⊆ ∂N(L)$ be peripheral tori, that is,
connected components of the boundary of a tubular neighbourhood of $L$.
The crossing label is closely related to the \emph{crossing arc $γ_c$}
which is an arc that traverses between the over- and underpassing segments
at the crossing $c$, while the edge label is closely related to the
\emph{peripheral edge $\vec{ε}_{\vec{e},R}$},
which is an arc that lives on the peripheral torus $T_i$ surrounding $e$
that runs parallel to $\vec{e}$.
More precisely, if $\vec{e}$ goes from crossing $c$ to $c'$,
and $q,q'$ are the intersections of $γ_c,γ_{c'}$ with $T_i$ respectively,
then $\vec{ε}_{\vec{e},R}$ is taken to be the ``simplest'' arc in $T_i$
going from $q$ to $q'$;
if $e$ is overpassing at $c$ and underpassing at $c'$ (or vice versa),
then there are two non-homotopic ``simplest'' paths,
so we disambiguate them with $R$ (see \cite[Fig. 3]{TTmethod}).

Fix an orientation on $L$.
We say that a labeling is an \emph{algebraic solution}
if it satisfies the following equations:
\begin{enumerate}
\item \emph{edge equation} for $\vec{e}$
	(see \cite[Sec. 3]{TTmethod}): if $R$ and $R'$ are the regions of $D$
	to the left and right of $\vec{e}$,
	and $\vec{e}$ agrees with the orientation of $L$,
	then
	\begin{equation}
	\label{e:edge-eqn-original}
		u_{\vec{e},R'} - u_{\vec{e},R} = \kappa
	\end{equation}
	where
	\begin{itemize}
	\item $\kappa = 1$ if $\vec{e}$ goes from underpassing to overpassing
	(equivalently, if $\vec{ε}_{\vec{e},R'} \vec{ε}_{\vec{e},R}^\inv$
	is homotopic to a meridian wrapping around $L$
	in the direction following the right-hand grip rule),
	\item $\kappa = -1$ if $\vec{e}$ goes from overpassing to underpassing,
	\item $\kappa = 0$ otherwise.
	\end{itemize}
\item \emph{region equation} for $R$
	(see \cite[Eqn. (4)]{TTmethod}):
	if the boundary of $R$ is $\vec{e}₁ \vec{e}₂ \cdots \vec{e}_n$,
	with $\vec{e}_i$ going from crossing $c_{i-1}$ to $c_i$,
	then
	\begin{equation}
	\label{e:region-eqn-original}
	\matu{u_{\vec{e}₁,R}}
	\matw{w_{c₁}}
	\;
	\cdots
	\;
	\matu{u_{\vec{e}_n,R}}
	\matw{w_{c_n}}
	\;
	\sim
	\;
	\matu{0}
	\end{equation}
	where $\sim$ denotes equivalence in $\PSL$.
\end{enumerate}

Note that in \cite{TTmethod}, they first formulate the region equation
in terms of shape parameters \cite[Eqn. (1)]{TTmethod},
then prove that it can be rewritten in the above form \cite[Prop. 4.2]{TTmethod}.
Formulating the region equations directly in terms of crossing and edge labels
as above reduces the requirements on the link diagram;
see \rmkref{r:taut-shape-param} for a discussion.

When $L$ is hyperbolic, and $D$ satisfies some condition,
then it is possible to construct a labeling,
which we call the \emph{geometric labeling},
that is an algebraic solution,
and reflects the hyperbolic geometry of the link complement,
and in fact contains all the information to reconstruct it.
The precise condition is that all crossing arcs must be ``geodesic-like'',
(see \defref{d:geodesic-like}, also \prpref{p:geom-sol-exist});
in \cite{TTmethod}, this condition is met by the stronger condition of ``tautness''
(see \defref{d:taut}).


\ \\
\subsection{Fully Augmented Links}
\label{s:FAL}
\ \\
We briefly review \emph{fully augmented links}, or \emph{FAL}s for short,
and the decomposition/triangulation of their complements,
to which we apply the generalized TT method.
We refer to \cite{purcell} for a comprehensive introduction
to FALs in $S³$, and to \cite{kwon2020} for FALs in $\TT$,
where $\TT$ denotes the thickened torus $T² × (-1,1)$.

\begin{definition}\label{def:falinT2}
A \emph{fully augmented link diagram in $S²$ (resp. $T²$)}
is a link diagram $D$ that is obtained from a
twist-reduced link diagram $D(K)$ in $S²$ (resp. $T²$)
of a link $K$ in $S³$ (resp. $\TT$) as follows:
\begin{enumerate}
\item \emph{augment} every twist region (maximal string of bigons),
	i.e. by adding a single unknotted component, an \emph{augmentation circle},
	around every twist region,
\item remove all full-twists (pairs of crossings in a twist region),
\item lay the augmentations ``as flat as possible'' so that they each make
	four crossings.
\end{enumerate}
See Figure \ref{fig:falS3}.

A \emph{fully augmented link (FAL)} in $S³$ (resp. $\TT$) is a link
that has a fully augmented link diagram in $S²$ (resp. $T²$).
\end{definition}

\begin{definition}
\label{d:bow-tie-graph}
The \emph{bow-tie graph} $B_L$ of a FAL $L$ in $S³$ or $\TT$
is obtained by removing all half-twists,
replacing each augmentation circle by a pair of triangles,
i.e. a \emph{bow-tie},
and shrink each segment of the link to a vertex.
See Figures \ref{fig:falDecomp}, \ref{f:DL-BL}.
\end{definition}

\begin{figure}
\centering
\includegraphics[height=4cm]{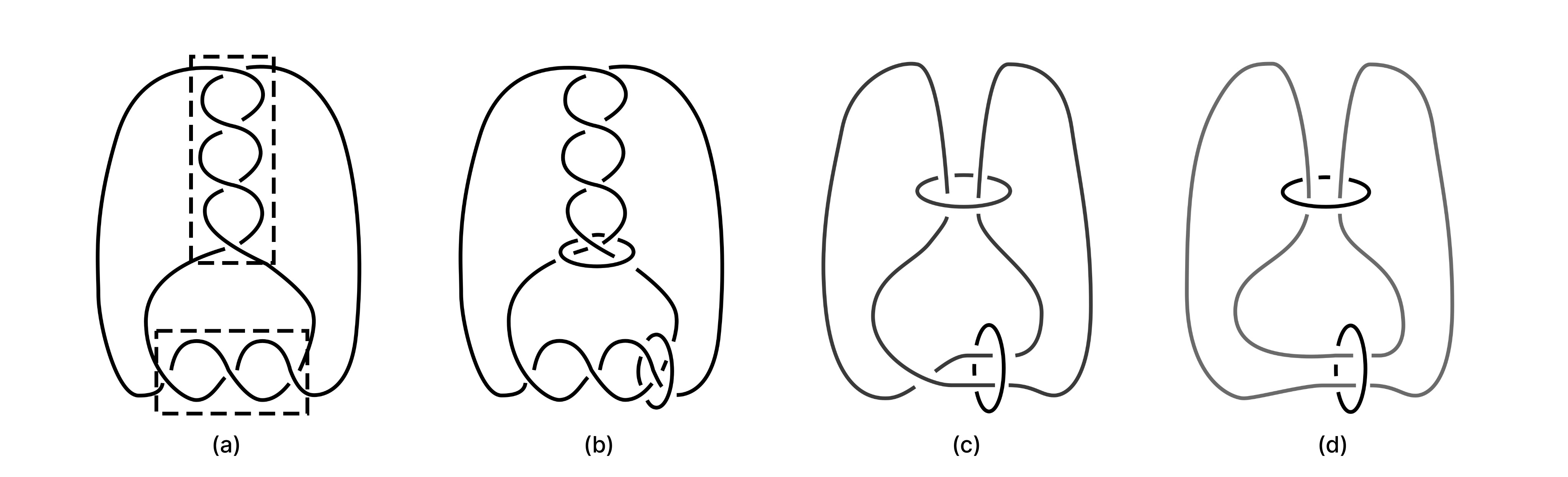}
\caption{(a) Link diagram of $K$ (b) crossing circles added to each twist region 
(c) the third picture is a fully augmented link diagram with all full-twists 
removed (d) fully augmented link diagram with no half-twists.
Note that removing a full-twist from a FAL does not change the
homeomorphism class of its complement,
but removing a half-twist might.}
\label{fig:falS3}
\end{figure}

\begin{figure}
\centering
\begin{tabular}{cccc}
\includegraphics [height=3cm]{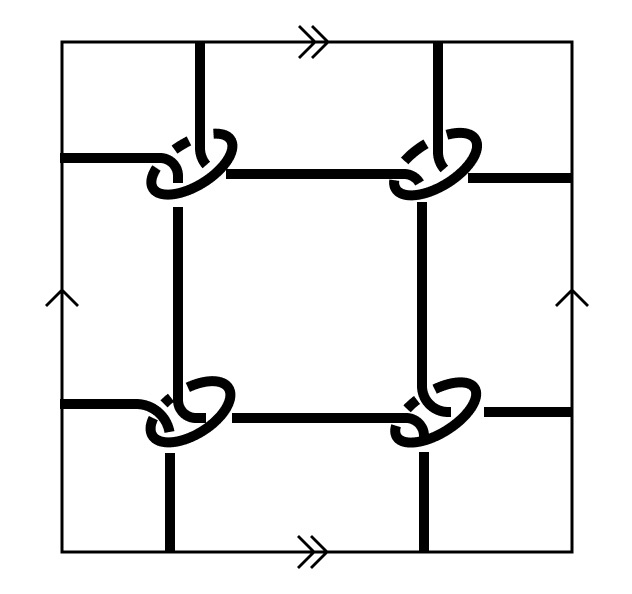}&
\includegraphics [height=3cm]{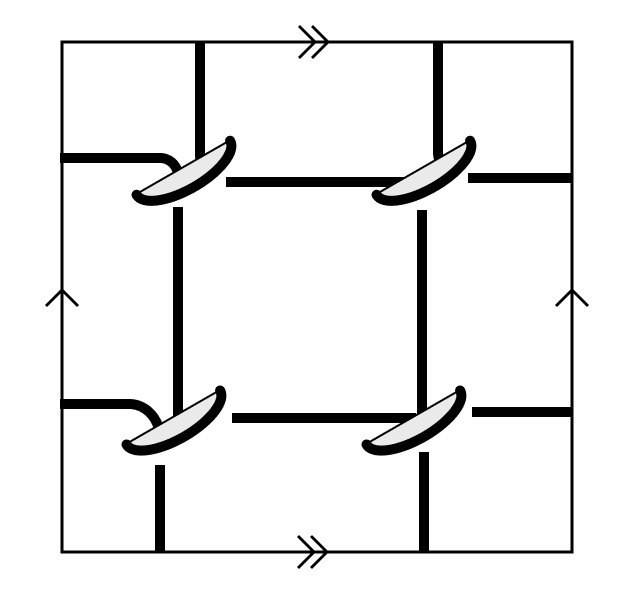}&
\includegraphics [height=3cm]{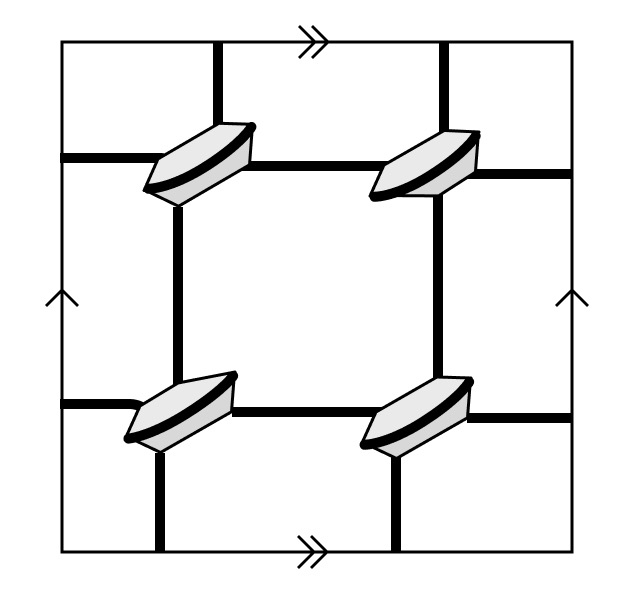}&
\includegraphics [height=3cm]{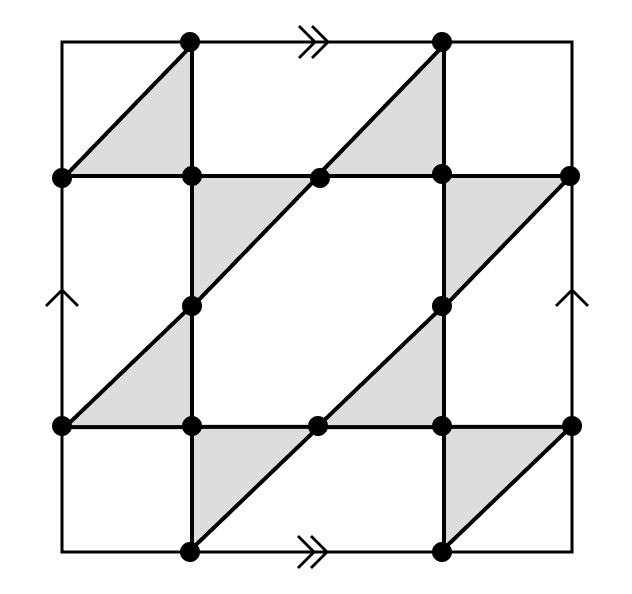}\\
(a)&(b)&(c)&(d)
\end{tabular}
\caption{
A brief review of the cut-slice-flatten method of
\cite{lackenby}, \cite{kwon2020}
for a FAL with no half-twists:
(a) A fundamental domain for a fully augmented square weave, $L$.
(b) Disks cut in half at each augmentation circle.
(c) Sliced and flattened half-disks at each augmentation circle
(d) Collapsing the bold strands in (c) to ideal points
gives the bow-tie graph $B_{L}$.
The half-disks become bow-ties, i.e. a shaded pair of triangular regions;
the white regions are the $F_R$'s corresponding to regions of
the link diagram.
}
\label{fig:falDecomp}
\end{figure}

\begin{figure}[ht]
\includegraphics[height=5cm]{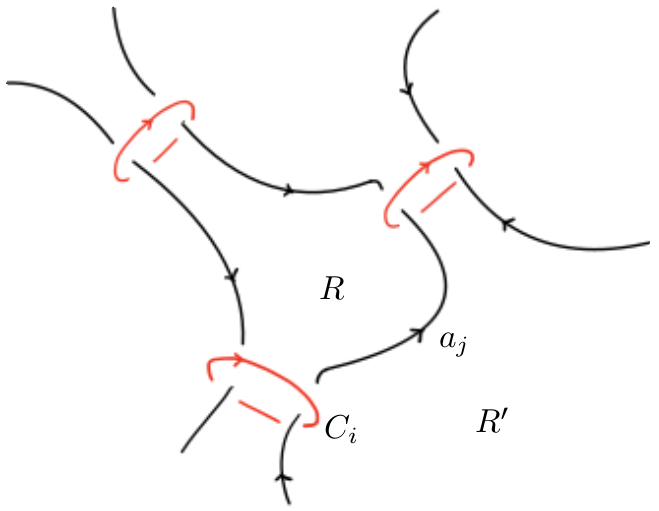}
\includegraphics[height=5cm]{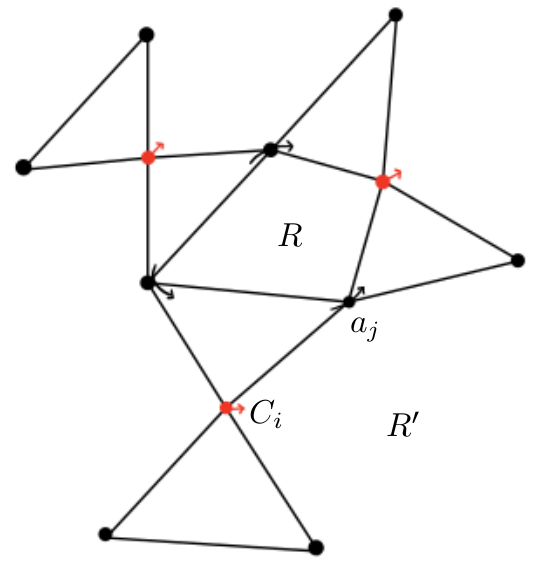}
\caption{Another example of a FAL diagram $D_L$
and corresponding bow-tie graph $B_L$;
the arrows on the vertices in $B_L$ indicate the orientations of
their corresponding components/segments of $L$;
for $C_i$, the arrow follows the direction of the top half of $C_i$.
}
\label{f:DL-BL}
\end{figure}

The cut-slice-flatten method of \cite{lackenby} for FALs in $S³$
produces a decomposition of the link complement into two ideal polyhedra:

\begin{definition}{\cite{lackenby}}
\label{d:decomp-S3}
Let $L$ be a FAL in $S³$; we first consider $L$ with no half-twists.
For an augmentation circle $C_i$,
denote by $F_{C_i}$ the \emph{spanning (twice-punctured) disk} bounded by $C_i$;
the projection plane cuts $F_{C_i}$ along three arcs, denoted $γ_i¹,γ_i⁰,γ_i²$,
into two half-disks which we denote by $F_{C_i}^+,F_{C_i}^-$
and refer to as the \emph{spanning} or \emph{bow-tie faces}.
See \figref{f:notation-spanning-face} for notation convention.

The arcs $γ_i^{•}$ and the link $L$ separate the projection plane into 2-cells,
one for each region $R$ of the link diagram of $L$
(here we only consider regions not created/bounded by augmentation circles;
these are in bijection with the non-bigon regions of $D(K)$).
We refer to these 2-cells as \emph{regional faces}, and denote them by $F_R$.

The ideal polyhedral decomposition of $S³ - L$ of \cite{lackenby}
consists of the arcs $γ_i^{•}$ as 1-cells,
the faces $F_{C_i}^+,F_{C_i}^-,F_R$ as 2-cells,
and two 3-cells which are the connected components
of $S³ - L$ minus the 1- and 2-cells.

For $L$ with half-twists, suppose $L'$ is the FAL that is obtained from $L$
by removing half-twists.
Then the polyhedral decomposition for $L$ is constructed by applying
the (non-continuous) transformation shown in \figref{f:half-twist-rotate}
to the polyhedral decomposition for $L'$.
(Further visualization aids are given in Figures \ref{f:FAL-S3-polyhedral}
and \ref{f:regional-face-half-twist}.)
\end{definition}

Kwon \cite{kwon2020} describes the analog of the cut-slice-flatten method
for FALs in $\TT$. Before we give the definition, we need the following:

\begin{definition}{\cite[Def. 2.5]{kwon2020}}
\label{d:torihedron}
A \emph{torihedron $\cT$} is a cone on the torus, i.e. $T² × [0,1]/(T² × \{1\})$,
together with a \emph{cellular} graph $G$ on $T² × \{0\}$,
i.e. every region of $G$ is a disk.
We refer to $G$ as the \emph{graph of $\cT$},
and refer to the edges (resp. faces) of $G$
as the edges (resp. faces) of the torihedron. 

An \emph{ideal torihedron} is a torihedron with the vertices of $G$ and 
the vertex $T² × \{1\}$ removed;
the removed vertices are referred to as its \emph{ideal vertices}.
\end{definition}

\begin{definition}
\label{d:torihedra-conical}
Let $\cT$ be a torihedron, with graph $G$.
The \emph{conical polyhedral decomposition} of $\cT$
is the polyhedral decomposition obtained by coning $G$ over the cone point,
i.e., it has a 3-cell for each face of $G$,
a 2-cell for each edge of $G$ (in addition to the faces of $G$),
and a 1-cell for each vertex of $G$ (in addition to the edges of $G$).
\end{definition}

\begin{definition}
\label{d:torihedra-conical-triangulation}
Let $\cT$ be a torihedron, with graph $G$.
Let $G'$ be a triangulation of $T² × \{0\}$ obtained from $G$
by adding only edges,
and let $\cT'$ be the torihedron with graph $G'$.
The \emph{conical triangulation of $\cT$ (based on $G'$)}
is the conical polyhedral decomposition of $\cT'$.

We may simply refer to \emph{a} conical triangulation of $\cT$
if we do not need to be specific about the choice of $G'$.
\end{definition}

\begin{definition}{\cite[Prop. 2.7]{kwon2020}}
\label{d:decomp-T2}
Let $L$ be a FAL in $\TT$; we first consider $L$ with no half-twists.
The definitions of arcs $γ_i^{•}$ and faces $F_{C_i}^{•}$, $F_R$
from \defref{d:decomp-S3} apply to $\TT - L$.
As before, we refer to $F_R$ as a \emph{regional face},
and $F_{C_i}^+,F_{C_i}^-$ as \emph{spanning} or \emph{bow-tie faces}
(see \defref{d:bow-tie-graph}).

The ideal torihedral decomposition of $\TT - L$ of \cite{kwon2020}
consists of the arcs $γ_i^{•}$ as 1-cells,
the faces $F_{C_i}^+,F_{C_i}^-,F_R$ as 2-cells,
and two ideal torihedra which are the connected components
of $\TT - L$ minus the 1- and 2-cells.

In each torihedron, the \emph{conical polyhedral decomposition}
is obtained by adding the following ``vertical'' 1- and 2-cells.
We add one 1-cell $γ_{C_i}$ for each augmentation circle $C_i$
and one 1-cell $γ_{a_j}$ for each segment $a_j$ of $L$ demarcated by spanning disks,
which connect $C_i$ and $a_j$ to the ideal vertex of the torihedron, respectively;
we refer to these 1-cells as \emph{vertical crossing arcs}.
For each pair $(C_i,a_j)$ of an augmentation circle $C_i$
and segment $a_j$ that are adjacent,
we add a 3-sided 2-cell $F_{C_i,a_j}$ characterized by having
boundary edges $γ_{C_i}, γ_{a_j}$, and $γ_i^k$ (where $k = 1$ or 2).
For segments $a_j,a_{j'}$ ending at a common $C_i$
(both on the same side of $C_i$),
we add a 3-sided 2-cell $F_{a_j,a_{j'};C_i}$ characterized by having
boundary edges $γ_{a_j}, γ_{a_{j'}}$, and $γ_i⁰$.
We refer to these 2-cells as \emph{vertical faces}.
See \figref{f:FAL-T2-polyhedral}.
\footnote{The notation $γ_{C_i},γ_{a_j},F_{C_i,a_j}$ is ambiguous
as it could refer to vertical crossing arcs of the top or bottom torihedron;
this should not be a problem as
we usually only work with the top torihedron,
with the same arguments applying to the bottom torihedron by symmetry.

In some cases, a segment $a_j$ may loop around and meet $C_i$ again,
so that $F_{C_i,a_j}$ is ambiguous even when restricted to one torihedron.
This ambiguity is overcome in the notation for the TT method for FALs
in Sections \ref{s:TT-FAL-S3}, \ref{s:TT-FAL-T2}.
}

For $L$ with half-twists, suppose $L'$ is the FAL that is obtained from $L$
by removing half-twists.
Then the torihedral decomposition for $L$ is constructed by applying
the (non-continuous) transformation shown in \figref{f:half-twist-rotate}
to the torihedral decomposition for $L'$.
(Further visualization aids are given in Figures \ref{f:FAL-T2-polyhedral}
and \ref{f:regional-face-half-twist}.)
\end{definition}

This torihedral decomposition of the complement of a FAL in $\TT$
is a modification of the torihedral decomposition of arbitrary links in $\TT$
by \cite{CKP}.
The graph on the boundary of the polyhedra/torihedra
is the bow-tie graph $B_L$ of $L$.

\begin{figure}[ht]
\centering
\begin{tabular}{ccc}
\includegraphics[width=6cm]{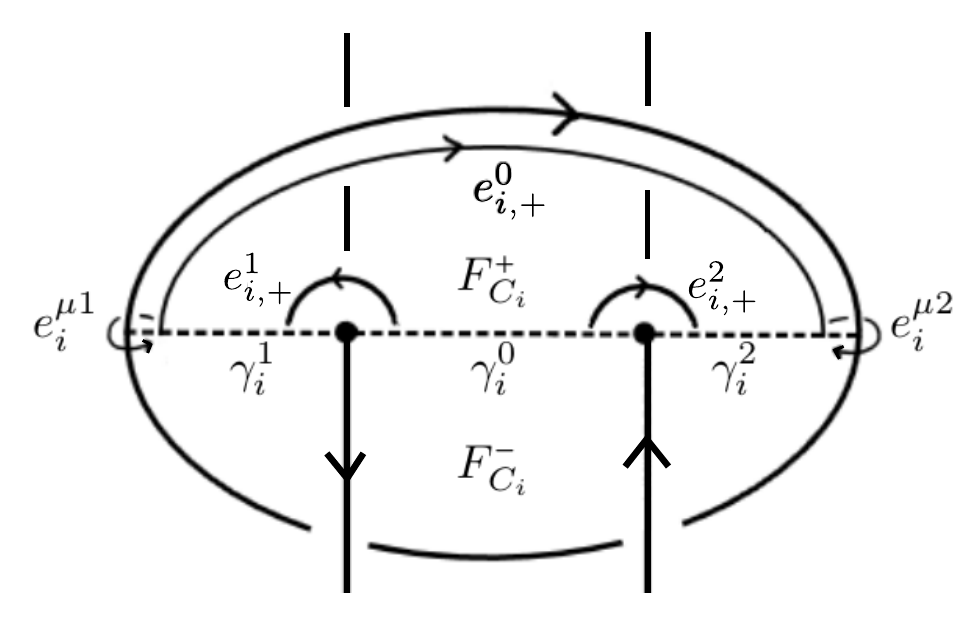}&
\includegraphics[width=5.5cm]{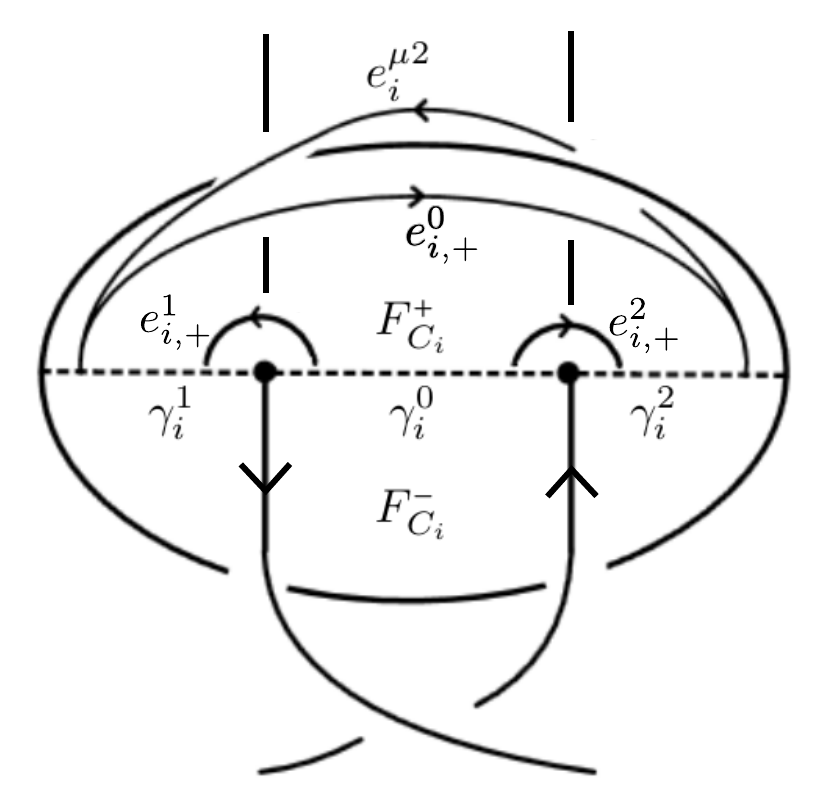}&
\includegraphics[width=3.5cm]{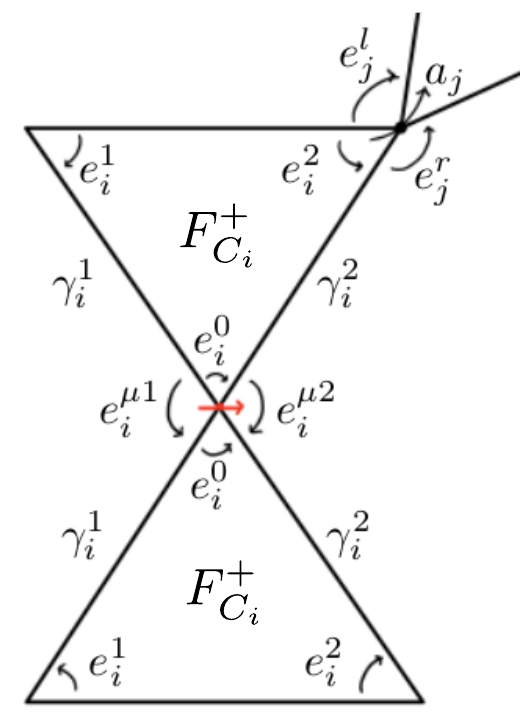}
\\
(a)&(b)&(c)
\end{tabular}
\caption{
Notation for crossing arcs and peripheral edges
near spanning face without (a) and with (b) half-twist.
By default, we orient $e_i⁰$ is oriented following $C_i$,
$e_{i,±}¹,e_{i,±}²$ are oriented by right-hand thumb rule around $L$,
and $e_i^{μ1},e_i^{μ2}$ are oriented by the right-hand thumb rule around $C_i$.
The order between $γ_i¹,γ_i²$ is chosen so that
$e_{i,+}⁰$ is oriented from $γ_i¹$ to $γ_i²$.
For the spanning disk region equations \eqnref{e:spanning-disk-eqn},
$χ_i¹ = +1$ and $χ_i² = -1$.
(c) Crossing arcs and peripheral edges depicted in $B_L$,
for $L$ with no half-twist.
}
\label{f:notation-spanning-face}
\end{figure}

\begin{figure}
\includegraphics[width=12cm]{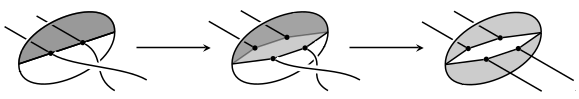}
\caption{
Diagrams modified from \cite[Fig. 7]{kwontham} depicting a
(non-continuous) operation on a FAL complement:
at an augmentation circle with a half-twist,
slice along the spanning twice-punctured disk,
rotate one side by $180^{∘}$ to undo the half-twist,
then glue the two sides back.
}
\label{f:half-twist-rotate}
\end{figure}

\begin{figure}
\centering
\includegraphics[height=3.5cm]{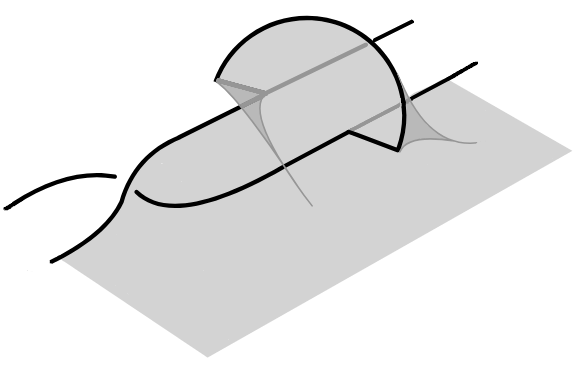}
\caption{
Regional face near a half-twist,
which is attached by following the edges in \figref{f:FAL-S3-polyhedral}(b),
starting at the bottommost gray edge and ending at the top-right.
}
\label{f:regional-face-half-twist}
\end{figure}

\begin{figure}
\centering
\begin{tabular}{cc}
\includegraphics[height=5cm]{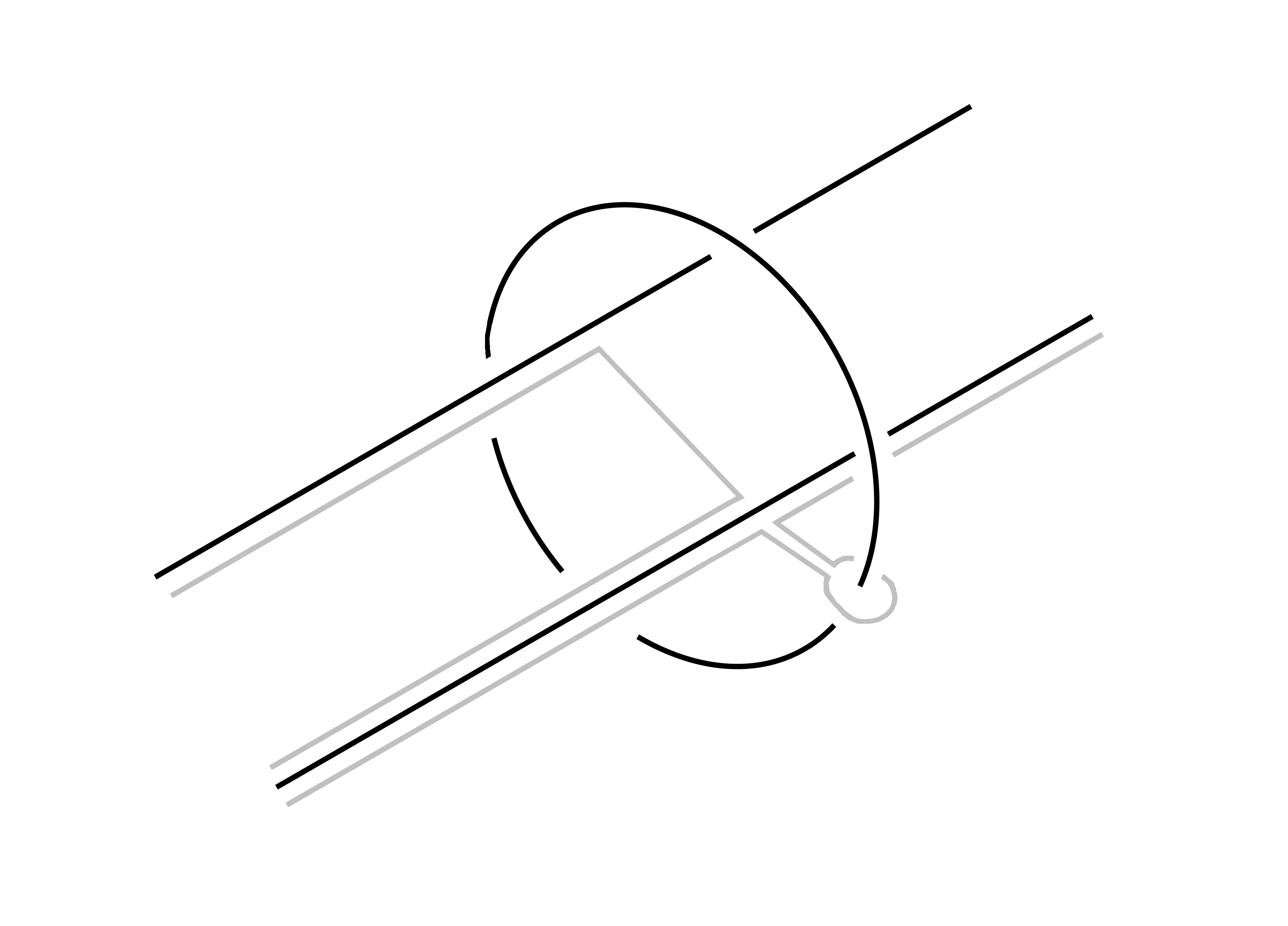}&
\includegraphics[height=5cm]{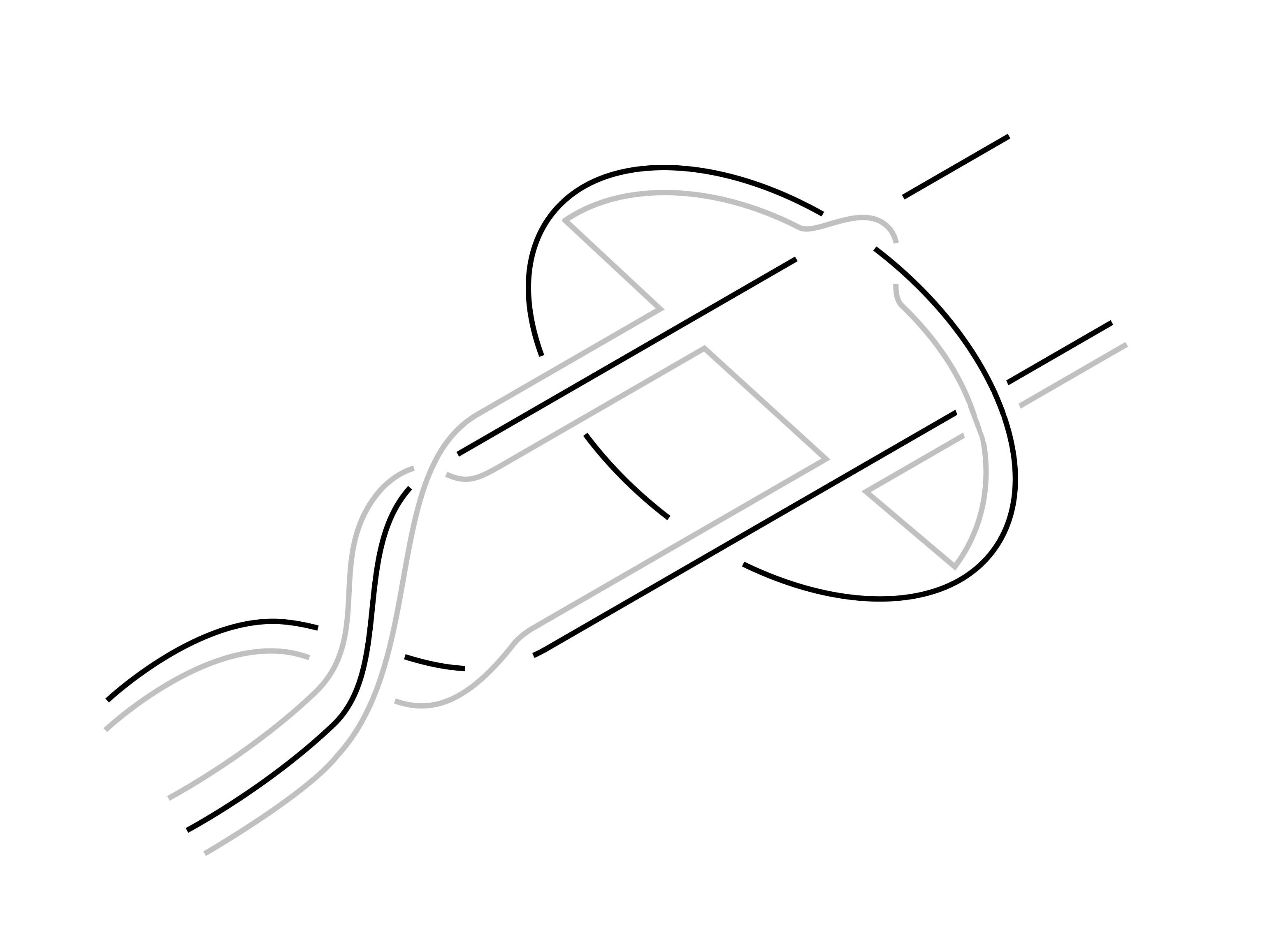}
\\
(a)&(b)
\end{tabular}
\caption{
Boundary edges of (truncated) polyhedral decomposition of $S³ - L$
near an augmentation circle $C_i$ without (a) and with (b) half-twist.
}
\label{f:FAL-S3-polyhedral}
\end{figure}

\begin{figure}
\centering
\begin{tabular}{cc}
\includegraphics[height=5cm]{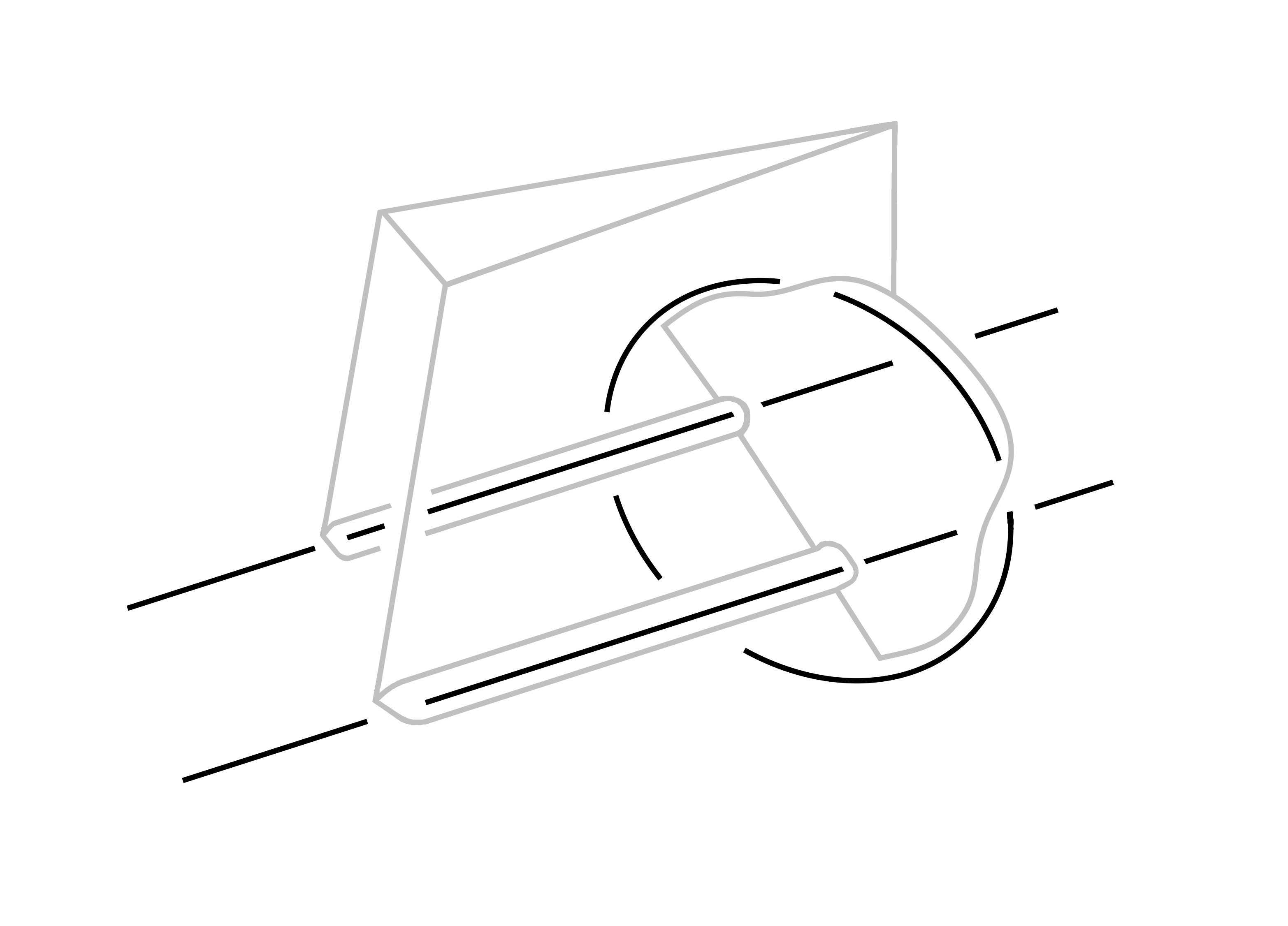}&
\includegraphics[height=5cm]{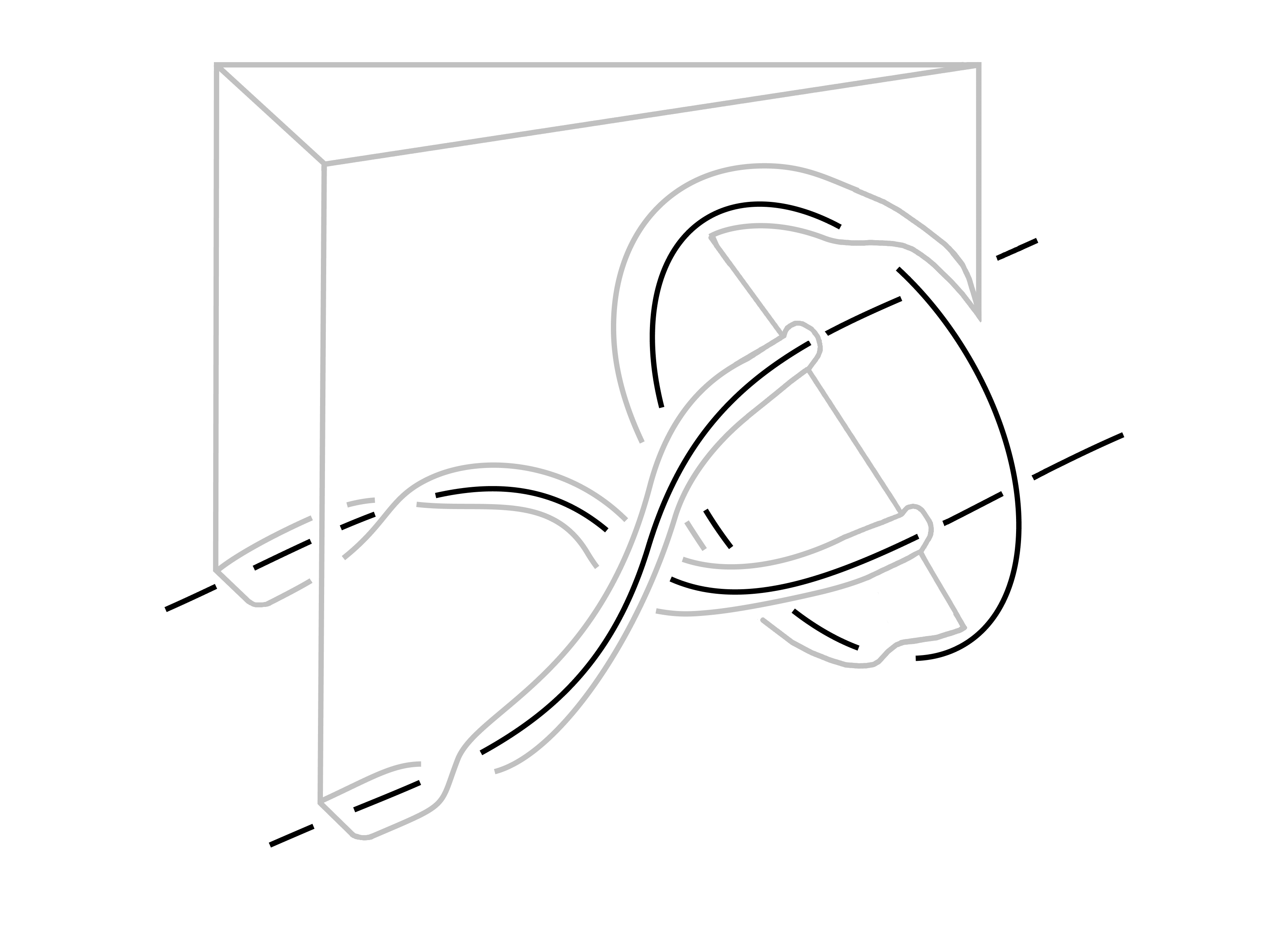}
\\
(a)&(b)
\end{tabular}
\caption{
Boundary edges of (truncated) canonical polyhedral decomposition of $\TT - L$,
which forms a (truncated) tetrahedron.
The edges at the top are peripheral edges around $\north$.
}
\label{f:FAL-T2-polyhedral}
\end{figure}

\begin{remark}
\label{r:FAL-triangulation-geometric}
For the polyhedral decomposition of FALs in $S³$,
one can also refine it to a triangulation
by triangulating each region of the bow-tie graph,
and then choosing a vertex in the bow-tie graph
and coning each 3-cell over this vertex.
This triangulation and the conical triangulations of
\defref{d:torihedra-conical-triangulation} are geometric.
\end{remark}


\begin{remark}
\label{r:FAL-half-twist-symmetry}
When a FAL $L$ has no half-twists,
its complement has an obvious symmetry that swaps the top and bottom
polyhedra/torihedra, which, in particular,
preserves the 1-cells $γ_i^{•}$ and 2-cells $F_R$ in the projection plane
and swaps the two 3-cells associated with each region.

Suppose $L$ has half-twists. Let $L''$ be the link
obtained from $L$ by applying this top-bottom symmetry
(i.e., flip the crossing at the half-twists).
Then the polyhedral/torihedral decomposition for $L''$
can be obtained from that of $L$ in two ways:
\begin{itemize}
\item by applying the top-bottom symmetry;
\item by performing a full-twist
	at each augmentation circle that harbors a half-twist.
\end{itemize}
Thus, the composition of the top-bottom symmetry and the full-twists
is a symmetry on the complement of $L$
that preserves the 1-cells $γ_i^{•}$ and 2-cells $F_R$,
and swaps the two 3-cells/torihedra;
furthermore, in the conical polyhedral decomposition of $\TT - L$,
this symmetry swaps the two 3-cells above and below each region $R$.
\end{remark}

\ \\
\subsection{Circle Packings and FALs}
\label{s:FAL-circ}
\ \\
We briefly recall circle packings and their relation to FALs,
first written in the appendix of \cite{lackenby} by Agol and D. Thurston.
In short, the complete hyperbolic geometry on the complement of a FAL
is determined by a circle packing, with one circle corresponding
to each region of the link.
We also refer to \cite{purcell} and \cite{kwon2020} for more details.

\begin{definition}
\label{d:circle-packing}
Let $Σ$ be a surface which is $\CC$ or $S²=\CC \cup \{\infty\}$,
and let $Γ$ be a simple graph (typically a triangulation) embedded in $Σ$.
A \emph{circle packing realizing $Γ$ in $Σ$} is a collection of circles in $Σ$,
$Ξ = \{X_i\}_{i ∈ V}$, indexed by the vertices $V$ of $Γ$,
such that $X_i,X_j$ are tangent if $i,j$ are adjacent vertices in $Γ$.
\end{definition}

(Note that our definition of a circle packing
slightly differs from the standard one,
in that the interiors of circles are considered as additional data;
see e.g. \cite{BandS} for more on circle packings/packings.)

\begin{definition}
\label{d:circle-packing-T2}
Let $Σ = T² = \CC/Λ$ be the torus given as a quotient of $\CC$
by a group $Λ \simeq \ZZ \oplus \ZZ$ of translations,
and let $Γ$ be a graph embedded in $Σ$.
A \emph{circle packing realizing $Γ$ in $Σ$}
is a circle packing realizing $\wdtld{Γ}$ in $\wdtld{Σ}\simeq \CC$
that is invariant under $Λ$,
where $\wdtld{Γ}$ is the preimage of $Γ$ in $\CC$.
\end{definition}

We allow ``accidental'' intersections or tangencies of circles,
that is, $i,j$ not adjacent in $Γ$
does not imply that $X_i,X_j$ are not tangent.
We also allow the circles to have interiors that intersect
(see \defref{d:univalence}).

Note that in the $Σ = T²$ case,
the existence of a circle packing on $Γ$ may impose a constraint on $Λ$.

In the following definitions,
$Σ$ may be one of the surfaces above,
and $Γ \subset Σ$ is an embedded simple graph as above.

\begin{definition}
\label{d:locally-order}
A circle packing is \emph{locally order-preserving}
if the cyclic order on points of tangencies on $X_v$
agrees with the clockwise or counterclockwise cyclic order
on the vertices adjacent to $v$ in $Γ$
(imposed by its embedding in $Σ$).
\end{definition}


\begin{definition}
\label{d:circle-packing-interior}
An \emph{interior filling} of a circle packing $Ξ$ is a choice,
for each circle $X_i ∈ Ξ$, of a connected component of $Σ \backslash X_i$;
the chosen component is called the \emph{interior (region)}.
\end{definition}

Equivalently, it is a choice of orientation of each $X_i$
(the corresponding interior filling begins defined by
the region to the left as the interior).
Of course, for $Σ = \CC, T²$, a circle packing has a unique interior filling.

\begin{definition}
\label{d:univalence} 
An interior filling is \emph{locally univalent}
if the interiors of $X_i,X_j$ are disjoint for
adjacent vertices $i,j$ of $Γ$;
if disjointness holds for \emph{all} pairs of vertices,
then we say the interior filling is \emph{(globally) univalent}.
(See \cite{Stephenson}.)
\end{definition}

Note that if a circle packing admits a (locally) univalent interior filling,
then such a locally univalent interior filling is unique,
in which case we say that the circle packing is (locally) univalent.

Univalence implies the resulting graph of tangencies
(vertices are the circles, edges are tangencies) is isotopic to $Γ$
(or the image of $Γ$ under a diffeomorphism of $Σ$).
In fact, in our use case, it is a local condition:

\begin{lemma}
\label{l:univalence-is-local}
Let $Ξ$ be a circle packing of a triangulation $Γ$ of $Σ = \CC$ or $S²$.
Then univalence is a local condition in the following sense:
if the circle packing is locally order-preserving and locally univalent,
then it is univalent.
\end{lemma}

Note that there is no finiteness condition on $Γ$,
so this applies to a biperiodic graph covering a graph in $T²$.

\begin{proof}
First, observe that if a vertex of $Γ$ has degree 2,
then since $Γ$ is a triangulation,
the only possibility is that $Σ = S²$,
and $Γ$ consists of three vertices, all connected to each other,
so there is nothing to prove.
Hence let us assume that the degree of every vertex of $Γ$
is at least 3.

Let $Γ_D$ be the dual graph to $Γ$,
which has a natural ``transverse to $Γ$'' embedding in $Σ$
(uniquely defined up to isotopy).
Let $R_v$ be the region of $Γ_D$ corresponding to vertex $v$ of $Γ$.
We show that there is a map from $S²$ to itself
that sends $X_v$ and its interior into $R_v$ for each vertex $v$,
which would immediately imply univalence.

Consider a triangular region of $Γ$ with vertices $v₁,v₂,v₃$.
The interiors of $X_{v₁},X_{v₂},X_{v₃}$ are pairwise disjoint.
Let the points of tangencies among the $X$'s be $q₁,q₂,q₃$,
with $q_{i+2}$ between $X_{v_i}$ and $X_{v_{i+1}}$.

\begin{figure}[ht]
\includegraphics[width=4cm]{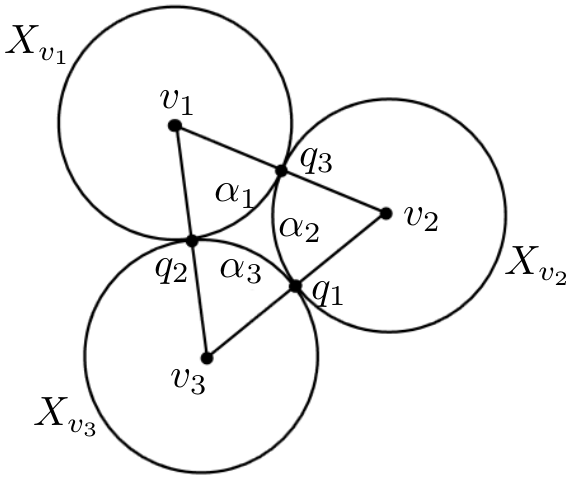}
\caption{``Triangular'' region between $X_{v₁},X_{v₂},X_{v₃}$}
\label{f:univalent-q-alpha}
\end{figure}

There are two arcs of $X_{v₁}$ between $q₂$ and $q₃$;
by locally order-preserving-ness, and $\deg v₁ \geq 3$,
exactly one of these arcs has no other points of tangencies.
Let $α₁$ be this arc.
Similarly define $α₂, α₃$.
These three arcs enclose a region
that is exactly one of the connected components of
the complement of the union of $X_{v₁},X_{v₂},X_{v₃}$ and their interiors.
Pick a point $q_{v₁ v₂ v₃}$ in this region,
and draw curves from each $q_i$ to $q_{v₁ v₂ v₃}$
within this enclosed region.

It is easy to see that these points $q_{v₁ v₂ v₃}$
and the curves we draw to the points of tangencies
make an embedded graph that is isomorphic to $Γ_D$;
the self-diffeomorphism of $Σ$ sending this graph to $Γ_D$
would send $X_v$ and its interior into $R_v$, hence we are done.
\end{proof}

\begin{definition}
\label{d:dual-ish-bow-tie-graph}
Let $L$ be a FAL in $S³$ (resp. $\TT$)
with bow-tie graph $B_L$ in $S²$ (resp. $T²$).
The \emph{region graph}, denoted by $Γ_L$, is the embedded graph
consisting of a vertex in each non-bow-tie region $R$ of $B_L$,
and an edge for each vertex of $B_L$
(which connects the two non-bow-tie regions adjacent to it).
The faces of $Γ_L$ correspond bijectively with bow-tie faces,
hence $Γ_L$ is a triangulation of $S²$ (resp. $T²$).
\end{definition}

A univalent circle packing on $Γ_L$ determines geometries on the
polyhedra in the (conical) polyhedral decomposition of the link complement,
which follows from the observation that the graph on $Σ$ ($= S²$ or $T²$)
with vertices given by the points of tangencies of the circle packing
and edges connecting neighboring points of tangencies on a circle
is precisely the bow-tie graph $B_L$.
In fact, this is the complete hyperbolic geometry on the link complement:

\begin{proposition}
\label{p:geometric-univalence} 
Let $L$ be a FAL in $S³$ (resp. $\TT$)
and let $Γ_L$ be its region graph in $S²$ (resp. $T²$).
A univalent circle packing on $Γ_L$ is unique up to conformal maps,
and determines the geometry of the polyhedra in the
polyhedral (resp. conical polyhedral) decomposition of the link complement
that makes it complete hyperbolic;
that is, the complete hyperbolic metric on the link complement
determines a realization of each polyhedron
as an ideal hyperbolic polyhedron,
whose ideal vertices are the points of tangencies of the circle packing
(resp. together with $∞$).
\end{proposition}
\begin{proof}
For uniqueness, see \cite{koebe} and \cite{BandS}.
For circle packing to geometry, see \cite{lackenby} and \cite{kwon2020}.
\end{proof}




\section{Generalization of the TT method}
\label{s:TT-3mfld}

We consider a generalization of the TT method to 3-manifolds with
toric ends and an ideal polyhedral decomposition
(or equivalently, by truncation, 3-manifolds with toric boundaries
with a polyhedral decomposition having only vertices on the boundary).
The original TT method, as we will see in \secref{s:original-TT},
implicitly uses the Menasco decomposition on the complement of a link in $S³$.
This allows us to systematically consider variations of the TT method
for other types of links, which we explore in \secref{s:application-links}.
The original TT method also uses the ``obvious'' choice of meridians
around each link component for normalization, but in general,
we have to make a choice $\{μ_i\}$ of meridians (i.e. simple closed curves)
for each toric end; in particular, for links in the thickened torus,
we need to make additional choices of meridians for the top and bottom ends.

\subsection{General case} 
\label{s:TT-general} 
\ \\ \indent
3-manifold $M$ refers to an oriented smooth 3-manifold with toric ends,
whose truncation $\ov{M} ⊆ M$ is compact.
We may also embed $M$ into $\ov{M}$ so that its closure is $\ov{M}$,
and view $\ov{M}$ as a completion of $M$ by adding boundary tori.

Let $τ$ be an ideal polyhedral decomposition of $M$;
its truncation, $\ov{τ}$, is a polyhedral decomposition of $\ov{M}$,
which is related to $τ$ as follows.
The 3-cells of $τ$ and $\ov{τ}$ are in bijection with each other.
The \emph{internal} 1- and 2-cells of $\ov{τ}$
consist of the intersections of the 1- and 2-cells of $τ$ with $\ov{M}$,
while the \emph{boundary} or \emph{peripheral} 2-, 1- and 0-cells of $\ov{τ}$
consist of the intersections of 3-, 2- and 1-cells of $τ$ with
the boundary tori of $\ov{M}$.
Note that all 0-cells (i.e. vertices) of $\ov{τ}$ are peripheral,
so we usually drop the adjective.
For a 1-cell $γ$ of $τ$, we refer to the endpoints of its corresponding
internal 1-cell $\ov{γ}$ of $\ov{τ}$ as its \emph{peripheral endpoints}.
Note that the boundary of every interal 2-cell of $\ov{τ}$
consists of 1-cells that alternate between internal and peripheral.

\begin{definition}
\label{d:labeling}
A \emph{labeling} $Ω = (w(-), u(-))$ of $\ov{τ}$ is an assignment of a
\emph{crossing label} $w(γ) ∈ \CC$ to each internal 1-cell $γ$
and an \emph{edge label} $u(\vec{e}) ∈ \CC$
to each oriented boundary 1-cell $\vec{e}$.
\end{definition}

We also say that $Ω$ is a labeling of an ideal polyhedral decomposition $τ$
when $Ω$ is a labeling of the corresponding truncation $\ov{τ}$.
We sometimes write $w_{γ}$ and $u_{\vec{e}}$
for $w(γ)$ and $u(\vec{e})$ respectively.

\begin{definition}
\label{d:alg-soln}
Let $M$ be a 3-manifold, with choices of meridians $\{μ_i\}$,
and let $τ$ be an ideal polyhedral decomposition of $M$.
A labeling $Ω = (w(-), u(-))$ of $τ$ is an \emph{algebraic solution}
if it satisfies:
\begin{itemize}
\item Normalization: For each boundary component $T_i$,
  if $\vec{e}₁,\ldots,\vec{e}_k$ is a sequence of 1-cells in $T_i$
	such that their concatenation is homotopic to $μ_i$,
	then $Σ u_{e_j} = 1$;
\item Region equations:
  For each oriented internal 2-cell $P$ of $\ov{τ}$, with boundary
  $∂P = \vec{γ}₁ \vec{e}₁ \cdots \vec{γ}_k \vec{e}_k$
	which alternates between internal 1-cells $\vec{γ}_i$ and
	peripheral 1-cells $\vec{e}_i$,
	\begin{equation}
		\label{e:region-eqn-matrix}
		\matu{-u_{\vec{e}_k}}
		\matw{w_{γ_k}}
		⋅⋅⋅
    \matu{-u_{\vec{e}₁}}
    \matw{w_{γ₁}}
		\sim
		\matu{0}
	\end{equation}
	where $\sim$ denotes equality in $\PSL$.
\end{itemize}
\end{definition}

Note that the 2-cell with the opposite orientation gives the same equation: from
$∂\cev{P} = \cev{γ}_k \cev{e}_{k-1} \cdots \cev{γ}₂ \cev{e}₁ \cev{γ}₁ \cev{e}_k$,
we would get the region equation
(using $u_{\cev{e}_i} = -u_{\vec{e}_i}$)
	\begin{equation}
    \matu{u_{\vec{e}_k}}
    \matw{w_{γ₁}}
    \matu{u_{\vec{e}₁}}
    \matw{w_{γ₂}}
		⋅⋅⋅
		\matu{u_{\vec{e}_{k-1}}}
		\matw{w_{γ_k}}
		\sim
		\matu{0}
	\end{equation}
which is actually the same as \eqnref{e:region-eqn-matrix}
(move the first matrix to the end, and then take the inverse of both sides),
so we may ignore the orientation of $P$ when we discuss its region equation.

\begin{remark}
\label{r:no-hyperbolic}
We note that no hyperbolic geometry is involved so far, and indeed,
one can speak of algebraic solutions on a non-hyperbolic 3-manifold.
However, algebraic solutions are closely related to $\PSL$-representations
of the fundamental group of the 3-manifold, which we discuss in
\secref{s:alg-to-geom-reconstruction}; in this view,
algebraic solutions have geometric meaning as they describe
some (possibly degenerate) local hyperbolic structure.
\end{remark}

\begin{remark}
\label{r:meridian-scaling}
It is useful to consider the TT method with different normalizations.
Fix $λ_i ∈ \CC^\times$ for each boundary $T_i$.
Suppose we change the normalization equation to $Σ u_{e_j} = λ_i$,
and proceed with the TT method as usual everywhere else.
If $Ω' = (w'(-),u'(-))$ is an algebraic solution to this modified TT method,
then we can construct an algebraic solution $Ω = (w(-),u(-))$ of the usual TT method,
where $u(\vec{e}) = u'(\vec{e})/λ_i$ for a peripheral edge on $T_i$,
and $w(γ) = w'(γ)/λ_iλ_j$ for an internal 1-cell $γ$ whose endpoints are on
$T_i$ and $T_j$ (possibly $i = j$).
Clearly this construction is invertible.

See also \rmkref{r:meridian-scaling-geom-label}
concerning rescaling and geometric meaning.
\end{remark}

\begin{remark}
\label{r:meridian-choice}
Changing the choice of meridians typically has the same effect as
using a different normalization as in \rmkref{r:meridian-scaling}, but not always.
Suppose we want to switch to a different choice meridians $\{μ_i'\}$.
For an algebraic solution $Ω$ for $M$ with the original choices $\{μ_i\}$,
if $u(μ_i') ≠ 0$ for all $i$,
then $Ω$ can be thought of as a solution for $M$ with $\{μ_i'\}$,
with normalizations $λ_i = u(μ_i')$,
or as in \rmkref{r:meridian-scaling}, we have an algebraic solution
$Ω' = (\{w'(γ) = w(γ)/λ_iλ_j, u'(\vec{e}) = u(\vec{e})/λ_i)$.
Thus we have a (partial) correspondence between algebraic solutions
for $M$ with $\{μ_i\}$ and algebraic solutions for $M$ with $\{μ_i'\}$.
Note that for the geometric labeling (Section \ref{s:geometric-labeling}),
which is ultimately the labeling we care about, we have $u(μ) ≠ 0$ for any meridian.
Note also that the $λ_i$ above depends on the algebraic solution $Ω$.
\end{remark}

The region equation is particularly simple for $k = 3$,
though its meaning will be clearer in terms of shape parameters
(see \rmkref{r:region-eqn-shape}):
\begin{lemma}
\label{l:region-eqn-k3}
For an internal 2-cell $P$ with $k = 3$ and
$∂P = \vec{γ}₁ \vec{e}₁ \vec{γ}₂ \vec{e}₂ \vec{γ}₃ \vec{e}₃$,
the region equation for $P$ is equivalent to
(writing $w_i := w_{γ_i}, u_i := u_{\vec{e}_i}$)
\begin{equation}
\label{e:region-eqn-k3}
w₁ = -u₁u₂
\; ; \;
w₂ = -u₂u₃
\; ; \;
w₃ = -u₃u₁
\end{equation}
\end{lemma}
\begin{proof}
Moving the last three terms on the left side of the region equation to the right,
we get
\[
\begin{pmatrix}
w₃ & 0 \\
-u₃ & w₂
\end{pmatrix}
\sim
\begin{pmatrix}
u₁ & w₁ + u₁u₂ \\
1 & u₂
\end{pmatrix}
\]
and comparing the top-right component gives $w₁ = -u₁u₂$.
We leave the rest of the details to the reader.
\end{proof}

The region equation is also quite manageable for $k = 4$
(again it will be clearer with shape parameters, see \rmkref{r:region-eqn-shape}):
\begin{lemma}
\label{l:region-eqn-k4}
For an internal 2-cell $P$ with $k = 4$ and
$∂P = \vec{γ}₁ \vec{e}₁ \vec{γ}₂ \vec{e}₂ \vec{γ}₃ \vec{e}₃ \vec{γ}₄ \vec{e}₄$,
the region equation for $P$ is equivalent to
(writing $w_i := w_{γ_i}, u_i := u_{\vec{e}_i}$)
\begin{align}
\label{e:region-eqn-k4}
\begin{split}
u₁w₂ + u₃w₁ = -u₁u₂u₃
\; ; \;
u₂w₃ + u₄w₂ = -u₂u₃u₄
\; &; \;
u₃w₄ + u₁w₃ = -u₃u₄u₁
\; ; \;
u₄w₁ + u₂w₄ = -u₄u₁u₂
\\
w₁w₃ &= (w₂ + u₂u₃)(w₄ + u₄u₁)
\end{split}
\end{align}
\end{lemma}
\begin{proof}
Expanding the product on the left side of the region equation, we get
\[
w₂w₄ + u₄u₁w₂ = u₂u₃w₄ + (w₁ + u₁u₂)(w₃ + u₃u₄)
\]
from the diagonal components being equal and non-zero, and
\[
- u₂w₄ - u₄w₁ - u₄u₁u₂ = 0
\;\; ; \;\;
- u₁w₃w₂ - u₃w₄w₂ - u₃u₄u₁w₂ = 0
\]
from the bottom-left and top-right components;
considering cyclic permutations, the second equation is implied by the first.
Thus we have the 4 equations in the first row of \eqnref{e:region-eqn-k4}.
Multiplying them by an appropriate $u_i$, we have
\[
u₄u₁w₂ + u₃u₄w₁ = -u₁u₂u₃u₄
\]
and its cyclic variants, and in particular,
\[
u₄u₁w₂ = u₂u₃w₄ \; ; \; u₁u₂w₃ = u₃u₄w₁
\]
so the equation from the diagonal components reduce to
the second row of \eqnref{e:region-eqn-k4}.
Note that its cyclic variant follows from \eqnref{e:region-eqn-k4}:
\begin{align*}
(w₁ + u₁u₂)(w₃ + u₃u₄)
&= w₁w₃ + u₃u₄w₁ + u₁(u₂w₃ + u₂u₃u₄) \\
&= (w₂ + u₂u₃)(w₄ + u₄u₁) + u₃u₄w₁ - u₁u₄w₂ \\
&= w₂w₄ + u₂u₃w₄ + u₁u₂u₃u₄ + u₃u₄w₁
= w₂w₄
\end{align*}
\end{proof}

When we specialize to link complements (see \secref{s:original-TT}),
the conditions above reduce to those of the original TT method:
the region equation \eqnref{e:region-eqn-matrix} 
becomes the region equation \eqnref{e:region-eqn-original},
while the normalization condition becomes the edge equation
\eqnref{e:edge-eqn-original}.

\subsubsection{\textbf{Geometric Labeling}}
\label{s:geometric-labeling}
\ \\ \indent

In this section, we construct the \emph{geometric labeling}
for a polyhedral decomposition $τ$ of a hyperbolic 3-manifold $M$,
which is also an algebraic solution.
This construction fails if $τ$ has an (internal) 1-cell that is
null-homotopic (by proper homotopy), and conversely,
if $τ$ does not have such 1-cells, then this construction works.
In \cite{TTmethod}, the latter case is guaranteed by the tautness condition.

As alluded to in \rmkref{r:no-hyperbolic},
algebraic solutions describe some hyperbolic geometry
in a piecemeal and local fashion; in order to keep track of
changing perspectives, we introduce the following:

\begin{definition}
\label{d:p-centered-view}
Let $\mathbf{p} = (p,v,H)$ be a triple consisting of a point $p ∈ H ⊆ \HH$
on a horosphere $H$ with a tangent vector $v ∈ T_p H ⊆ T_p\HH$ of length 1.
The \emph{$\mathbf{p}$-centering isometry $φ_{\mathbf{p}} ∈ \isomHH$}
is the unique isometry such that
$φ(p) = (0,1) ∈ \CC × \RR^+ \simeq \HH$,
$φ(H) = \{(w,z) \;|\; |w|²+(z-\frac{1}{2}\}$
is the horosphere centered at $0 ∈ \CC ⊆ ∂\HH$ containing $φ(p)$,
and $φ_*(v)$ projects to $1 ∈ \CC \simeq T_0\CC$
under the natural projection $\HH \simeq \CC \times \RR^+ \to \CC$
onto the first factor.
\end{definition}

Under most use cases, the choice of $H$ is clear from the context,
and often comes with additional data (a parametrization;
see \defref{d:param-horosphere}).
We employ the following convention:

\begin{convention}
\label{cv:centering-isometry}
When $p$ is implicitly understood to lie on a given parametrized horosphere $H$
(see \defref{d:param-horosphere}),
then we take $v ∈ T_pH$ to be the vector that is equal to 1
under the identification $T_pH \simeq T₀\CC \simeq \CC$
coming from the parametrization of $H$,
and $φ_p$, the $p$-centering isometry,
is taken to be the $(p,v,H)$-centering isometry.
\end{convention}


\begin{remark}
\label{r:centering-isometry}
The reason we use these $\mathbf{p}$-centering isometries is that
the matrix representing them are particularly simple for
the following two types of $\mathbf{p} = (p,v,H)$ that we will often encounter:
\begin{enumerate}
\item $p ∈ γ₀ := \{0\} \times \RR^+$ and
$H$ is centered at $0 ∈ \CC ⊆ ∂\HH$,
wherein $φ_{\mathbf{p}} =
    \begin{pmatrix}
    0 & w \\
    1 & 0
    \end{pmatrix}
$
for some $w ∈ \CC$, and
\item $p ∈ H₀ := \{z = 1\}$ and $H = H₀$,
wherein $φ_{\mathbf{p}} =
    \begin{pmatrix}
    1 & u \\
    0 & 1
    \end{pmatrix}
$
for some $u ∈ \CC$.
\end{enumerate}
\end{remark}

\begin{definition}
\label{d:param-horosphere}
A \emph{parametrization} of a horosphere $H ⊆ \HH$
is an orientation-preserving isometry $H \simeq \CC$.
\end{definition}

In practice, we only care about measuring the difference between two points
on a horosphere, so the parametrization only matters up to
a translation of $\CC$.
As such, we will consider two parametrizations to be equivalent
if they are related by a translation.

\begin{definition}
A \emph{parametrization} of a point $x ∈ ∂\HH$ at infinity
is a choice of a horosphere $H$ centered at $x$
with a parametrization.

Given a parametrized point $(x, H, ψ: H \simeq \CC)$,
an isometry $φ$ of $\HH$ sends it to the parametrized point
$φ_*(x, H, ψ: H \simeq \CC) := (φ(x), φ(H), ψφ^\inv: φ(H) \simeq \CC)$;
we say that these two parametrized points are related by $φ$.
\end{definition}

Note that there is an isometry relating any two given parametrized points,
and this isometry is unique up to postcomposing with a parabolic isometry
fixing the target parametrized point
(or equivalently, up to precomposing with a parabolic isometry
fixing the source parametrized point).

\begin{definition}
\label{d:geometric-crossing-label}
For a geodesic $γ$ in $\HH$ with parametrized endpoints on $∂\HH$,
the \emph{geometric crossing label $\geomw{γ}$ of $γ$} is defined as follows.
Let $(x,H,ψ:H \simeq \CC)$ and $(x',H',ψ':H' \simeq \CC)$
be the parametrized endpoints of $γ$.
Let $q = γ ∩ H$ and $q' = γ ∩ H'$.
Apply the $q$-centering isometry $φ_{q}$
(see \cvnref{cv:centering-isometry});
the image of $q'$ under this isometry, $φ_{q}(q')$, lies on the geodesic
$\{0\} \times \RR^+$,
and $φ_{q}(H)$ is centered at $0 ∈ \CC ⊆ ∂\HH$.
Then the $φ_{q}(q')$-centering isometry is given by
some off-diagonal matrix $\matwil{\geomw{γ}}$
for some $\geomw{γ} ∈ \CC$;
we define $\geomw{γ}$ to be the geometric crossing label of $γ$.
\end{definition}

Note that $\geomw{γ}$ is independent of the ordering of $q$ and $q'$,
hence we do not need to orient $γ$ to have a well-defined $\geomw{γ}$.
The definition of geometric crossing label is slightly different
from \cite[Sec. 3]{TTmethod}, but gives the same quantity up to a sign,
as can be inferred from the note after \cite[Fig. 4]{TTmethod}.
Geometric crossing label is also known as \emph{intercusp parameter}
in \cite{intercusp}.

\begin{definition}
For two geodesics $γ,γ'$ with a parametrized (common) endpoint on $∂\HH$,
the \emph{geometric edge label $\geomu{γγ'}$ associated to $(γ,γ')$}
is defined as follows.
Let $(x,H,ψ:H \simeq \CC)$ be the parametrized endpoint,
and let $q = γ ∩ H, q' = γ' ∩ H$.
Then $\geomu{γγ'}$ is defined to be the number $ψ(q') - ψ(q) ∈ \CC$.
\end{definition}

We emphasize that $\geomw{γ}$ and $\geomu{γγ'}$ depend on the parametrizations
of the endpoints of the geodesics.
Note that $\geomu{γ'γ} = - \geomu{γγ'}$.
We may think of $\geomu{γγ'}$ as the edge label of
the oriented segment $\ov{qq'}$, hence the name geometric \emph{edge} label.
We can also define geometric edge label in a similar way to the
geometric crossing label:
apply the $q$-centering isometry $φ_{q}$, so that
the images of $q$ and $q'$ lie in the horosphere $H₀ := \{z = 1\}$,
with parametrization given by projection to $\CC ⊆ ∂\HH$.
Then the $φ_{q}(q')$-centering isometry is precisely given by
the upper-triangular matrix $\matuil{-\geomu{γγ'}}$.

\begin{proposition}[Region equation for geometric labels in $\HH$]
\label{p:region-eqn-geometric}
Let $P$ be an ideal $k$-gon embedded in $\HH$
with geodesic sides $γ₁,...,γ_k$.
Let the vertices of $P$ be parametrized.
Then
\begin{equation}
\label{e:region-eqn-geometric}
	\matu{-\geomu{γ_k γ₁}}
	\matw{\geomw{γ_k}}
	⋅⋅⋅
  \matu{-\geomu{γ₁ γ₂}}
  \matw{\geomw{γ₁}}
	\sim
	\matu{0}
\end{equation}
\end{proposition}
\begin{proof}
Let $x_i ∈ ∂\HH$ be the vertex of $P$ between $γ_i$ and $γ_{i+1}$.
Let $γ_i$ and $γ_{i+1}$ intersect the horosphere $H_i$
(from the parametrization of $x_i$) at $p_i$ and $q_i$ respectively,
and let $\vec{ε}_i$ be the geodesic on $H_i$ going from $p_i$ to $q_i$.
Apply the $q_k$-centering isometry $φ_{q_k}$ to $\HH$,
and let $p_i^{(1)} := φ_{q_k}(p_i)$ and $q_i^{(1)} := φ_{q_k}(q_i)$
be their new positions.
Then the $p₁^{(1)}$-centering isometry $φ_{p₁^{(1)}}$
is precisely $\matwil{\geomw{γ₁'}}$.
Let $p_i^{(2)} := φ_{p₁^{(1)}}(p_i^{(1)})$ and
$q_i^{(2)} := φ_{p₁^{(1)}}(q_i^{(1)})$.
Then the $q₁^{(2)}$-centering isometry $φ_{q₁^{(2)}}$
is precisely $\matuil{-\hspace{-2pt}\geomu{γ₁' γ₂'}}$.
We continue in this fashion, applying the centering isometry for
$p₂^{(3)},q₂^{(4)},...,q_{k-1}^{(2k-2)},p_k^{(2k-1)},q_k^{(2k)}$,
where $p_i^{(j)}$ is the image of $p_i^{(j-1)}$ after applying the previous
centering isometry, and likewise for $q_i^{(j)}$.
The composition of these isometries
(excluding the initial $q_k$-centering isometry)
is then given by the left-hand side of the region equation above,
and it remains to observe that this composition is also
the $q_k^{(1)}$-centering isometry, which is trivial.
\end{proof}

The region equations \eqnref{e:region-eqn-geometric}
(or \eqnref{e:region-eqn-matrix})
were originally given in terms of \emph{shape parameters} in
\cite[Eqn. (1)]{TTmethod},
which are independent of parametrizations.
We recall the definition of shape parameters below,
but not the corresponding region equations
(however see \rmkref{r:region-eqn-shape}).
One potential drawback of using shape parameters is that they may become
undefined in very degenerate situations
(e.g. they are undefined for the region $R'$ in $D'$
in \rmkref{r:non-taut}).

\begin{definition}
\label{d:shape-parameter}
Let $γ,γ',γ''$ be geodesics in $\HH$,
where $γ$ and $γ'$ share a common endpoint $z₁$,
and $γ$ and $γ''$ share a common endpoint $z₂ ≠ z₁$.
Let $z₀,z₃$ be the other endpoint of $γ',γ''$ respectively.
Then the \emph{shape parameter of $(γ,γ',γ'')$}
is the cross ratio
\[
ζ_{γ;γ',γ''} := \frac{(z₀ - z₁)(z₂ - z₃)}{(z₀ - z₂)(z₁ - z₃)}
\]
\end{definition}

We usually simply say ``shape parameter of $γ$'', and write $ζ_γ$.
Note that $ζ$ is invariant with respect to swapping $γ'$ and $γ''$.
The geometric labels are related to shape parameters:

\begin{lemma}{\cite[Prop. 4.1]{TTmethod}}
\label{l:shape-param-labels}
Let $γ,γ',γ''$ be geodesics in $\HH$,
with $γ$ sharing endpoints with $γ'$ and $γ''$ as in \defref{d:shape-parameter}.
Suppose the endpoints of $γ$ are parametrized,
and suppose $\geomu{γγ'} ≠ 0 ≠ \geomu{γγ''}$.
Then
\[
ζ_{γ;γ',γ''} = - \frac{\geomw{γ}}{\geomu{γγ'}\geomu{γγ''}}
\]
\end{lemma}

\begin{remark}
\label{r:region-eqn-shape}
When geometric edge labels are non-zero,
we can rewrite the region equations in terms of shape parameters,
as in \cite[Prop. 4.2]{TTmethod}.
The cases of $k = 3$ and $k = 4$ sides are particularly simply and intuitive.
Write $ζ_i = ζ_{γ_i;γ_{i-1},γ_{i+1}}$.
For $k = 3$, \lemref{l:region-eqn-k3} implies that the region equation
is equivalent to $ζ_i = -1$ for $i = 1,2,3$,
while for $k = 4$, \lemref{l:region-eqn-k4} implies that the region equation
is equivalent to $ζ_i + ζ_{i+1} = -1$ for $i = 1,2,3,4$
(the second row in \eqnref{e:region-eqn-k4} is redundant).
These recover the consequences of \cite[Eqn. (2)]{TTmethod}.
However, \lemref{l:region-eqn-k4} is slightly more general,
as it allows for solutions with $u_i \equiv 0$.
\end{remark}

Now we give the corresponding definitions in the setting of
a cusped hyperbolic 3-manifold.

\begin{definition}
\label{d:param-torus}
For a flat torus $T ⊆ M$ embedded as a torus slice in a cusp
in a cusped hyperbolic 3-manifold,
a \emph{parametrization} of $T$ is a local isometry $\CC \to T$.
\end{definition}

We also consider two parametrizations to be equivalent
if they are related by a translation.
Thus, if $\HH \to M$ is a local isometry (hence a covering map),
and the horosphere $\wdtld{T}$ is a
connected component of the preimage of $T$,
then a parametrization of $\wdtld{T}$ descends to a parametrization of $T$,
and this gives a bijective correspondence between
parametrizations on $\wdtld{T}$ and $T$.

\begin{definition}
\label{d:param-cusp}
A \emph{parametrization} of an end of a cusped hyperbolic 3-manifold
is a choice of parametrized flat torus (as in \defref{d:param-torus}).
Given an oriented simple loop $μ$ on the boundary torus,
the \emph{canonical parametrization with respect to $μ$}
is the flat torus $T$ on which $\hat{μ}$, the geodesic homotopic to $μ$,
has length 1, with the parametrization of $T$ that lifts $\hat{μ}$ to 1.
\end{definition}

As before, this is related to the $\HH$ setting:
given a local isometry $\HH \to M$, and a point $x ∈ ∂\HH$
that descends to an ideal vertex (corresponding to an end) of $M$,
a parametrization of the ideal vertex/end lifts to a parametrization of $x$,
and this is a bijective correspondence between sets of parametrizations.

\begin{definition}
\label{d:geodesic-like}
Given a cusped hyperbolic 3-manifold $M$,
a \emph{geodesic-like curve $γ$} is a properly embedded curve
traveling between ends of $M$,
such that $γ$ is homotopic (rel. boundary) to a geodesic;
equivalantly, $γ$ is not homotopic to a curve contained in a cusp.
\end{definition}

\begin{definition}
Given a hyperbolic 3-manifold $M$ and two parametrized ends
(possibly the same ends with different parametrizations),
the \emph{geometric crossing label $\geomw{γ}$} of a geodesic-like curve $γ$
connecting the two ends
is defined as the geometric crossing label $\geomw{\wdtld{γ}}$ of $\wdtld{γ}$,
a geodesic in $\HH$ homotopic (rel. boundary) to the lift of $γ$
along the covering map $\HH \to M$.
\end{definition}


\begin{definition}
Given a hyperbolic 3-manifold $M$ with a parametrized ends
and two geodesic-like curves $γ,γ'$,
together with a \emph{peripheral edge $\vec{ε}$ from $γ$ to $γ'$}, that is,
an oriented curve $\vec{ε}$ connecting $γ$ to $γ'$ contained in some cusp,
the \emph{geometric edge label $\geomu{\vec{ε},γ,γ'}$ associated to
$(\vec{ε}, γ, γ')$}
is defined as the geometric edge label $\geomu{\wdtld{γ}\wdtld{γ'}}$,
where $\wdtld{γ},\wdtld{γ'}$ are geodesics in $\HH$,
such that there exists a curve $\wdtld{\vec{ε}}$ in $\HH$
that connects $\wdtld{γ}$ to $\wdtld{γ'}$,
and $\wdtld{γ}$, $\wdtld{γ'}$, and $\wdtld{\vec{ε}}$
are homotopic to lifts of $γ$, $γ'$ and $\vec{ε}$
(and throughout the homotopy, $\wdtld{\vec{ε}}$ always connects
$\wdtld{γ}$ to $\wdtld{γ'}$).
\end{definition}

\begin{definition}
\label{b:geometric-labels}
Let $τ$ be an ideal polyhedral decomposition of a cusped hyperbolic 3-manifold $M$,
with choice of meridian $μ_i$ for each end,
which determines the canonical parametrization with respect to $μ_i$ for each end.
Suppose every edge of $τ$ is geodesic-like.
Then the \emph{geometric labeling of $\ov{τ}$}, denoted by
$\geomlabel = (\geomw{}(-), \geomu{}(-))$,
is the labeling that assigns the geometric crossing label
$\geomw{}(\ov{γ}) := \geomw{γ}$ to the internal 1-cell $\ov{γ}$ of $\ov{τ}$
corresponding to the edge $γ$ of $τ$,
and assigns the geometric edge label
$\geomu{}(\vec{e}) := \geomu{\vec{e},γ,γ'}$
to the oriented boundary 1-cell $\vec{e}$ that goes from
$γ$ to $γ'$, where $γ,γ'$ are edges of $τ$.
\end{definition}

\begin{proposition}
\label{p:geom-sol-exist}
Let $τ$ be an ideal polyhedral decomposition of a cusped hyperbolic 3-manifold $M$,
such that every edge of $τ$ is geodesic-like.
For any choice of meridians on each end,
the resulting geometric labeling $\geomlabel = (\geomw{}(-), \geomu{}(-))$
of $\ov{τ}$ is an algebraic solution (as in \defref{d:alg-soln}).
\end{proposition}

\begin{proof}
The normalization condition follows almost by definition of geometric edge labels.
For the region equations, consider a 2-cell $P$ of $τ$,
with boundary $∂P = \vec{γ}₁ \cdots \vec{γ}_k$,
and corresponding internal 2-cell $\ov{P}$ of $\ov{τ}$,
with boundary $∂\ov{P} = \ov{γ}₁ \vec{e}₁ \cdots \ov{γ}_k \vec{e}_k$
alternating between internal 1-cells $\ov{γ}_i$ and
peripheral 1-cells $\vec{e}_i$.
Let $\wdtld{P}$ be a lift of $P$ to $\HH$,
with boundary $\wdtld{\vec{γ}}₁ \cdots \wdtld{\vec{γ}}_k$.
Homotope $\wdtld{P}$ to an ideal polygon $P'$,
so that the boundary edges $\wdtld{\vec{γ}}_i$
become geodesics $γ_i' ⊆ ∂P'$.
Then it is clear that $\geomw{}(\ov{γ}_i) = \geomw{γ_i'}$
and $\geomu{}(\vec{e}_i) = \geomu{γ_i' γ_{i+1}'}$,
so by the region equation for geometric labels in $\HH$
(i.e. \prpref{p:region-eqn-geometric}), we are done.
\end{proof}


\begin{remark}
\label{r:lin-indep-crit}
For a loop $α$ on a peripheral torus $T_i$,
if $α$ is not homotopic to a multiple of the chosen meridian $μ_i$,
then $\geomu{}(α)$ must be $\RR$-linearly independent from
$\geomu{}(μ_i) = 1$, i.e. $\geomu{}(α) ∈ \RR$.
This can be used as a simple necessary criterion to rule out that
a given algebraic solution $Ω$ is the geometric labeling.
\end{remark}

\begin{remark}
\label{r:taut-shape-param}
In \cite{TTmethod}, the tautness condition (see \defref{d:taut}) is used
in the proof of \cite[Prop 2.1]{TTmethod} to ensure that the ideal vertices of a region
are pairwise distinct;
as we see above, it is not necessary to have all pairs of vertices be distinct,
but rather only pairs of adjacent vertices (i.e. those connected by crossing arcs),
for the geometric labeling to be definable.
Indeed, we give an example in \rmkref{r:non-taut} of a link diagram that is not taut
but admits a geometric labeling that corresponds to the geometric structure,
which shows that tautness is sufficient but not necessary.

Tautness may be required in \cite{TTmethod}
because they write their region equation \cite[Eqn. (1)]{TTmethod} 
in terms of shape parameters,
which requires certain regularities (i.e. more pairs of ideal vertices being distinct),
so it is convenient to simply give a strong condition (i.e. tautness)
that guarantees pairwise distinctness
(see also \cite[Thm. 3.1, Prop. 4.1]{TTmethod}).
However, as we had noted before, they also give the region equation as we did
(in terms of crossing and edge labels only),
where edge labels are not required to be nonzero,
so it would have been a simple matter, had they wanted to,
for the authors of \cite{TTmethod} to consider relaxing the tautness condition as well.
\end{remark}


\begin{proposition}
\label{p:labels-isom}
Let $M$ and $M'$ be hyperbolic 3-manifolds with parametrized ends,
and let $f: M \to M'$ be an isometry.
For the flat torus $T_i$ parametrizing the end $i$ of $M$,
let $T_i'$ be the flat torus parametrizing the end $f(i)$ of $M'$;
the collar neighborhood around $f(i)$
gives a natural identification $T_i' \simeq f(T_i)$
that is a similarity transformation.
Suppose, under their parametrizations and the natural identification,
$f|_{T_i}: T_i \to T_i'$ sends 1 to $λ_i ∈ \CC$;
note that since $f$ and the parametrizations are local isometries,
$|λ_i| = 1$.

Let $γ,γ'$ be geodesic-like curves in $M$ with common endpoint at end $i$,
and the other endpoint of $γ$ be at end $j$.
Let $\vec{ε}$ be an oriented curve connecting $γ$ to $γ'$
contained in some collar neighborhood of $i$.
If $f$ is orientation-preserving, then
\[
\geomu{f(\vec{ε}),f(γ),f(γ')} = λ_i \geomu{\vec{ε},γ,γ'}
\;\; ; \;\;
\geomw{f(γ)} = λ_i λ_j \geomw{γ}
\]
and if $f$ is orientation-reversing, then
\[
\geomu{f(\vec{ε}),f(γ),f(γ')} = λ_i \ov{\geomu{\vec{ε},γ,γ'}}
\;\; ; \;\;
\geomw{f(γ)} = λ_i λ_j \ov{\geomw{γ}} \ .
\]
\end{proposition}

In particular, if we take $M'$ to be $M$ with the opposite orientation,
take $f = \id_M$, and the parametrization of $M'$ that is
the complex conjugate of the parametrization of $M$,
then the geometric labels in $M'$ are just the complex conjugate of
the geomtric labels in $M$.

\begin{proof}
The geometric edge label $\geomu{\vec{ε}}$ ``measures the ratio'' between
$\vec{ε}$ and 1.
Hence, in the orientation-preserving case,
$\geomu{\vec{ε}}$ also ``measures the ratio'' between $f(\vec{ε})$ and $λ_i$;
in the orientation-reversing case, we have to complex conjugate it.
Since $\geomu{f(\vec{ε})}$ ''measures the ratio'' between $f(\vec{ε})$ and 1,
we just have to scale by $λ_i$.
For the geometric crossing labels,
one can prove it first principles, but a more intuitive approach
is through shape parameters (see \defref{d:shape-parameter}):
since $f$ is an isometry,
$ζ_{f(γ);f(γ'),f(γ'')} = ζ_{γ;γ',γ''}$ (orientation-preserving)
or $ζ_{f(γ);f(γ'),f(γ'')} = \ov{ζ_{γ;γ',γ''}}$ (orientation-reversing).
Then using the relation between shape parameters and labels
(see \lemref{l:shape-param-labels}),
the desired result follows.
\end{proof}

\begin{remark}
\label{r:meridian-scaling-geom-label}
Modified normalization equation of the TT method (as in \rmkref{r:meridian-scaling})
has a clear geometric meaning for the geometric labeling:
in \defref{b:geometric-labels}, instead of the canonical parametrization at $T_i$,
we choose the flat torus such that the meridian $μ_i$
has length $|λ_i|$ and the parametrization identifies it with $λ_i$.
See also \rmkref{r:meridian-scaling-geom-realz}
for another perspective on the geometric meaning of rescaling.
\end{remark}


\subsection{Extending an Algebraic Solution to Arbitrary Arcs and Peripheral Edges}
\label{s:extension-adm}
\ \\ \indent

Given a parametrized cusped hyperbolic 3-manifold $M$,
and an ideal polyhedral decomposition $τ$,
the geometric crossing label is defined for any geodesic-like arc,
not just internal 1-cells of the (truncated) polyhedral decomposition;
likewise, the geometric edge label is defined for any peripheral edge
between geodesic-like arcs, not just for peripheral 1-cells.
In the following, we describe a procedure for extending an algebraic solution
$Ω = (w(-),u(-))$ to other geodesic-like arcs and peripheral edges,
that is, to assign crossing and edge labels to them;
when such labels exist, they are well-defined and do not depend on
the choices made during the construction below.
These labels are in fact geometric labels under the ``geometric realization''
as we will consider in the next section;
see \eqnref{e:phi-xtilde}, \rmkref{r:geom-realzn-adm}.

Let $γ$ be a geodesic-like arc.
There exists 1-cells $γ₁,...,γ_k$ of $τ$,
with peripheral edge $\vec{ε}₀$ from $γ$ to $γ₁$,
peripheral edge $\vec{ε}_k$ from $γ_k$ to $γ$,
and peripheral 1-cell $\vec{ε}_i$ of $\ov{τ}$ from $γ_i$ to $γ_{i+1}$
for $i = 1,...,k-1$,
such that $\vec{ε}_0 \ov{γ₁} \vec{ε}₁ \cdots \ov{γ_k} \vec{ε}_k \ov{γ}$
bounds a truncated ideal polygon (not necessarily embedded).

We say that $γ$ is \emph{$Ω$-admissible} if there exist 
$u(\vec{ε}_0), u(\vec{ε}_k), w(γ)$ that satisfies:
\begin{equation}
\label{e:arb-crossing-label-1}
\matu{-u(\vec{ε}₀)} \matw{(w(γ_k)}
⋅⋅⋅
\matu{-u(\vec{ε}₁)} \matw{w(γ₁)}
\matu{-u(\vec{ε}₀)} \matw{w}
\sim
\begin{pmatrix}
1 & 0 \\
0 & 1
\end{pmatrix}
\end{equation}

Rearranging, we have
\begin{equation}
\label{e:arb-crossing-label-2}
\matw{w(γ_k)}
⋅⋅⋅
\matu{-u(ε₁)}
\matw{w(γ₁)}
\sim
\begin{pmatrix}
u(\vec{ε}_k) & u(\vec{ε}₀) u(\vec{ε}_k) + w(γ)
\\
1 & u(\vec{ε}₀)
\end{pmatrix}
\end{equation}
and if the bottom left entry of the left-hand side is non-zero,
then we can rescale it to 1 and easily obtain
$u(\vec{ε}_0), u(\vec{ε}_k), w(γ)$,
and otherwise, $γ$ is not $Ω$-admissible.
Note that all internal 1-cells of $τ$ are $Ω$-admissible.
We also note that $γ$ may be $Ω$-admissible for $Ω$,
but not $Ω'$-admissible for some other algebraic solution $Ω'$;
however, all geodesic-like arcs are $\geomlabel$-admissible.
See \rmkref{r:geom-realzn-adm} for a geometric interpretation of $Ω$-admissibility.

Next consider a peripheral edge $\ov{ε}$ from $γ$ to $γ'$,
where $γ,γ'$ are $Ω$-admissible arcs.
Let $γ \sim \vec{ε}_0 γ₁ ⋅⋅⋅ γ_k \vec{ε}_k$,
$γ' \sim \vec{ε}_0' γ₁' ⋅⋅⋅ γ_l' \vec{ε}_l'$,
as in the construction above.
Let $α$ be a concatenation of peripheral 1-cells
that is homotopic to $\vec{ε}_0^\inv \vec{ε} \vec{ε}_0'$.
We define
\begin{equation}
\label{e:arb-edge-label}
u(\vec{ε}) = u(\vec{ε}_0) + u(α) - u(\vec{ε}_0')
\end{equation}
where $u(α)$ is the sum of the edge labels of the peripheral 1-cells in $α$.

In the following, we will abuse notation by
dropping the overline from the truncation $\ov{γ}$,
as it should be clear which we are referring to.
Let us introduce some notation:

\begin{definition}
\label{d:W}
For a curve $α$ that is a 1-cell of $\ov{τ}$,
we define $W_Ω(α)$ to be the matrix given by
\begin{equation}
\label{e:label-matrix}
W_Ω(γ) = \matw{w(γ)}
\;\;\; ; \;\;\;
W_Ω(\vec{ε}) = \matu{-u(\vec{ε})}
\end{equation}
More generally, for a concatenation $α = α₁ ⋅⋅⋅ α_k$ of such curves,
we define $W_Ω(α)$ to be the composition $W_Ω(α_k) ⋅⋅⋅ W_Ω(α₁)$.
\end{definition}

Note that $W_Ω(-)$ is contravariant with respect to composition,
i.e., $W_Ω(α ∘ α') = W_Ω(α') W_Ω(α)$,
which is natural in terms of order of operation,
since the $W_Ω$'s are meant to represent isometries,
particularly for $Ω = \geomlabel$
(see \rmkref{r:centering-isometry}).

\begin{lemma}
\label{l:TT-W-homotopic}
For any two homotopic curves $α,α'$
that are concatenations of peripheral edges or crossing arcs in $τ₀$,
we have $W_Ω(α) \sim W_Ω(α')$.
\end{lemma}

\begin{proof}
The region equations \eqnref{e:region-eqn-matrix}
boil down to $W_Ω(∂P) \sim I$, the identity matrix, for any 2-cell $P$ of $\ov{τ}$.
Since $α^\inv α'$ is homotopic to a trivial loop,
and removing the 3-cells in $\ov{τ}$ does not affect the fundamental group,
it follows that $W_Ω(α^\inv α') \sim I$.
\end{proof}

\begin{proposition}
\label{p:labels-well-defn}
The constructions above of labels $w(γ), u(\vec{ε})$
for $Ω$-admissible arcs $γ$ and peripheral edges $\vec{ε}$ between them
are independent of the choices made during the construction.
\end{proposition}
\begin{proof}
Suppose $γ$ is homotopic to a concatenation
of crossing arcs and peripheral edges
in two ways as in the construction above,
say $γ \sim ε_0 γ₁ ⋅⋅⋅ γ_k ε_k$ and $γ \sim ε_0' γ₁' ⋅⋅⋅ γ_l' ε_l'$
(we omit the overhead arrow).
We want to show that the crossing labels $w,w'$ for $γ$
that the construction above yields using these
two presentations of $γ$ are equal.

Let $ε$ be a concatenation of peripheral edges in $τ₀$
that is homotopic (in the relevant peripheral torus) to $ε_0^\inv ε_0'$;
similarly, let $ε'$ be a concatenation of peripheral edges in $τ₀$
that is homotopic (in the relevant peripheral torus) to $ε_l' ε_k^\inv$.

Let $α = γ₁ ε₁ ⋅⋅⋅ γ_k$,
$β = γ₁' ε₂' ⋅⋅⋅ γ_l'$,
$α' = ε β ε'$.
By definition, $α$ and $α'$ are homotopic
(to $ε_0^\inv γ ε_k^\inv$),
so by \lemref{l:TT-W-homotopic},
$W_Ω(α) \sim W_Ω(α') = W_Ω(ε') W_Ω(β) W_Ω(ε)$.

Thus we have
\begin{align*}
\begin{pmatrix}
0 & w \\
1 & 0
\end{pmatrix}
\sim
W_Ω(ε_k) W_Ω(α) W_Ω(ε_0)
\sim
W_Ω(ε_k) W_Ω(α') W_Ω(ε_0)
\sim
W_Ω(ε' ε_k) W_Ω(β) W_Ω(ε_0 ε)
\\
\sim
W_Ω(ε_l'^\inv ε' ε_k)
\begin{pmatrix}
0 & w' \\
1 & 0
\end{pmatrix}
W_Ω(ε_0 ε ε_0'^\inv)
\end{align*}

Since $ε_l'^\inv ε' ε_k$ and $ε_0 ε ε_0'^\inv$ lie in a peripheral torus,
$W_Ω(ε_l'^\inv ε' ε_k)$ and $W_Ω(ε_0 ε ε_0'^\inv)$
must be strict upper triangular,
hence by some basic algebraic argument,
the above equality implies that they must both be identity.
This immediately implies $w = w'$,
and $u(ε) = -u(ε_0) + u(ε_0')$, $u(ε') = u(ε_l') - u(ε_k)$.
The analogous result for edge labels follows easily from this as well.
\end{proof}

A similar argument shows that we have the analog of \lemref{l:TT-W-homotopic}
for an arbitrary collection of $Ω$-admissible arcs
and peripheral edges between them
(not just those from $\ov{τ}$):

\begin{lemma}
\label{l:TT-W-homotopic-arbitrary}
For any two homotopic curves $α,α'$
that are concatenations of $Ω$-admissible arcs and peripheral edges between them,
we have $W_Ω(α) \sim W_Ω(α')$.
\end{lemma}

\ \\
\subsection{Algebraic Solutions and $\PSL$-representations of the
Fundamental Group}
\label{s:alg-to-geom-reconstruction}
\ \\ \indent

For a hyperbolic 3-manifold $M$ with an ideal polyhedral decomposition $τ$,
the geometric labeling records data about the geometry
in terms of local data concentrated at edges of $τ$.
We give a procedure to reconstruct a ``geometry'' on $M$ from any algebraic solution,
which produces the complete hyperbolic geoemtry on $M$ when we start with
the geometric labeling.
The ``geometry'' is closely related to the induced parabolic representation
of the link group into $\PSL$ from \cite[Sect. 5]{TTmethod}.

Let us recall some notation from \cite{repnvol}.
Denote by $\wdht{M} = M ∪ \{x_i\}$ the compactification of $M$
obtained by adding an \emph{ideal point} to each end of $M$;
equivalently, $\wdht{M} = \ov{M} / T_i \sim x_i$.
Fix a basepoint $m₀ ∈ M$ and consider the universal cover $\wdtld{M}$.
Let $\hattld{M} = \wdtld{M} ∪ \{\wdtld{x}_i\}$ be constructed from $\wdtld{M}$
by adding lifts of ideal points in $\wdht{M}$; more precisely,
keeping in mind the concrete definition of $\wdtld{M}$ as the set of
homotopy classes of paths in $M$ starting at $m₀$,
i.e. $\wdtld{M} = \{[α : [0,1] \to M ] \;|\; α(0) = m₀\}$,
the ideal points of $\hattld{M}$ are homotopy classes of paths $α$
in $\wdht{M}$ starting at $m₀$ and ending at an ideal point,
such that $α([0,1)) ⊆ M$.
The action of the fundamental group $π₁(M,m₀)$ on $\wdtld{M}$,
given by $g ⋅ [α] = [gα]$ (i.e. by concatenation of paths),
naturally extends to $\hattld{M}$.
Each ideal point $\wdtld{x} = [α]$ of $\hattld{M}$ defines the peripheral subgroup
$T_{\wdtld{x}} \simeq \ZZ ⊕ \ZZ$ consisting of elements fixing $\wdtld{x}$;
it is given explicitly by
\[
T_{\wdtld{x}} := \{ [\ov{α}δ\ov{α}^\inv \;|\;\
	δ: \textrm{peripheral loop in neighborhood of }\wdtld{x};
	\;\; \ov{α}: \textrm{truncation of } α \}
⊆ π₁(M,m₀)
\]


\begin{definition}
\label{d:geom-realz}
Let $M$ be a 3-manifold,
and let $τ$ be an ideal polyhedral decomposition of $M$.
Fix a basepoint $m₀ ∈ M$, so that we have $\hattld{M}$
from the above discussion.
Choose a meridian $μ_i$ for each end of $M$.
Given an algebraic solution $Ω = (w(-),u(-))$ of $τ$,
a \emph{geometric realization of $Ω$ (for ideal pionts of $\hattld{M}$)}
is an assignment $z(\wdtld{x})$ of a parametrized point in $∂\HH$
to each ideal point $\wdtld{x}$ of $\hattld{M}$,
such that the geometric crossing and edge labels among the $z(\wdtld{x})$
are given by $Ω$. More precisely,
\begin{itemize}
\item for any 1-cell $γ$ of $τ$ with a lift $\wdtld{γ}$
that connects two ideal points $\wdtld{x},\wdtld{x}'$,
we have $\geomw{\wdtld{γ}_{\HH}} = w(γ)$, where
$\wdtld{γ}_{\HH}$ is the geodesic with endpoints $z(\wdtld{x})$ and $z(\wdtld{x}')$,

\item for any peripheral 1-cell $\vec{ε}$ in $\ov{τ}$ from $γ$ to $γ'$,
which lifts to some segment $\wdtld{ε}$ in $\wdtld{M}$
connecting $\wdtld{γ}$ to $\wdtld{γ}'$,
we have $\geomu{\wdtld{γ}_{\HH} \wdtld{γ}_{\HH}'} = u(\vec{ε})$.
\end{itemize}
\end{definition}

We can construct geometric realizations as follows.
Choose distinct ideal points $\wdtld{x}₀,\wdtld{x}₁$ of $\hattld{M}$
that are connected by the lift $\wdtld{γ}₀$ of a 1-cell $γ$ of $τ$.
Choose two distinct points $z(\wdtld{x}₀),z(\wdtld{x}₁) ∈ ∂\HH$
and parametrizations such that the geometric crossing label of the geodesic between
$z(\wdtld{x}₀)$ and $z(\wdtld{x}₁)$ is $w(γ₀)$.
Assign $z(\wdtld{x}₀)$ and $z(\wdtld{x}₁)$ to $\wdtld{x}₀$ and $\wdtld{x}₁$
respectively. The rest of the construction is based on an inductive procedure:
choose a new ideal point $\wdtld{x} ∈ \hattld{M}$
such that there is a 1-cell $γ$ of $τ$ with lift $\wdtld{γ}$ connecting $\wdtld{x}$
to some ideal point $\wdtld{x}'$ that has been assigned $z(\wdtld{x}') ∈ ∂\HH$,
and then we find the unique point parametrized point $z(\wdtld{x}) ∈ ∂\HH$
that is compatible with the assignments already made.
More precisely, there must be some lift $\wdtld{γ}'$ of a 1-cell $γ'$
used in a previous step to define the assignment of $z(\wdtld{x}')$ to $\wdtld{x}'$,
and it is realized in $\HH$ as some geodesic $γ_{\HH}'$
(with one endpoint at $z(\wdtld{x}')$).
We take $γ_{\HH}$ to be the geodesic
emanating from $z(\wdtld{x}')$ such that the geometric edge label
$\geomu{γ_{\HH}' γ_{\HH}}$ associated to $(γ_{\HH}',γ_{\HH})$
is exactly the sum of edge labels along the path of peripheral edges connecting
$\wdtld{γ}'$ to $\wdtld{γ}$.
Then we set $z(\wdtld{x})$ to be the other endpoint of $γ_{\HH}$,
and we parametrize $z(\wdtld{x})$ so that the geometric crossing label
$\geomw{γ_{\HH}}$ is $w(γ)$.

We can get a more concrete picture of this procedure
by choosing $z(\wdtld{x}₀) = ∞$ and $z(\wdtld{x}₁) = 0$,
and give the parametrization $(∞,H₀,ψ₀)$ on $z(\wdtld{x}₀)$,
where $H₀$ is the horizontal plane at height 1, and $ψ₀$ is the projection
down to $\CC ⊆ ∂\HH$.
Then the parametrization for $z(\wdtld{x}₁)$ is obtained by sending
the parametrization of $z(\wdtld{x}₀)$ to $z(\wdtld{x}₁)$ by the isometry
$W_{Ω}(γ₀) = (0 \; w(γ₀); 1 \; 0) ∈ \PSL$.
In general, let $\wdtld{x} ∈ \hattld{M}$ be an ideal point,
and let $\wdtld{x}₀' = \wdtld{x}₀, \wdtld{x}₁',...,\wdtld{x}_k' = \wdtld{x}$
be a sequence of ideal points in $\hattld{M}$ connected via
lifts $\wdtld{γ}₁,...,\wdtld{γ}_k$ of 1-cells $γ₁,...,γ_k$ of $τ$.
Let $\vec{ε}_i$ a path of peripheral edges connecting $γ_{i-1}$ to $γ_i$
that lifts to a path connecting $\wdtld{γ}_{i-1}$ to $\wdtld{γ}_i$,
and let $u_i$ be the sum of the edge labels along $\vec{ε}_i$.
Then we assign to $\wdtld{x}$ the parametrized point $(φ_{\wdtld{x}})_*((∞, H₀, ψ₀))$,
where $φ_{\wdtld{x}}$ is the isometry
\begin{equation}
\label{e:phi-xtilde}
φ_{\wdtld{x}} := W_{Ω}(\vec{ε}₁ γ₁ ⋅⋅⋅ \vec{ε}_k γ_k)
= \matw{γ_k} \matu{-u_k} ⋅⋅⋅ \matw{γ₁} \matu{-u₁}
\end{equation}

\begin{lemma}
\label{l:geom-realzn-indep}
The geometric realization $z(-)$ constructed above
only depends on the choice of the first two ideal points and their assignments.
Moreover, for any geometric realization $z'(-)$ constructed using different choices,
there exists a unique isometry $φ$ sending $z(\wdtld{x})$ to $z'(\wdtld{x})$
as parametrized points for every ideal point $\wdtld{x} ∈ \hattld{M}$.
\end{lemma}
\begin{proof}
All the constructions in the inductive portion are local,
in the sense that the parametrized point $z(\wdtld{x})$
only depends on a previous point $z(\wdtld{x}')$ and some 1-cell connected to it.
More precisely, suppose we choose the same ideal points $x₀,x₁$ to start,
but choose different parametrized points
$z'(\wdtld{x}₀)$ and $z'(\wdtld{x}₁)$;
there exists an isometry $φ'$ that sends $z(\wdtld{x}₀)$ to $z'(\wdtld{x}₀)$
(as parametrized points), and there is a unique such $φ'$ that sends
the geodesic between $z(\wdtld{x}₀)$ and $z(\wdtld{x}₁)$ to
the geodesic between $z'(\wdtld{x}₀)$ and $z'(\wdtld{x}₁)$,
which forces $φ_*'(z(\wdtld{x}₁)) = z'(\wdtld{x}₁)$ as well.
Then, assuming we keep the order in which we choose the next new ideal point,
we must have $z'(\wdtld{x}) = φ_*'(z(\wdtld{x}))$.

So we can assume without loss of generality that $z(\wdtld{x}₀) = ∞$
with parametrization $(∞,H₀,ψ₀)$ and $z(\wdtld{x}₁) = 0$ as above.
By \lemref{l:TT-W-homotopic-arbitrary},
the isometry \eqnref{e:phi-xtilde} is independent on the path taken,
and hence $z(\wdtld{x})$ is independent on the order of choosing new ideal points.

Finally, if we choose a different pair of starting ideal points,
say $\wdtld{x}₀'$ and $\wdtld{x}₁'$, and construct some $z'(-)$,
then by similar considerations, the construction starting with
$\wdtld{x}₀$ and $\wdtld{x}₁$ with initial assignments
$z''(\wdtld{x}₀) := z'(\wdtld{x}₀)$ and $z''(\wdtld{x}₁) := z'(\wdtld{x}₁)$
gives the same assignment, i.e. $z'' \equiv z'$.
\end{proof}

\begin{remark}
\label{r:geom-realzn-adm}
Geometric realizations give a very concrete picture of $Ω$-admissibility:
a 1-cell $γ$ is $Ω$-admissible if and only if $z(\wdtld{x}) ≠ z(\wdtld{x}')$,
where $\wdtld{x},\wdtld{x}'$ are ideal points of $\hattld{M}$
that are connected by a lift of $γ$.
\end{remark}

\begin{remark}
\label{r:geom-realzn-basepoint}
Let $\wdtld{M}'$ be the universal cover constructed with basepoint $m₀'$.
A path $γ₀: m₀' \to m₀$ determines a map $γ₀⋅-: \wdtld{M} \to \wdtld{M}'$,
given by $[α] \mapsto [γ₀α]$, which commutes with the covering maps to $M$.
The map also extends to ideal points, and we also write
$γ₀⋅-: \hattld{M} \to \hattld{M}'$.
Conjugation by $γ₀$ gives a group homomorphism
$γ₀⋅-⋅γ₀^\inv: π₁(M,m₀) \to π₁(M,m₀')$,
and the map $γ₀⋅-$ is equivariant with respect to this group homomorphism.
Let $z'(-)$ be a geometric realization of $Ω$ for ideal points of $\hattld{M}'$.
Precomposing with $γ₀⋅-$, we obtain a geometric realization of $Ω$
for ideal points of $\hattld{M}$.
Thus, the dependence of geometric realizations on the choice of basepoint is mild,
even essentially trivial if we only work locally, in the sense that
the restriction of geometric realizations to simply connected domains
does not depend on the basepoints.

Even more concretely,
consider a polyhedron $P ∈ τ$ in the polyhedral decomposition of $M$.
Let $\wdtld{P}$ be a lift of $P$ in $\wdtld{M}$.
Let $z|_{\wdtld{P}}$ denote the restriction of the geometric realization $z'$
to the vertices of $\wdtld{P}$.
Identifying $\wdtld{P}$ and $P$ by the covering map,
$z|_{\wdtld{P}}$ is an assignment of parametrized points in $∂\HH$
to the vertices of $P$ such that crossing and edge labels align
(as in \defref{d:geom-realz});
we call such an assignment a geometric realization of $Ω$ restricted to $P$.
If we chose a different lift $\wdtld{P}'$, the restriction $z|_{\wdtld{P}'}$
will differ from $z|_{\wdtld{P}}$ by a global isometry of $\HH$.
Choosing a different $z$ also changes $z|_{\wdtld{P}}$ by a global isometry.
Indeed, it is easy to show that any geometric realization of $Ω$ restricted to $P$
is given by $z|_{\wdtld{P}}$ for some $z,\wdtld{P}$.
Then the ``essentially trivial'' dependence of geometric realizations on
the choice of basepoint refers to the fact that using geometric realizations
of $Ω$ for $\wdtld{M}'$ in the above discussion yields
the same geometric realizations of $P$ (up to global isometry).
\end{remark}

Observe that given a geometric realization $z(-)$ and some $g ∈ π₁(M,m₀)$,
we can consider the assignment $z^{g}(\wdtld{x}) := z(g⋅\wdtld{x})$,
which can easily be checked to be a geometric realization as well.
Hence by \lemref{l:geom-realzn-indep}, there exists a unique isometry
$φ^g$ such that $z^{g}(\wdtld{x}) = φ_*^g(z(\wdtld{x}))$.
The assigment $g \mapsto φ^g$ respects composition, hence:

\begin{definition}
The \emph{$\PSL$-representation underlying a geometric realization $z$ of $Ω$}
is the group homomorphism $ρ_{Ω,z}: π₁(M,m₀) \to \PSL$ given by $g \mapsto φ^g$,
where $φ^g$ is the unique isometry such that $φ_*^g(z(\wdtld{x})) = z(g⋅\wdtld{x})$.
\end{definition}

Note that $ρ_{Ω,z}$ heavily depends on $z$, and importantly,
it is not enough to record the image of $z$
(as a collection of parametrized points in $∂\HH$),
but also the rest of the data of $z$, i.e. the exact assignments to ideal points.
Indeed, given another geometric realization $z'$,
there exists some isometry $φ$ such that $φ_* ∘ z = z'$.
By definition, we have (omitting the subscript $*$):
\[
ρ_{Ω,z'}(g)(z'(\wdtld{x})) = z'(g⋅\wdtld{x})
= φ(z(g⋅\wdtld{x}))
= φ(ρ_{Ω,z}(g)(z(\wdtld{x}))
= φ(ρ_{Ω,z}(g)(φ^\inv(z'(\wdtld{x})))
\]
In other words, $ρ_{Ω,z'}(g) = φ ∘ ρ_{Ω,z}(g) ∘ φ^\inv$.
In particular, when $z'(\wdtld{x}) = z(h⋅\wdtld{x})$ for some $h ∈ π₁(M,m₀)$,
i.e. $z' = ρ_{Ω,z}(h) ∘ z$, then $ρ_{Ω,z'}(g) = ρ_{Ω,z}(hgh^\inv)$;
in this case, the images of $z'$ and $z$ are the same as sets,
but the resulting $\PSL$-representations are different,
albeit only by an inner automorphism.

We also note that $ρ_{Ω,z}$ is always a parabolic representation,
since for an element $g$ in the peripheral subgroup $T_{\wdtld{x}}$
associated to an ideal point $\wdtld{x}$ of $\hattld{M}$,
$g$ fixes $\wdtld{x}$ by definition,
so $ρ_{Ω,z}(g)$ must preserve the parametrized point $z(\wdtld{x})$,
and hence is a parabolic isometry (or the identity).

\begin{remark}
When $M$ is a link complement as in \secref{s:original-TT},
the representations $ρ_{Ω,z}$ are the same as the
``holonomy representation associated to $Ω$'' in \cite[Sect. 5]{TTmethod}.
\end{remark}

\begin{remark}
\label{r:meridian-scaling-geom-realz}
In \rmkref{r:meridian-scaling}, we discussed
the mild modification of the TT method by changing the normalizations,
where the chosen meridian $μ_i$ on each torus $T_i$ has edge label
$λ_i ∈ \CC^\times$.
Let $Ω'$ be an algebraic solution to this modified TT method;
we know that it is a simple rescaling of an algebraic solution $Ω$
to the original TT method (i.e. $μ_i$ are normalized to 1).
We may apply the methods above to produce a geometric realization $z'$ of $Ω'$.
Let $x_i ∈ \hat{M}$ be the ideal point corresponding to $T_i$,
and let $\wdtld{x_i} ∈ \hattld{M}$ be a lift of $x_i$.
By construction, two geodesics emanating from $z'(\wdtld{x_i})$
that are separated by $μ_i$ would have a geometric edge label of $λ_i$ between them.
There is a unique parametrization of $z'(\wdtld{x_i})$ where
this geometric edge label would be 1.
denote by $z''(\wdtld{x_i})$ the parametrized point $z'(\wdtld{x_i})$
with this parametrization.
It is not hard to see that $z''(-)$ is in fact a geometric realization of $Ω$.

This is a similar to the observation about the effect on the geometric labeling
when we change normalizations, in that it amounts to changing parametrizations.
The upshot is that the geometry, which in this context of geometric realizations
refers to the relative position of the $z(\wdtld{x})$'s,
is not affected by changing normalizations.
\end{remark}

Now we describe the converse, i.e. extracting $Ω$ and $z$ from a representation.
Let $M$ be a 3-manifold with basepoint $m₀$
and choices of meridians $μ_i$ on each peripheral torus of $M$
(recall that the meridians are needed in the TT method for normalization).
Suppose we are given a parabolic representation $ρ: π₁(M,m₀) \to \PSL$.
For an ideal point $\wdtld{x} = [α]$ in $\hattld{M}$,
the peripheral subgroup $T_{\wdtld{x}} = \{ [\ov{α}δ\ov{α}^\inv] \} ⊆ π₁(M,m₀)$
is mapped by $ρ$ to some parabolic subgroup.
For the element $h_{μ_i,α} := [\ov{α}μ_i\ov{α}^\inv]$,
if $ρ(h_{μ_i,α})$ is not the identity,
then we take $z(\wdtld{x})$ to be the fixed point of $ρ(h_{μ_i,α})$,
and there exists a unique parametrization $(z(\wdtld{x}),H,ψ)$ such that
$ρ(h_{μ_i,α})$ restricts to a translation on $H$ that moves everything
by distance of 1, and under $ψ$ this translation is just the operation of adding 1.
Thus, $ρ(h_{μ_i,α})$ is conjugate to $(1 \; 1; 0 \; 1)$ under an isometry
sending the parametrized point $(∞,H₀,ψ₀)$ to $(z(\wdtld{x}),H,ψ)$.

\begin{lemma}
\label{l:geom-realzn-conj}
Let $M$ be a 3-manifold with basepoint $m₀$
and choices of meridians $μ_i$ on each peripheral torus,
and suppose we are given a parabolic representation $ρ: π₁(M,m₀) \to \PSL$
satisfying the non-degeneracy condition $ρ(h_{μ_i,α}) ≠ \id$ as above.
Then assignment $z$ constructed above is $π₁(M,m₀)$-equivariant.
Thus, for an ideal polyhedral decomposition $τ$ of $M$,
if the endpoints of every 1-cell $γ$ of $τ$ have different images under $z$,
then we can define the algebraic solution $Ω = (w(-),u(-))$ such that
$z$ is the geometric realization of $Ω$.
In this case, we have $ρ = ρ_{Ω,z}$.

Moreover, if $ρ'$ is related to $ρ$ by an outer automorphism,
i.e. $ρ'(-) = φρ(-)φ^\inv$ for some isometry $φ$,
then the assignment $z'$ constructed from $ρ'$ is $z' = φ_* ∘ z$,
which is another geometric realization of the same algebraic solution $Ω$.
\end{lemma}
\begin{proof}
Let $g ∈ π₁(M,m₀)$ and let $\wdtld{x} = [α] ∈ \hattld{M}$ be an ideal point.
Recall that $g$ sends it to $g⋅\wdtld{x} = [gα]$.
The construction assigns to $z(g⋅\wdtld{x})$ the fixed point of $ρ(h_{μ_i,gα})$.
Since $h_{μ_i,gα} = [\ov{gα}μ_i\ov{gα}^\inv] = gh_{μ_i,α}g^\inv$,
the fixed point of $ρ(h_{μ_i,gα}) = ρ(g)ρ(h_{μ_i,α})ρ(g)^\inv$
is $ρ(g)(z(\wdtld{x}))$, which establishes $π₁$-covariance.

Under the additional condition of endpoints having different images,
we can set, for each 1-cell $γ$ of $τ$ with endpoints $x,x'$ in $\hat{M}$,
the crossing label $w(γ) = \geomw{\wdtld{γ}_{\HH}}$ for $Ω$,
where $\wdtld{γ}_{\HH}$ is the geodesic from $z(\wdtld{x})$ to $z(\wdtld{x}')$,
and $\wdtld{x}, \wdtld{x}' ∈ \hattld{M}$ are ideal points
connected by a lift $\wdtld{γ}$ of $γ$.
By $π₁$-covariance, $\geomw{\wdtld{γ}_{\HH}}$ does not depend on the choice
of the lift $\wdtld{γ}$.
The edge labels for $Ω$ are similarly taken from the geometric edge labels.
By definition, $z$ is the geometric realization of the labeling $Ω$ so constructed,
and since $z(g⋅\wdtld{x}) = ρ(g)_*(z(\wdtld{x}))$, we have $ρ = ρ_{Ω,z}$.

For the last part, given $ρ'(-) = φρ(-)φ^\inv$,
we assign to $z'(\wdtld{x})$ the fixed point of
$ρ'(h_{μ_i,α}) = φρ(h_{μ_i,α})φ^\inv$, which is $φ(z(\wdtld{x})$.
The rest follows from similar arguments as the above.
\end{proof}

We summarize the results above as follows:
\begin{proposition}
\label{p:alg-soln-repn-class}
Let $M$ be a 3-manifold with basepoint $m₀$
and choices of meridians $μ_i$ on each peripheral torus.
Let $τ$ be an ideal polyhedral decomposition of $M$.
The underlying representation $ρ_{Ω,z}: π₁(M,m₀) \to \PSL$
of a geometric realization $z$ of $Ω$ changes by an (outer) automorphism with $z$,
hence we have a well-defined map
\begin{align*}
\{\textrm{algebraic solution}\} &\to
\{\textrm{conjugacy classes of parabolic representations }
	π₁(M,m₀) \to \PSL \}
\\
Ω &\mapsto [ρ_{Ω,z}] \ .
\end{align*}
Then this map is injective,
and its image consists of representations $ρ$ that satisfy
the following non-degeneracy conditions:
\begin{itemize}
\item $ρ(μ) ≠ \id$ for any element $μ ∈ π₁(M,m₀)$ in the conjugacy class
	represented one of the meridians $μ_i$,
\item the fixed points of the parabolic subgroups $ρ(T_{\wdtld{x}})$
	and $ρ(T_{\wdtld{x}'})$ are distinct for ideal points
	$\wdtld{x},\wdtld{x}' ∈ \hattld{M}$ connected by a lift of a 1-cell of $τ$
	(recall that $T_{\wdtld{x}}$ is the peripheral subgroup fixing $\wdtld{x}$).
\end{itemize}
\end{proposition}

\ \\
\subsection{Algebraic Solutions and Volume}
\label{s:alg-vol}
\ \\ \indent

We establish one of the main results of the paper,
namely that the algebraic solution with the maximum volume is the geometric labeling.
We heavily rely on the results of \cite{repnvol},
which concerns volumes of $\PSL$-representations of the fundamental group
of a 3-manifold, and we translate them into results about algebraic solutions
by the methods discussed in \secref{s:alg-to-geom-reconstruction}.

In this section, we deal with 3-manifolds with toric ends
that are not necessarily hyperbolic,
but nonetheless are ``cusped'', i.e. toric ends come with
a fixed tubular neighborhood structure $T_i \times [0,∞) ↪ M$
that extend to $T_i \times [0,∞] \to \hat{M}$ such that
$T_i \times \{∞\}$ is sent to ideal point $x_i$,
and these neighborhoods are pairwise disjoint.
All results will be independent of this choice.

We first recall several definitions and results from \cite{repnvol}.

\begin{definition}{\cite{dunfield}, \cite[Defn. 2.5]{repnvol}}
Let $M$ be a 3-manifold, with basepoint $m₀$,
and let $\wdtld{M}$ be its universal cover.
Let $ρ : π₁(M,m₀) \to \PSL$ be a representation.
A \emph{pseudo-developing map for $ρ$} is a piecewise smooth map
$D_{ρ} : \wdtld{M} \to \HH$ which is equivariant with respect to
the actions of $π₁(M,m₀)$ on $\wdtld{M}$ via deck transformations
and on $\HH$ via $ρ$,
such that $D_{ρ}$ extends to a continuous map
$D_{ρ} : \hattld{M} \to \ov{\HH}$, sending ideal points to $∂\HH$.
Moreover, each ray $\{t\} \times \{0 ≤ s < ∞\}$ in the end tubular neighborhoods
is sent to a geodesic in $\HH$ such that $s$ is an arc-length parameter.
\end{definition}

\begin{definition}{\cite[Def. 4.1]{repnvol}}
Let $D_{ρ}$ be a pseudo-developing map for $ρ$.
Let $ω$ be the volume form on $\HH$.
Let $D_{ρ}^*ω$ be the pull-back of $ω$ to $\wdtld{M}$.
Since $D_{ρ}$ is $π₁$-equivariant, $D_{ρ}^*ω$ is a lift of some 3-form on $M$,
which we denote also by $D_{ρ}^*ω$.
The \emph{volume of $D_{ρ}$}, denoted $\vol(D_{ρ})$, is the integral
\[
\vol(D_{ρ}) = \int_M D_{ρ}^*ω \ .
\]
\end{definition}

\begin{proposition}{\cite[Thm. 4.9]{repnvol}}
\label{p:pseudo-vol-inv}
The volume of a pseduo-developing map $D_{ρ}$ for $ρ$ only depends on $ρ$.
\end{proposition}

By \cite[Lem. 2.7]{repnvol}, any representation $ρ$ admits a pseudo-developing map.
Thus, we may define the volume of a representation to be
\[
\vol(ρ) := \vol(D_{ρ}) \ .
\]

A weaker version of \prpref{p:pseudo-vol-inv} was proved \cite{dunfield},
which required that the pseudo-developing maps that are being compared
have the same extensions to the ideal points.
The proof of \prpref{p:pseudo-vol-inv} is based on this result:
by modifying one pseudo-developing map at ends
so that it agrees with the other on ideal points,
the volume is changed only by a small amount, which can be made arbitrarily small.

\begin{proposition}{\cite[Thm. 7.2]{repnvol}}
\label{p:repn-max-vol}
Let M be a complete, oriented hyperbolic 3-manifold of finite volume.
Let $Γ ≅ π₁(M)$ be the subgroup of $\PSL$ such that $M = \HH / Γ$.
Let $ρ : Γ \to \PSL$ be a representation.
Then $|\vol(ρ)| ≤ |\vol(M)|$ and
equality holds if and only if $ρ$ is discrete and faithful.
\end{proposition}

Now we want to apply these results to algebraic solutions of the TT method.
We fix an ideal polyhedral decomposition $τ$ of $M$
and choices of meridians $μ_i$ on each peripheral torus.
We have seen in \secref{s:alg-to-geom-reconstruction} that
algebraic solutions $Ω$ are intimately related to representations
of the fundamental group into $\PSL$, namely by \prpref{p:alg-soln-repn-class},
to every $Ω$ is associated a conjugacy class of representations $[ρ_{Ω,z}]$,
which (by \lemref{l:geom-realzn-conj}) is the same as the set of
$ρ_{Ω,z}$ over all geometric realizations $z$ of $Ω$.
By definition, isometries of $\HH$ preserve the volume form on $\HH$,
hence the volume of $ρ_{Ω,z}$ is unchanged under conjugation,
so we may define:

\begin{definition}
\label{d:vol-alg-soln}
The \emph{volume of $Ω$}, denoted $\vol(Ω)$,
is the volume of the underlying representation of a geometric realization $z$ of $Ω$,
i.e. $\vol(Ω) := \vol(ρ_{Ω,z})$.
\end{definition}

We give an alternative method to compute the volume of an algebraic solution,
and show that one can compute it from $Ω$ without using a geometric realization.
The following is adapted from \cite[Def 3.4]{repnvol}:

\begin{proposition}
\label{p:repn-vol-sum}
Let $τ'$ be an ideal triangulation of $M$ (not necessarily a refinement of $τ$).
Let $z$ be a geometric realization of $Ω$.
For each tetrahedron $Δ$ of $τ'$, let $\wdtld{Δ}$ be a lift of $Δ$ to $\wdtld{M}$.
Define the \emph{signed volume $\vol(Ω,Δ)$} to be
the volume of the convex hull of vertices of $z(\wdtld{Δ})$,
which is 0 if the vertices of $z(\wdtld{Δ})$
lie in the boundary of a geodesic plane,
and is negative if $z|_{\wdtld{Δ}}$ is orientation-reversing.
Then
\[
\vol(ρ_{Ω,z}) = \sum_{Δ ∈ τ'} \vol(Ω,Δ)
\]
\end{proposition}
\begin{proof}
The sum on the right-hand side is simply the straight volume \cite[Def 4.3]{repnvol},
and the equality above is exactly the content of \cite[Thm 4.10]{repnvol}.
\end{proof}

Using the methods of \secref{s:extension-adm},
we compute the crossing and edge labels
of the interal and peripheral edges of $\ov{τ}'$ whenever possible;
recall that such internal edges are called $Ω$-admissible.
For a tetrahedron $Δ$ in $τ'$ that has at least one non-$Ω$-admissible edge,
we define its volume to be 0.
For a tetrahedron $Δ$ in $τ'$ whose edges are all $Ω$-admissible,
its shape parameter/modulus along an edge $e$ of $Δ$
is given by $ζ_{Δ,e} := ζ_{γ_{01};γ_{02}γ_{13}} = w(γ_{01}) /
u(\vec{ε}_{γ_{01}γ_{02}}) u(\vec{ε}_{γ_{01}γ_{13}})$,
where $γ_{ij}$ is the edge between vertices $v_i$ and $v_j$ of $Δ$,
$e = γ_{01}$,
and the orientation $[\vec{γ}_{01}, \vec{γ}_{02}, \vec{γ}_{03}]$
agrees with the orientation on $M$
($\vec{γ}_{ij}$ is oriented from $v_i$ to $v_j$).
The shape parameter $ζ_{Δ,e}$ takes values in $\CC \backslash \{0,1\}$.
To any such value $ζ ∈ \CC \backslash \{0,1\}$,
there exists a geometric ideal tetrahedron $Δ_{ζ}$ in $\HH$
such that $ζ$ is the shape parameter/modulus of $Δ_{ζ}$ along one of its edges;
we define $\vol(ζ)$ to be the volume of $Δ_{ζ}$,
which is negative when $ζ$ has negative imaginary part
and is 0 when $ζ$ is real.
Clearly $\vol(ζ_{Δ,e}) = \vol(Ω,Δ)$, hence we have:

\begin{proposition}
\label{p:vol-adm}
Let $τ'$ be an ideal triangulation of $M$ (not necessarily a refinement of $τ$).
Then
\[
\vol(Ω) = \sum_{\substack{Δ ∈ τ' \\ Ω\textrm{-adm}}} \vol(ζ_{Δ,e})
\]
where the sum is over tetrahedra whose edges are all $Ω$-admissible.
\end{proposition}

\begin{remark}
\prpref{p:vol-adm} implies in particular that the sum on the right
is independent of triangulation $τ'$.
We can also prove it in a more combinatorial manner, which we sketch here.
By \cite{matveev, piergallini},
any two ideal triangulations are related by a sequence of
Pachner 2-3 and 2-0 moves (also known as Matveev-Piergallini-moves),
thus it suffices to prove that the sum is invariant under these moves.
This is straightforward but rather tedious, as there are quite a few cases
to consider, e.g. combining positively and negatively tetrahedra.
We just consider the Pachner 2-0 move,
and leave the 2-3 move to the interested reader.

The Pachner 2-0 move removes a pair of tetrahedra $Δ,Δ'$ that share two faces $F,F'$
and glues up opposite faces. $Δ,Δ'$ share all of their vertices,
and share all except one of their edges, namely the edge that does not meet $F$ nor $F'$,
say $e ⊆ Δ$ and $e' ⊆ Δ'$.
It is not hard to see that if all of the edges are $Ω$-admissible, then $w(e) = w(e')$,
and in fact the shape parameters must be complex conjugates of each other,
$ζ_{Δ',e'} = \ov{ζ_{Δ,e}}$, hence their volumes add up to 0,
and the Pachner 2-0 move does not change the sum.
We also need to consider the cases when some edge is not $Ω$-admissible.
If an edge besides $e$ and $e'$ is not $Ω$-admissble,
then both tetrahedra have volume 0.
If $e$ is not $Ω$-admissible, then so is $e'$, and vice versa,
so their volumes are both 0.
\end{remark}

\begin{remark}
Triangulations with non-positive volume tetrahedra
(shape parameter real or negative imaginary value)
was first considered in \cite{negative-triangulation}.
\end{remark}

Finally, applying \prpref{p:repn-max-vol} to volumes of algebraic solutions
with \defref{d:vol-alg-soln}, we have:

\begin{theorem}
\label{t:max-vol}
Let $M$ be a cusped hyperbolic 3-manifold,
and let $τ$ be an ideal polyhedral decomposition of $M$
such that every edge is geodesic-like.
Fix meridians $μ_i$ for each end of $M$,
determining a canonical parametrization of $M$.
Then the algebraic solution with the maximal volume is the geometric labeling:
\[
\geomlabel = \underset{Ω: \text{ alg.\,sol.}}{\argmax} \  \vol(Ω)
\]
\end{theorem}

\section{Application to Links}
\label{s:application-links}

We apply the general method, as laid out in \secref{s:TT-3mfld},
to 3-manifolds arising as complements of links.
We first recover the original TT method of \cite{TTmethod} for links in $S³$,
then consider links in the thickened torus $\TT$.
We also consider FALs in $S³$ and $\TT$ separately from general links,
as their complements admit certain polyhedral decompositions (see \secref{s:FAL})
for which the TT method is simpler and more amenable to symmetry arguments.
We note that the case of FALs in $S³$ has already been considered in \cite{rochyTT}
(see more in \secref{s:TT-FAL-S3}).

\subsection{Links in $S³$}
\label{s:original-TT}
\ \\ \indent
We observe that the original TT method, as laid out in \secref{s:recap},
is obtained from the above generalization
by appying it to the Menasco decomposition.
Recall that the Menasco decomposition of the complement $S³ \backslash L$
of a link $L$ with link diagram $D$
is an ideal polyhedral decomposition of $S³\backslash L$
with exactly two 3-cells, one 2-cell for each region of $D$,
and one 1-cell for each crossing.
In the truncation of the Menasco decomposition,
the internal 1-cells of correspond to the crossing arcs
while the boundary 1-cells correspond to the peripheral edges,
as suggested by the names of the labels,
and the vertices correspond to peripheral endpoints.
It is easy to see that the edge equations and normalization
are equivalent, and the region equations are actually the same
after taking the inverse on both sides.

\begin{definition}{\cite[Def. 1.1]{TTmethod}}
\label{d:taut}
A diagram of a hyperbolic link is \emph{taut}
if each associated checkerboard surface is
incompressible and boundary incompressible in the link complement,
and moreover does not contain any simple closed curve
representing an accidental parabolic.
\end{definition}


The following proposition is just a reformulation of the argument
after Definition 1.1 in \cite{TTmethod}:

\begin{proposition}{\cite{TTmethod}}
\label{p:taut-menasco}
Let $τ$ be the Menasco decomposition associated to a link diagram $D$
of a hyperbolic lik $L$.
If $D$ is taut, then all 1-cells of $τ$ are geodesic-like.
\end{proposition}

In particular, in light of \prpref{p:geom-sol-exist},
this means that the TT method on a taut link diagram
always admits the geometric labeling as a solution.

\begin{remark}
\label{r:non-taut}
The TT method can be applied to non-taut link diagrams,
i.e. there are hypberolic links that admit a non-taut link diagram
such that every crossing arc is geodesic-like.
For example, consider a reduced alternating link diagram $D$
of a hyperbolic link $L$, which is taut by \cite[Prop. 1.2]{TTmethod}.
Consider a region $R$ of $D$ and two segments $a,a'$ of $L$ in its boundary.
Let $γ$ be an arc connecting $a$ to $a'$;
since $D$ is taut, by \cite[Prop. 2.1]{TTmethod}, $γ$ is geodesic-like.
Let $D'$ be the link diagram obtained by performing two Reidemeister II moves,
bringing $a'$ over $a$ as in \figref{f:non-taut},
and let $R'$ be the new 4-sided region formed, bounded by $a$ and $a'$.
Then the four crossing arcs around $R'$ (connecting the top and bottom segments
at the four vertices of $R'$) are homotopic to $γ$,
and in particular are all geodesic-like.
Thus, the TT method works on $D'$, since every crossing arc is geodesic-like.
However, $D'$ is not taut, as the checkerboard surface not containing $R'$
admits a compression disk (essentially given by $R'$).
\end{remark}

\begin{figure}[ht]
\includegraphics[height=5cm]{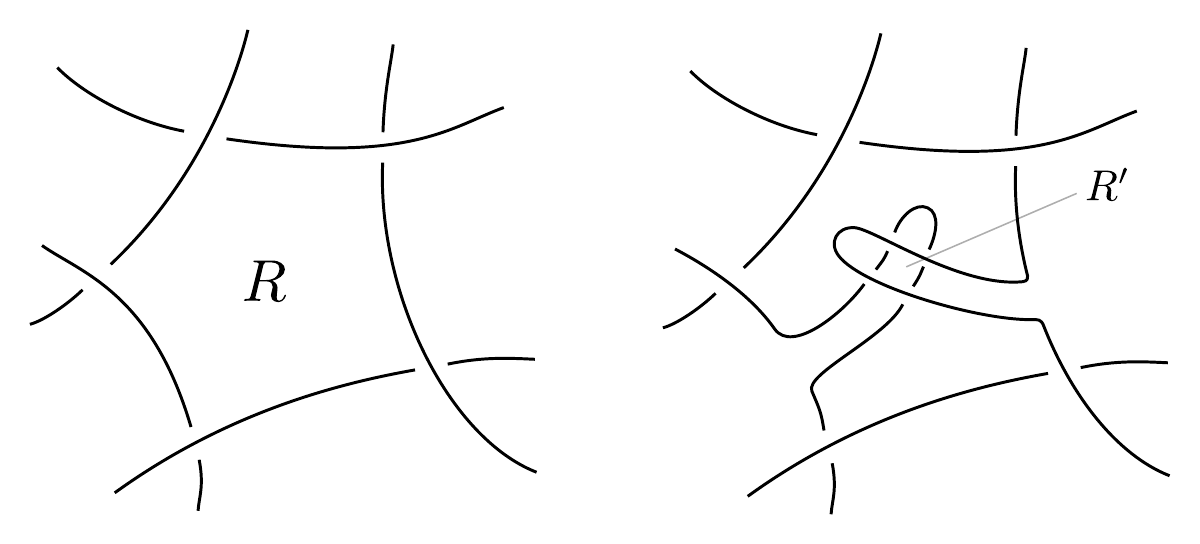}
\caption{Modifying a taut diagram $D$ (left) into a non-taut diagram $D'$ (right).
}
\label{f:non-taut}
\end{figure}


\ \\
\subsection{Links in $T² \times (-1,1)$}
\label{s:TT-thickened-torus}
\ \\ \indent
We apply the generalized TT method of \secref{s:TT-general}
to an ideal polyhedral decomposition of $\TT - L$
(recall that $\TT = T² \times (-1,1)$)
that closely related to the Menasco decomposition (see \secref{s:original-TT}),
which we describe below.
Note that one could identify $\TT$ with the complement of the Hopf link $H$ in $S³$,
then apply the original TT method directly to $H ∪ L$,
but this is somewhat unnatural as it requires an additional choice
of a link diagram for $H ∪ L$.

Let $L$ be a link in $\TT$ with a cellular link diagram $D$ in $T²$.
Let $\north,\south$ denote the top and bottom ideal vertices of $\TT$,
i.e., $\north = T² \times 1$, $\south = T² \times -1$.
Choose simple closed curves $μ_N,μ_S$ in the peripheral tori
around $\north, \south$ respectively to act as meridians.
For each crossing $c$ of $D$, let $γ_c$ be the crossing arc at $c$.
For a region $R$ of $D$, with boundary consisting of
edges and vertices $e₁,c₁,\ldots,e_k,c_k$,
let $F_R$ be the ideal polygon embedded in $\TT - L$ with sides
$γ_{c₁},\ldots,γ_{c_k}$ and ideal vertices at segments $e₁,\ldots,e_k$ of $L$,
such that the interior of $F_R$ projects homeomorphically onto $R$
under the link projection.
We call $F_R$ a \emph{regional face}.

The union of all the $F_R$'s separates $\TT$ into two pieces.
Denote the piece connected to $\north,\south$ by $Q_N, Q_S$ respectively.
In relation to the torihedral decomposition,
$Q_N$ and $Q_S$ are obtained from the two ideal torihedra by identifying
pairs of edges corresponding to the same crossing.
At this point, the procedure is the same as with the Menasco decomposition.
However, since $Q_N$ and $Q_S$ are not 3-balls,
we need to decompose them further, wherein the cellularity of $D$ is needed.

An \emph{overpass segment}
refers to a maximal segment of the link
(or equivalently, a maximal contiguous sequence of edges of $D$)
that lies between two underpasses,
and similarly for an \emph{underpass segment}.
For overpass (resp. underpass) segments $a₁,a₂$,
we say that $a₁$ \emph{runs over (resp. under)} $a₂$
if $a₁$ goes over (resp. under) a crossing which is an endpoint of $a₂$;
we say that they are \emph{adjacent} if one runs over the other.

For each overpass segment $a$,
let $γ_a$ be a crossing arc
\footnote{
We refer to these arcs $γ_a$ as crossing arcs
even though they are not indexed by crossings,
because they play a similar role to $γ_c$'s,
and can in fact be thought of as crossing arcs
when $L$ is considered as a link in
the complement of the Hopf link, $S³ \backslash H$.
}
that travels from $a$ to $\north$.
Likewise, for each underpass segment $b$,
let $γ_b$ be a crossing arc that travels from $b$ to $\south$.
For each pair of adjacent overpass segments $a₁,a₂$,
with $a₁$ running over $a₂$ at the crossing $c$,
let $F_{a₁,a₂,c}$ be an embedded ideal triangle in $\TT$
with sides $γ_{a₁},γ_c,γ_{a₂}$
and ideal vertices at $a₁,a₂,\north$
(see \figref{f:square-weave}).
Similarly construct $F_{b₁,b₂,c}$ for underpass segments
$b₁,b₂$, with $b₁$ running under $b₂$ at crossing $c$.
We call $F_{a₁,a₂,c}$ and $F_{b₁,b₂,c}$ \emph{vertical faces}.

For each region $R$ of $D$,
the union of vertical faces $F_{a₁,a₂,c}$ for $a₁,a₂,c$ adjacent to $R$,
together with $F_R$,
cut out a volume $V_R^N$ above $R$.
Since $D$ is cellular, $R$ is a disk, so the interior of $V_R^N$ is a ball.
More specifically, $V_R^N$ is an ideal pyramid over $R$,
with vertical edges given by the crossing arcs $γ_a$ associated to
overpass segments.
Similarly, we get a volume $V_R^S$ below $R$ (connected to $\south$).

We thus have an ideal polyhedral decomposition $τ$ of $\TT - L$,
where the 3-cells are the volumes $V_R,V_R'$,
the 2-cells are regional faces $F_R$
and vertical faces $F_{a₁,a₂,c}$,$F_{b₁,b₂,c}$,
and the 1-cells are crossing arcs $γ_c,γ_a,γ_b$.

To describe the truncation $\ov{τ}$, we need to introduce more notation.
For crossing arcs $γ_c$ corresponding to crossings,
we refer to their peripheral endpoints by $p_c^o$ (near the overpass)
and $p_c^u$ (near the underpass).
For crossing arcs $γ_a$ corresponding to an overpass segment $a$,
we have one peripheral endpoint $p_N^a$ near $\north$
and another one $p_a$ near $L$.
Likewise, for underpass segment $b$, $γ_b$ has peripheral endpoint
$p_S^b$ near $\south$ and another one $p_b$ near $L$.
We have peripheral edges $\vec{ε}_{\vec{e},R}$
which come from the intersection of a regional face $F_R$
with the peripheral torus near an edge $\vec{e}$ of the link diagram,
similar to the case of a truncated Menasco decomposition.
For overpass segment $a₁$ going over $a₂$ at crossing $c$,
the vertical face $F_{a₁,a₂,c}$ gives rise to three new types of
peripheral edges:
\begin{itemize}
\item $\vec{ε}_{a₁,c}^*$ connecting $p_{a₁}$ to $p_c^o$,
\item $\vec{ε}_{a₂,c}^*$ connecting $p_{a₂}$ to $p_c^u$, and
\item $\vec{ε}_{a₁,a₂}^*$ connecting $p_N^{a₁}$ to $p_N^{a₂}$ at $\north$.
\end{itemize}
The superscript $*$ is a placeholder for a disambiguating symbol,
which we address later.
Similarly for underpass segment $b₁$ going under $b₂$ at crossing $c$,
the vertical face $F_{b₁,b₂,c}$ gives rise to three new types of
peripheral edges:
\begin{itemize}
\item $\vec{ε}_{b₁,c}^*$ connecting $p_{b₁}$ to $p_c^u$,
\item $\vec{ε}_{b₂,c}^*$ connecting $p_{b₂}$ to $p_c^o$, and
\item $\vec{ε}_{b₁,b₂}^*$ connecting $p_S^{b₁}$ to $p_S^{b₂}$ at $\south$.
\end{itemize}

Note that there is some ambiguity in the notation $\vec{ε}_{a₁,c}^*$,
as $a₁$ and $c$ are implicated in two vertical faces.
We may choose an orientation on $a₁$
(say from some orientation on $L$ chosen beforehand),
and add an additional superscript,
$\vec{ε}_{\vec{a}₁,c}^l$ or $\vec{ε}_{\vec{a}₁,c}^r$, to indicate
whether the relevant vertical face is to the left or right of $\vec{a}₁$.
In some cases, $a₂$ could loop around and meet $a₁$ at $c$ from either side;
we may write $\vec{ε}_{\vec{a}₂,c}^s$ or $\vec{ε}_{\vec{a}₂,c}^t$
to indicate whether the relevant vertical face is near the
source or target of $\vec{a}₂$ respectively.
For $\vec{ε}_{a₁,a₂}^*$, we may use an orientation on either $a₁$ or $a₂$
with an appropriate superscript for disambiguation.
Of course, the above discussion applies to the underpass counterparts.

The peripheral endpoints $p_N^a$ and
peripheral edges $\vec{ε}_{a₁,a₂}$
on the peripheral torus at $\north$ make up a graph,
which we refer to as the \emph{overpass graph};
likewise, the \emph{underpass graph} is the graph consisting of
vertices $p_S^a$ and edges $ε_{a₁,a₂}$ at $\south$.
The regions of these graphs are in bijection with the regions of $D$.

Note that the overpass graph can be obtained from the link diagram $D$
of $L$ by essentially a quotient operation, as follows.
For each overpass segment $a$ that does not go over any crossings,
we add a vertex to the middle of the edge $e$ of $D$
which $a$ naturally corresponds to, splitting $e$ into two edges.
For each overpass segment $a$ that goes over at least one crossing,
we identify those crossings that $a$ goes over into one vertex
(and remove the edges between those crossings).
The resulting graph has vertices in bijection with overpass segments,
and the degree of a vertex is $2(n_a + 1)$,
where $n_a$ is the number of crossing that $a$ runs over.
It is straightforward to check that
the resulting graph is isomorphic to the graph
$(\{p_a\}, \{ε_{a,a'}\})$ on the peripheral torus of $\north$.
A similar procedure yields the underpass graph.
Note that if $L$ is alternating, then the overpass and underpass graphs
are isomorphic to the link diagram $D$ of $L$.

For a 3-cell $V_R^N$ above a region $R$ of the link diagram $D$,
its truncation gives rise to 2 types of peripheral 2-cells.
The first type comes from truncating the tip of $V_R^N$ (as an ideal pyramid),
which lies on the peripheral torus of $\north$,
which is simply the region of the overpass graph corresponding to $D$.
The second type comes from truncating off a vertex of $V_R^N$ in its base,
which lies on the peripheral torus around an overpass segment $a$
containing an edge $e$ of $R$;
if $c,c'$ are crossings at the ends of $e$,
then we have the peripheral 2-cell bounded by peripheral edges
$ε_{a,c},ε_{a,c'},ε_{cc'}$.

\subsubsection{\textbf{Equations for Thickened Torus Case}}
\label{s:TT-thickened-torus-eqn}
\ \\ \indent
We apply our generalization (\secref{s:TT-general})
to the polyhedral decomposition of $\TT - L$ described above.
In order to make things appear similar to the original TT method,
we use labels $w_*, u_*^*$ instead of $w(γ_*), u(\vec{ε}_*^*)$,
e.g. $u_{\vec{a}₁,c}^l$ instead of $u(\vec{ε}_{\vec{a}₁,c}^l)$.
A \emph{labeling} $Ω = (\{w_-\},\{u_-^-\})$ consists of:
\begin{itemize}
\item a crossing label $w_c$ for each crossing $c$,
\item a crossing label $w_a$ for each overpass segment,
\item a crossing label $w_b$ for each underpass segment,
\item an edge label $u_{\vec{e},R}$ for each oriented edge $e$ of $D$
	adjacent to region $R$ (with $u_{\vec{e},R} = - u_{\cev{e},R}$
	for the same but oppositely oriented edges),
\item two edge labels $u_{\vec{a},c}^l, u_{\vec{a},c}^r$
	for each oriented overpass segment going over a crossing $c$
	(with $u_{\vec{a},c}^r = - u_{\cev{a},c}^l$),
\item two edge labels $u_{\vec{a},c}^s, u_{\vec{a},c}^t$
	for each oriented overpass segment ending at a crossing $c$
	(with $u_{\vec{a},c}^t = - u_{\cev{a},c}^s$),
\item an edge label $u_{a₁,a₂}$ for overpass segments $a₁,a₂$
	where $a₁$ runs over $a₂$
	(we may need to disambiguate with a superscript, as with $\vec{ε}_{a₁,a₂}$),
\item two edge labels $u_{\vec{b},c}^l, u_{\vec{b},c}^r$
	for each oriented underpass segment going under a crossing $c$
	(with $u_{\vec{b},c}^r = - u_{\cev{b},c}^l$),
\item two edge labels $u_{\vec{b},c}^s, u_{\vec{b},c}^t$
	for each oriented underpass segment ending at a crossing $c$
	(with $u_{\vec{b},c}^t = - u_{\cev{b},c}^s$),
\item an edge label $u_{b₁,b₂}$ for underpass segments $b₁,b₂$
	where $b₁$ runs under $b₂$
	(we may need to disambiguate with a superscript, as with $\vec{ε}_{b₁,b₂}$),
\end{itemize}

Fix an orientation on $L$.
We say that a labeling is an \emph{algebraic solution}
if it satisfies the following equations:
\begin{enumerate}
\item \emph{edge equation} for $\vec{e}$:
	if $R$ and $R'$ are the regions of $D$
	to the left and right of $\vec{e}$,
	and $\vec{e}$ agrees with the orientation of $L$,
	then
	\begin{equation}
	\label{e:edge-eqn-torus}
		u_{\vec{e},R'} - u_{\vec{e},R} = \kappa
	\end{equation}
	where $\kappa$ is as in \eqnref{e:edge-eqn-original}.

\item \emph{edge equation} for $\vec{a},R$:
	given an oriented overpass segment $\vec{a}$
	that contains an oriented edge $\vec{e} ∈ ∂R$,
	and $R$ lies to the right of $\vec{e}$,
	if $\vec{e}$ goes from crossing $c$ to $c'$, then
	\begin{equation}
	u_{\vec{a},c}^* + u_{\vec{e},R} = u_{\vec{e},c'}^*
	\end{equation}
	where the superscripts $*$ in $u_{\vec{a},c}^*$ and $u_{\vec{a},c'}^*$
	depend on whether $c$ and $c'$ are endpoints of $\vec{a}$:
	if $c$ is an endpoint of $\vec{a}$, then we put $s$ in the superscript,
	otherwise we put $l$,
	and similarly, if $c'$ is an endpoint of $\vec{a}$,
	then we put $t$ in the superscript, otherwise we put $l$.

\item \emph{normalization}:
  if $\vec{e}₁,\ldots,\vec{e}_k$ is a sequence of peripheral 1-cells around
	$\north$ (resp. $\south$),
	such that their concatenation is homotopic to $μ_N$ (resp. $μ_S$),
	then $Σ u_{e_j} = 1$.

\item \emph{region equation} for $R$:
	if the boundary of $R$ is $\vec{e}₁ \vec{e}₂ \cdots \vec{e}_n$,
	with $\vec{e}_i$ going from crossing $c_{i-1}$ to $c_i$,
	then
	\begin{equation}
	\label{e:region-eqn-torus}
	\matw{w_{c_n}}
	\matu{-u_{\vec{e}_n,R}}
	\;
	\cdots
	\;
	\matw{w_{c₁}}
	\matu{-u_{\vec{e}₁,R}}
	\;
	\sim
	\;
	\matu{0} \; ,
	\end{equation}
\item \emph{vertical region equation} for adjacent overpass segments $a₁,a₂$:
	if $\vec{a}₁$ runs over $\vec{a}₂$ at crossing $c$, then
	\begin{equation}
	w_{a₁} = u_{a₁a₂}^* u_{\vec{a}₁c}^*
	\; ; \;
	w_{a₂} = -u_{a₁a₂}^* u_{\vec{a}₂c}^*
	\; ; \;
	w_c = -u_{\vec{a}₁c}^* u_{\vec{a}₂c}^*
	\end{equation}
\end{enumerate}

\begin{remark}
We can approximately halve the number of edge labels by only considering
oriented edges/segments that agree with the orientation on $L$.
In this case, we must add minus signs to edge labels
in the equations above.
\end{remark}

\begin{remark}
The crossing arcs between the ideal vertices $\north$/$\south$ of the torihedra
and overpass/underpass segments are automatically geodesic-like
since their endpoints are at different ends,
hence in order to know that a geometric labeling exists,
it is sufficient to check whether the crossing arcs $γ_c$ are geodesic-like.
\end{remark}



\ \\
\subsection{Fully Augmented Links in $S³$}
\label{s:TT-FAL-S3}
\ \\ \indent
As noted in the beginning of this section,
the TT method for FALs in $S³$ has already been considered by Flint \cite{rochyTT}.
As far as we can tell, Flint arrives at the adaptation of the TT method to FALs
by observing the properties of the geometric labeling for FAL diagrams
and adding some crossing arcs until the method works.
We arrive at the same TT method adaptation for FALs,
in a manner that is, in our view, more systematic.

Oberve that in a FAL diagram (\defref{def:falinT2}),
if there are no half-twists, then every crossing is between an augmentation circle
and a component of $L$ that is not an augmentation circle.
In particular, every crossing arc is geodesic-like,
and hence the TT method can be applied without modification.
However, the connection between algebraic solutions and circle packings
is not so clear, hence we develop an alternative approach.
We apply the generalized TT method of \secref{s:TT-general}
to the polyhedral decomposition of the complement of FALs in $S³$
described in \defref{d:decomp-S3}, 
thus recovering the adaptation of the TT method to FALs in \cite{rochyTT}.

Let $L$ be a hyperbolic FAL in $S³$.
We need to verify that all 1-cells in the polyhedral decomposition
of $S³ - L$ (from \defref{d:decomp-S3}) are geodesic-like.
Recall that the 1-cells come in sets of three,
one set for each augmentation circle $C_i$.
For a given $C_i$, they cut a spanning (twice-punctured) disk into
two ideal triangles; it follows from standard results about FALs,
in particular the correspondence between FAL geometry and circle packings
(e.g. \prpref{p:geometric-univalence}),
that the ideal triangles lift to ideal triangles in $\HH$
with distinct vertices in $∂\HH$.
In particular, the 1-cells are geodesic-like,
hence we may apply the TT method.

As in \secref{s:TT-thickened-torus}, we would need to consider
the truncated version of this polyhedral decomposition.
Let us recall some notation from \secref{s:FAL}.
Let $D_L$ be a FAL diagram of $L$.
For the sake of uniformity when assigning labels,
if there are half-twists, we demand that the half-twist
appear as in \figref{f:notation-spanning-face} (see caption),
which can be accomplished by either moving the half-twist through
the augmentation circle, or adding a full-twist so that the half-twist
switches chirality, or by performing both.
Let $B_L$ be the bow-tie graph.
Let $\{C_i\}$ be the collection of augmentation circles of $L$,
with spanning (twice-punctured) disks $F_{C_i}$.
Let $\{a_j\}$ be the collection of segments of $L$ demarcated by spanning disks.
The spanning disks $F_{C_i}$ intersect the projection plane in three arcs
$γ_i¹,γ_i⁰,γ_i²$, splitting $F_{C_i}$ into two ideal triangles $F_{C_i}^{±}$.
Arbitrarily orient each $C_i$; then we choose $γ_i¹$ to be the arc
such that $C_i$ is oriented upwards at the endpoint of $γ_i¹$
(see \figref{f:notation-spanning-face}).

The truncation of the spanning face $F_{C_i}^{±}$ has six sides,
with three peripheral edges denoted $e_{i,±}²,e_{i,±}⁰,e_{i,±}¹$
opposite the three crossing arcs $γ_i¹,γ_i⁰,γ_i²$ respectively.
The (non-truncated) regional face $F_R$ has boundary consisting of
a sequence of crossing arcs such that an arc of type $γ_i¹$ (resp. $γ_i²$)
is always followed by $γ_i¹$ (resp. $γ_i¹$) again if $C_i$ has no half-twist
or $γ_i²$ (resp. $γ_i¹$) if $C_i$ has a half-twist.
Truncating $F_R$ adds peripheral edges to the boundary,
one between each pair of consecutive crossing arcs.
There are two types of peripheral edges: those that go between crossing arcs
at the same augmentation circle,
and those that go between crossing arcs at different augmentation circles.
The former type are denoted by $e_i^{μ1}$ or $e_i^{μ2}$,
where $i$ indicates the relevant augmentation circle $C_i$
(see \figref{f:notation-spanning-face}),
and the latter type are denoted by $e_j^l$ or $e_j^r$,
where $j$ indicates the relevant link segment $a_j$.






%

\subsubsection{\textbf{Equations for FALs in $S³$}}
\label{s:TT-FAL-S3-eqn}
\ \\ \indent
Let $L$ be a FAL. Fix an orientation on $L$.
Here we only explicitly gives edge labels to one of the orientations
on a peripheral edge, usually the one that roughly agrees with the
orientation on $L$.
As before, $i$ indexes the augmentation circles
and $j$ indexes the segments on $L$ demarcated by spanning disks of
augmentation circles.

A \emph{labeling} $Ω = (w(-), u(-))$ consists of:
\begin{itemize}
\item three crossing labels $w(γ_i⁰),w(γ_i¹),w(γ_i²)$ for each $i$,
\item six edge labels $u(e_{i,±}⁰),u(e_{i,±}¹),u(e_{i,±}²)$ for each $i$,
\item two edge labels $u(e_i^{μ1}),u(e_i^{μ2})$ for each $i$,
\item two edge labels $u(e_j^l),u(e_j^r)$ for each $j$.
\end{itemize}
We also write $w_i¹ = w(γ_i¹)$, $u(e_j^l) = u_j^l$ etc.,
where the labels correspond to the crossing arc/peripheral edge
with the same super-/subscripts as in \secref{s:TT-thickened-torus-eqn}.

We say that a labeling is an \emph{algebraic solution}
if it satisfies the following equations:

\begin{enumerate}
\item \emph{edge equations for meridians around non-augmentation circle components}:
	\begin{equation}
	\label{e:edge-eqn-meridian-non-aug}
	u_{i,+}¹ + u_{i,-}¹ = 1 = u_{i,+}² + u_{i,-}²
	\end{equation}
\item \emph{edge equations for meridians around $C_i$}:
	if $C_i$ has no half-twist,
	\begin{equation}
	u_i^{μ1} = u_i^{μ2} = 1
	\end{equation}
	and if $C_i$ has a half-twist,
	\begin{equation}
	\label{e:edge-aug-half-twist}
	u_i^{μ1} + u_{i,-}⁰ = 1 = u_i^{μ2} + u_{i,+}⁰
	\end{equation}
\item \emph{edge equations along $a_j$}:
	\begin{equation}
	u_j^l + u_{i,±}^* - u_j^r - u_{i,±}^* = 0
	\end{equation}
	where $* = 1,2$,
\item \emph{region equation for $F_R$}:
	if the boundary of the truncated $F_R$ is $e₁ γ₁\cdots e_n γ_n$,
	where each $γ_l = γ_i^k$ for some $i$ and $k = 0,1,2$
	and each $e_l$ is the peripheral edge between $γ_{l-1}$ and $γ_l$, then
	\begin{equation}
	\label{e:region-eqn-torus}
	\matw{w(γ_n)}
	\matu{-χ_n u(e_n)}
	\;
	\cdots
	\;
	\matw{w(γ₁)}
	\matu{-χ₁ u(e₁)}
	\;
	\sim
	\;
	\matu{0} \; ,
	\end{equation}
	where $χ_l = +1$ if the orientation on $e_l$ agrees with $∂F_R$,
	and $χ_l = -1$ otherwise,
\item \emph{spanning disk region equations for $C_i$}:
	\begin{equation}
	\label{e:spanning-disk-eqn}
	w_i⁰ = - χ_i¹χ_i² u_{i,±}² u_{i,±}¹
	\; ; \;
	w_i¹ = - χ_i¹ u_{i,±}¹ u_{i,±}⁰
	\; ; \;
	w_i² = - χ_i² u_{i,±}⁰ u_{i,±}²
	\ .
	\end{equation}
	where $χ_i¹ = ±1$, with $χ_i¹ = 1$ if $e_{i,+}¹$ is oriented
	from $γ_i⁰$ to $γ_i¹$, and $χ_i¹ = -1$ otherwise,
	and $χ_i² = ±1$, with $χ_i² = 1$ if $e_{i,+}²$ is oriented
	from $γ_i²$ to $γ_i⁰$, and $χ_i² = -1$ otherwise
	(see \figref{f:notation-spanning-face}).
\end{enumerate}

The requirement that $D_L$ is of a standard form near half-twists
as in \figref{f:notation-spanning-face} is used to give
\eqnref{e:edge-aug-half-twist};
other relative configurations between a half-twist and its augmentation circle
would lead to variations of \eqnref{e:edge-aug-half-twist}.


\subsubsection{\textbf{Additional Equations for FALs in $S³$}}
\label{s:TT-FAL-S3-eqn-simplified}
\ \\ \indent
Recall from \rmkref{r:FAL-half-twist-symmetry}
that the complement of a FAL has the symmetry
of a reflection in the projection plane
(with additional full-twists operations if there are half-twists).
This means that the geometric labeling must satisfy additional equations.
In the following we keep the notation from \secref{s:TT-FAL-S3-eqn}.

\begin{lemma}
\label{l:FAL-S3-pure-imag}
Let $L$ be a FAL in $S³$.
The geometric labeling $\geomlabel = (\geomw{}(-), \geomu{}(-))$
satisfies the following equations:
\begin{enumerate}[label=(\arabic*)]
\item \label{i:c1} $\geomu{}(e_{i,+}¹) = \geomu{}(e_{i,-}¹)
			= \geomu{}(e_{i,+}²) = \geomu{}(e_{i,-}²) = 1/2$,
\item \label{i:c1p} $\geomw{}(γ_i¹) = -\frac{χ_i¹}{2} \geomu{}(e_{i,+}⁰)$,
	\; $\geomw{}(γ_i²) = -\frac{χ_i²}{2} \geomu{}(e_{i,+}⁰)$
	\; (where $χ_i¹,χ_i² = ±1$ from \eqnref{e:spanning-disk-eqn}),
\item \label{i:c2} $\geomu{}(e_j^l) = \geomu{}(e_j^r) ∈ i\RR$,
\item \label{i:c3} $\geomw{}(γ_i⁰) ∈ \RR$,
\end{enumerate}
and for augmentation circle $C_i$ with no half-twist,
\begin{enumerate}[label=(\arabic*)]
\setcounter{enumi}{4}
\item \label{i:c5} $\geomu{}(e_{i,+}⁰) = \geomu{}(e_{i,-}⁰) ∈ i\RR$,
\end{enumerate}
while for augmentation circle $C_i$ with half-twist,
with $λ_i = 1 - \geomu{}(e_{i,+}⁰) - \geomu{}(e_{i,-}⁰)$,
\begin{enumerate}[label=(\arabic*')]
\setcounter{enumi}{4}
\item \label{i:c5h} $\geomu{}(e_{i,+}⁰) = \geomu{}(e_{i,-}⁰) ∈ i\sqrt{λ_i}\RR$,
\item \label{i:c6h} $\geomu{}(e_i^{μ1}) = \geomu{}(e_i^{μ2}) ∈ \sqrt{λ_i}\RR$,
\item \label{i:c7h} $|λ_i| = 1$,
\item \label{i:c8h} $\geomu{}(e_i^{μ1}) ≠ 0$.
\end{enumerate}
\end{lemma}
\begin{proof}
The orientation-reversing symmetry from \rmkref{r:FAL-half-twist-symmetry}
gives an isometry from $S³ - L$ to itself
which preserves the polyhedral decomposition; more precisely,
the homeomorphism preserves all internal and peripheral 1-cells,
except the peripheral 1-cells of the form $e_{i,±}^k$,
which get sent to $-e_{i,∓}^k$,
(where the minus sign indicates opposite orientation;
note also that the $±$ subscript switches to $∓$).
The symmetry sends a meridian around a non-augmentation circle component $l$ of $L$
to itself but with opposite orientation,
so $λ_l = -1$ for such a component $l$.
For an augmentation circle $C_i$, the meridian is preserved
if there is no half-twist at $C_i$, i.e. $λ_i = 1$,
while the meridian is sent to ``itself minus a longitude''
if there is a half-twist at $C_i$, i.e.
$λ_i = 1 - \geomu{}(e_{i,+}⁰) - \geomu{}(e_{i,-}⁰)$
(the longitude in question is the curve giving the blackboard framing
in the FAL diagram, i.e. $e_{i,+}⁰e_{i,-}⁰$);
note that since $λ_i$ is the image of 1 under an isometry of $S³ - L$,
we have \ref{i:c7h}.
Thus, by \prpref{p:labels-isom}, we have:
\begin{enumerate}[label=(\alph*)]
\item \label{i:a} $\geomw{}(γ_i⁰) = \ov{\geomw{}(γ_i⁰)}$
	(both endpoints of $γ_i⁰$ are on non-augmentation circle components,
	so the two $λ_l = -1$ cancel out)
\item \label{i:b} $\geomw{}(γ_i¹) = -λ_i \ov{\geomw{}(γ_i¹)}, \;\;
	\geomw{}(γ_i²) = - λ_i \ov{\geomw{}(γ_i²)}$
\item \label{i:c} $\geomu{}(e_{i,±}¹) = \ov{\geomu{}(e_{i,∓}¹)}$, \;\;
	$\geomu{}(e_{i,±}²) = \ov{\geomu{}(e_{i,∓}²)}$
	(the $λ_l = -1$ cancels out the orientation-reversal)
\item \label{i:d} $\geomu{}(e_{i,±}⁰) = -λ_i \ov{\geomu{}(e_{i,∓}⁰)}$
\item \label{i:e} $\geomu{}(e_i^{μ1}) = λ_i \ov{\geomu{}(e_i^{μ1})}$, \;\;
	$\geomu{}(e_i^{μ2}) = λ_i \ov{\geomu{}(e_i^{μ2})}$,
\item \label{i:f} $\geomu{}(e_j^*) = -\ov{\geomu{}(e_j^*)}$ for $* = l,r$,
\end{enumerate}

From \ref{i:a} we immediately obtain \ref{i:c3}.
\ref{i:c}, together with the edge equation \eqnref{e:edge-eqn-meridian-non-aug},
gives \ref{i:c1}.
\ref{i:c1}, together with the spanning disk edge equations
\eqnref{e:spanning-disk-eqn},
gives $\geomw{}(γ_i¹) = -\frac{χ_i¹}{2} \geomu{}(e_{i,±}⁰)$,
and similarly for $\geomw{}(γ_i²)$, which proves \ref{i:c1p}
and the equality part of \ref{i:c5} and \ref{i:c5h},
i.e. that $\geomu{}(e_{i,+}⁰) = \geomu{}(e_{i,-}⁰)$;
from the edge equations \eqnref{e:edge-aug-half-twist},
we have
$\geomu{}(e_i^{μ1}) = 1 - \geomu{}(e_{i,-}⁰)
= 1 - \geomu{}(e_{i,+}⁰) = \geomu{}(e_i^{μ2})$,
which establishes the equality part of \ref{i:c6h}.
\ref{i:c1} implies the equality in \ref{i:c2},
and \ref{i:f} forces them to be the pure imaginary.
Since we have proven \ref{i:c7h}, or equivalently $\ov{λ_i} = λ_i^\inv$,
we may rewrite \ref{i:d} and \ref{i:e} as
\begin{align*}
σ_i^\inv \geomu{}(e_{i,±}⁰) &= -\ov{σ_i^\inv \geomu{}(e_{i,∓}⁰)}
\\
σ_i^\inv \geomu{}(e_i^{μ1}) = \ov{σ_i^\inv \geomu{}(e_i^{μ1})},
&\;\;
σ_i^\inv \geomu{}(e_i^{μ2}) = \ov{σ_i^\inv \geomu{}(e_i^{μ2})}
\end{align*}
where $σ_i = \sqrt{λ_i}$ is a choice of square root.
Thus, \ref{i:e} implies that 
$σ_i^\inv \geomu{}(e_i^{μ1})$ and $σ_i^\inv \geomu{}(e_i^{μ2})$
are real, which establishes the ``$∈$'' part of \ref{i:c6h},
and \ref{i:d}, together with $\geomu{}(e_{i,+}⁰) = \geomu{}(e_{i,-}⁰)$,
implies that $σ_i^\inv \geomu{}(e_{i,+}⁰)$ is pure imaginary,
which establishes the ``$∈$'' part of \ref{i:c5} and \ref{i:c5h}.

The final criterion \ref{i:c8h}, $\geomu{}(e_i^{μ1}) ≠ 0$,
follows from \rmkref{r:lin-indep-crit} applied to $e_i^{μ2}e_i^{μ1}$:
since $e_i^{μ2}e_i^{μ1}$ is a loop that is not homotopic to
a multiple of the chosen meridian (given by $e_i^{μ2}e_{i,+}⁰$) for $C_i$,
we have $2 \geomu{}(e_i^{μ1}) = \geomu{}(e_i^{μ2}) + \geomu{}(e_i^{μ1}) ∈ \RR$,
in particular $≠0$.
This criterion is the only inequality in the list, and might seem out of place.
We include it to address an apparent asymmetry between FALs with and without
half-twists; we discuss this further in \rmkref{r:half-twist-relation}.
\end{proof}

\begin{remark}
\label{r:half-twist-relation}
Suppose a FAL $L$ has some half-twists, and $L'$ is the FAL obtained from $L$
by removing all half-twists.
The algebraic solutions for $L$ and $L'$ are closely related.
Let $Ω = (w(-),u(-))$ be an algebraic solution for $L$
that satisfies the criteria in \lemref{l:FAL-S3-pure-imag}.
For an augmentation circle $C_i$ with a half-twist,
$u(e_i^{μ1}) ≠ 0$ by \ref{i:c8h}.
We can consider the labeling $Ω' = (w'(-),u'(-))$
which has the same labels as $Ω$ except that for
crossing arcs and peripheral edges that meet augmentation circles $C_i$
that have a half-twist, we divide their label by $u(e_i^{μ1})$;
following Remarks \ref{r:meridian-scaling} and \ref{r:meridian-choice},
we can think of $Ω'$ as an algebraic solution for the TT method
with a different chosen meridian $μ_i' = e_i^{μ2}e_i^{μ1}$
and normalization $u'(μ_i') = 2$.

Since the 1-cells of the polyhedral decomposition for the complements of
$L$ and $L'$ are indexed by the same symbols ($e_i^{μ1}, e_{i,+}⁰$ etc.),
$Ω'$ can also be directly interpreted as a labeling for $L'$,
and furthermore is an algebraic solution for $L'$
that satisfies all the criteria in \lemref{l:FAL-S3-pure-imag}.
Conversely, if we started with an algebraic solution $Ω' = (w'(-),u'(-))$ for $L'$
that satisfies all the criteria in \lemref{l:FAL-S3-pure-imag},
then we can consider the labeling $Ω$ for $L$ with the same labels except that for 
crossing arcs and peripheral edges that meet augmentation circles $C_i$
that have a half-twist, we divide their label by $1 - u'(e_{i,+}⁰)$
($≠ 0$ since $u'(e_{i,+}⁰)$ is pure imaginary).
It is clear that these two operations are inverses of each other.
\end{remark}

\begin{lemma}
\label{l:FAL-S3-ortn}
Let $L$ be a FAL in $S³$. The geometric edge labels satisfy:
\begin{itemize}
\item $\geomu{}(e_{i,+}⁰) / \geomu{}(e_i^{μ1}) ∈ i\RR^{<0}$,
\item $\geomu{}(e_j^l) ∈ i\RR^{<0}$.
\end{itemize}
\end{lemma}

\begin{proof}
By \rmkref{r:half-twist-relation}, we may assume that $L$ has no half-twists,
so $\geomu{}(e_i^{μ1}) = 1$ and the first criterion is simply
$\geomu{}(e_{i,+}⁰) ∈ i\RR^{<0}$.
Choose a vertex $v ∈ B_L$; if $v$ corresponds to some $C_i$,
then we are interested in proving $\geomu{}(e_{i,+}⁰) ∈ i\RR^{<0}$,
and if $v$ corresponds to some $a_j$,
then we are interested in proving $\geomu{}(e_j^l) ∈ i\RR^{<0}$.

Let $P$ be the top polyhedron in the polyhedral decomposition of $S³ - L$.
As discussed in \rmkref{r:FAL-triangulation-geometric},
we can refine the polyhedral decomposition
by choosing a triangulation for each region
(of the bow-tie graph $B_L$), choosing a vertex $c ≠ v$ of $B_L$,
then subdividing $P$ (and the bottom polyhedra)
into tetrahedra consisting of cones from $c$ over triangles not adjacent to $c$.
This triangulation is geometric,
and in particular the tetrahedra in the triangulation have positive volume,
that is, $\vol(ζ_{Δ,e}) = \vol(Ω,Δ) > 0$
(see \prpref{p:repn-vol-sum} and the discussion thereafter),
or equivalently, the shape parameters $ζ_{Δ,e}$ have positive imaginary part.

For a vertex $v$ of $B_L$, let $t₁,...,t_k$ be the triangles meeting $v$
in counterclockwise order.
Let $x₁,...,x_k$ be the neighbors of $v$, with $γ_i = \ov{vx_i}$,
so that $t_i$ has vertices $v, x_i, x_{i+1}$.
Let $l_i = \vec{x_ix_{i+1}}$ be the edge of $t_i$ opposite from $v$,
oriented so that $l₁⋅⋅⋅l_k$ goes counterclockwise around $v$.
Let $T₁,...,T_k$ be the tetrahedra
that are the cones from $c$ over $t₁,...,t_k$ respectively.
Note that if the cone point $c$ is adjacent to $v$,
then all or some of the $T_j$ are degenerate.

Consider the geometric realization $z(-)$ of $\geomlabel$ restricted to $P$,
as in \rmkref{r:geom-realzn-basepoint}.
By applying a global isometry if necessary, we assume that $z(v) = ∞$,
and that the parametrization of $z(v)$ is given by
the horosphere $H₀$ (the horizontal plane at height 1)
with parametrization $ψ₀ : H₀ ≃ \CC$ given by projection.
Let $y_i$ be the intersection between $H₀$ and the geodesic from $z(v)$ to $z(x_i)$.
Let $ε_i = \vec{y_iy_{i+1}}$, the intersection between $H₀$ and ``$z(t_i)$''
(that is, the geodesic ideal triangle spanning $z(v),z(x_i),z(x_{i+1})$).
We may assume that the geodesic between $z(v)$ and $z(c)$
intersects $H₀$ at $0$, again by applying a global isometry if necessary.
Let $\ov{T_i}$ be the intersection of $H₀$ with ``$z(T_i)$''; equivalently
$\ov{T_i}$ is the triangle in $H₀$ with vertices $0, y_i, y_{i+1}$.

In the graph $B_L$ before triangulation, $v$ had 4 edges, say $E₁,E₂,E₃,E₄$
in counterclockwise order.
Let $e_{(j)}$ be the peripheral edge from $E_j$ to $E_{j+1}$.
If $v$ corresponds to an augmentation circle $C_i$,
we relabel the $E_j$'s (if necessary) so that $ε = e_i^{μ1}$;
then $e_{(3)} = -e_i^{μ2}$ and $e_{(2)} = e_{i,+}⁰ = -e_{(4)}$.
If $v$ corresponds to a segment $a_j$ going from $C_i$ to $C_{i'}$,
then we relabel so that $e_{(1)} = e_{i,+}^l$, $l = 1$ or 2;
then $e_{(3)} = e_{i',+}^{l'}$, $l' = 1$ or 2,
and $e_{(2)} = e_j^r, e_{(4)} = -e_j^l$.

Let $a,b$ be the indices such that $ε_a = ε, ε_b = ε'$;
equivalently, $E₁ = γ_a, E₃ = γ_b$.
Then in $H₀$ (or more accurately under $ψ₀$),
$y_{a+1} - y_a = y_b - y_{b+1} = \geomu{}(e_i^{μ1}) = 1$ and
$y_b - y_{a+1} = y_{b+1} - y_a = \geomu{}(e_{i,+}⁰)$
in the $v = C_i$ case, while
$y_{a+1} - y_a = y_b - y_{b+1} = \geomu{}(e_{i,+}^l) = 1/2$ and
$y_b - y_{a+1} = y_{b+1} - y_a = \geomu{}(e_j^l)$
in the $v = a_j$ case.
In both cases, $y_a,y_{a+1},y_b,y_{b+1}$ are the vertices of a rectangle.
Then the lemma is equivalent to proving that this rectanlge is non-degenerate,
and that these 4 vertices, in that order, go countercounterwise around it.

Since the tetrahedra $T_a$ has positive volume,
if $\ov{T_a}$ is non-degenerate (i.e. $c ≠ x_a,x_{a+1}$),
then $0,y_a,y_{a+1}$ go counterclockwise around $\ov{T_a}$.
Likewise for $\ov{T_b}$ and $0,y_b,y_{b+1}$.
In every case, 0 lies to the left of $\overset{\longrightarrow}{y_a y_{a+1}}$
and $\overset{\longrightarrow}{y_b y_{b+1}}$,
which shows that $y_a,y_{a+1},y_b,y_{b+1}$ go counterclockwise around
the rectangle they enclose (if the rectangle is non-degenerate).
Moreover, $\ov{T_a}$ and $\ov{T_b}$ cannot both be degenerate,
so 0 is strictly to the left of one of them.
This proves that the rectangle is non-degenerate, so we are done.
\end{proof}

\begin{lemma}
\label{l:FAL-S3-shape-param}
Let $L$ be a FAL in $S³$.
Let $v₀,v₁,v₂,v₃$ be distinct (ideal) vertices of $F_R$ in cyclic order
(but not necessarily consecutive vertices of $F_R$).
Let $γ',γ,γ''$ be diagonals or edges of $F_R$,
with $∂γ' = \{v₀,v₁\},∂γ = \{v₁,v₂\},∂γ'' = \{v₂,v₃\}$.
Then the shape parameter
$ζ_{γ;γ',γ''} = \geomw{}(γ) / \geomu{ε_{γ γ'}} \geomu{ε_{γ γ''}}$
is real and $0 < ζ_{γ;γ',γ''} < 1$.
\end{lemma}

\begin{proof}
From the top-bottom symmetry,
it follows that $F_R$ can be isotoped to lie in a geodesic plane,
and in particular, that the vertices of $F_R$
lie on a circle in the sphere at infinity
(this can also be shown to be a direct consequence
of the criteria in \lemref{l:FAL-S3-pure-imag},
see \prpref{p:crit-01-circ-packing}).
This implies that the shape parameter must be real.
Moreover, the polygon formed by $R$ must be a convex polygon.
Thus, if $v₀,v₁,v₂$ are placed that $1,∞,0$,
then $v₃$ is placed at the shape parameter $ζ$,
and then convexity of the polygon formed by $F_R$ implies that $0 < ζ < 1$.
\end{proof}

\begin{remark}
The three lemmas above describe properties of the geometric labeling.
The main reason we separate them out into three distinct lemmas
is that they play different roles in \secref{s:FAL-circ-packing}, 
where we discuss the relation of algebraic solutions to circle packings.
More precisely, we treat the properties as criteria on an algebraic solution,
so if an algebraic solution satisfies the criteria in
\lemref{l:FAL-S3-pure-imag}, then we can construct a circle packing,
and if it furthermore satisfies the criteria in Lemma
\ref{l:FAL-S3-ortn} or \ref{l:FAL-S3-shape-param},
then the corresponding circle packing satisfies additional properties.

%
Now the criteria in \lemref{l:FAL-S3-shape-param}
do not seem applicable to an arbitrary algebraic solution,
since the polyhedral decomposition does not necessarily have
the relevant 1-cells ($γ,ε_{γ γ'}$, and $ε_{γ γ''}$).
We can use the extension discussed in \secref{s:extension-adm},
which can be made quite concrete in this situation.
For example, if there is only one vertex $v$ between $v₁$ and $v₂$,
then we can use the triangle region equations \eqnref{e:region-eqn-k3}
to obtain $w(γ)$. More generally, we triangulate $F_R$
by adding diagonals, among which are $γ,γ',γ''$,
and repeatedly apply \eqnref{e:region-eqn-k3} to get the desired labels.
\end{remark}


\ \\
\subsection{FALs in $\TT$}
\label{s:TT-FAL-T2}
\ \\ \indent
Similar to \secref{s:TT-thickened-torus},
the TT method for FALs in $\TT$ is essentially obtained from
adding ``vertical labels'' to the $S³$ case.
We describe these labels, referring the reader to \secref{s:TT-FAL-S3}
for the ones also found in the $S³$ case.
We label the crossing arcs connecting $C_i$ to $\north$ and $\south$
by $γ_i^N$ and $γ_i^S$ respectively.
We label the crossing arcs connecting $a_j$ to $\north$ and $\south$
by $γ_j^N$ and $γ_j^S$ respectively.
We label the peripheral edges around the vertical face $F_{C_i,a_j}^N$
(the copy of $F_{C_i,a_j}$ in the top toriherdron)
by $ε_{i,j,N}^k$ for $k = 0,1,2$,
with $k = 0$ at $\north$, $k = 1$ at $C_i$, and $k = 2$ at $a_j$;
we orient $ε_{i,j,N}⁰$ from $γ_j^N$ to $γ_i^N$, and orient
$ε_{i,j,N}¹$ and $ε_{i,j,N}²$ following $C_i$ and $a_j$ respectively.
Likewise we label the peripheral edges of $F_{C_i,a_j}^S$ by $ε_{i,j,S}^k$.
We label the peripheral edges around the vertical face
$F_{a_j,a_{j'};C_i}^N$ (the copy of $F_{a_j,a_{j'};C_i}$ in the top toriherdron)
by $ε_{j,j',i,N}^k$ for $k = 0,1,2$,
with $k = 0$ at $\north$, $k = 1$ at $a_j$, and $k = 2$ at $a_j'$;
we orient $ε_{j,j',i,N}⁰$ from $γ_j^N$ to $γ_{j'}^N$,
and orient $ε_{j,j',i,N}¹$, $ε_{j,j',i,N}²$ following $a_j, a_{j'}$ respectively.
Likewise we label the peripheral edges of $F_{a_j,a_{j'};C_i}^S$ by $ε_{j,j',i,S}^k$.
We leave it as an exercise to the reader to produce the appropriate equations
for the TT method based on \secref{s:TT-general}
as we did for the $S³$ case in \secref{s:TT-FAL-S3-eqn}.


\subsubsection{\textbf{Additional Equations for FALs in $\TT$}}
\label{s:TT-FAL-T2-eqn-simplified}
\ \\ \indent

Similar to \secref{s:TT-FAL-T2-eqn-simplified},
the reflection symmetry swapping top and bottom torihedra
(\rmkref{r:FAL-half-twist-symmetry})
implies that the geometric labeling must satisfy additional equations.
As in \secref{s:TT-thickened-torus}, we need to choose meridians
$μ_N, μ_S$ for $\north, \south$, the ideal vertices of $\TT$;
to exploit the symmetry, it is convenient to choose them such that
they are sent to each other under the reflection symmetry.

\begin{lemma}
\label{l:FAL-T2-pure-imag}
Let $L$ be a FAL in $\TT$.
Assume that the meridians $μ_N,μ_S$ of the ideal vertices
are sent to each other under the reflection symmetry.
Then the geometric labeling $\geomlabel = (\geomw{}(-), \geomu{}(-))$
satisfies the equations in \lemref{l:FAL-S3-pure-imag},
as well as the following additional equations:
\begin{enumerate}[label=(\arabic*)]
\setcounter{enumi}{8}
\item \label{i:c9-0}
	$\geomu{}(ε_{j,j',i;S}⁰) = \ov{\geomu{}(ε_{j,j',i;N}⁰)}$,
\item \label{i:c9-12}
	$\geomu{}(ε_{j,j',i;S}^k) = -\ov{\geomu{}(ε_{j,j',i;N}^k)}$ for $k = 1,2$,
\item \label{i:c10-0}
	$\geomu{}(ε_{i,j,S}⁰) = \ov{\geomu{}(ε_{i,j,N}⁰)}$
\item \label{i:c10-2}
	$\geomu{}(ε_{i,j,S}²) = -\ov{\geomu{}(ε_{i,j,N}²)}$
\end{enumerate}
and for augmentation circle $C_i$ with no half-twist,
\begin{enumerate}[label=(\arabic*)]
\setcounter{enumi}{10}
\item \label{i:c11}
	$\geomu{}(ε_{i,j,S}¹) = \ov{\geomu{}(ε_{i,j,N}¹)}$,
\end{enumerate}
while for augmentation circle $C_i$ with half-twist,
with $λ_i = 1 - \geomu{}(e_{i,+}⁰) - \geomu{}(e_{i,-}⁰)$,
\begin{enumerate}[label=(\arabic*')]
\setcounter{enumi}{10}
\item \label{i:c11h}
	$\geomu{}(ε_{i,j,S}¹) = λ_i \ov{\geomu{}(ε_{i,j,S}¹)}$.
\end{enumerate}
\end{lemma}
\begin{proof}
The proof of \lemref{l:FAL-S3-pure-imag} works verbatim here,
so $\geomlabel$ satisfies those equations.
In each of the above equations, the peripheral edges on the left-hand side
is the image of the right-hand side under the reflection symmetry,
and the equations are direct consequences of \prpref{p:labels-isom},
each with different $λ_i$ (here we mean $λ_i$ is as in \prpref{p:labels-isom},
the image of the meridian under the symmetry).
For the equations involving peripheral edges around $\north$ and $\south$,
i.e. \eqnref{i:c9-0} and \eqnref{i:c10-0},
since we chose meridians that also map to each other, $λ_i = 1$.
The meridians around $a_j$'s get reversed in orientation,
so for \eqnref{i:c9-12} and \eqnref{i:c10-2}, we have $λ_i = -1$.
Finally the meridian around $C_i$ is fixed in the case with no half-twist,
i.e. $λ_i = 1$, while in the half-twist case,
the meridian is sent to $1 - \geomu{}(e_{i,+}⁰) - \geomu{}(e_{i,-}⁰)$
(i.e. the $λ_i$ as defined in the lemma).
\end{proof}

\begin{lemma}
\label{l:FAL-T2-ortn}
Let $L$ be a FAL in $\TT$,
with the same setup as \lemref{l:FAL-T2-pure-imag}. Then
\begin{itemize}
\item $\geomu{}(e_{i,+}⁰) / \geomu{}(e_i^{μ1}) ∈ i\RR^{<0}$,
\item $\geomu{}(e_j^l) ∈ i\RR^{<0}$.
\end{itemize}
\end{lemma}

\begin{proof}
Similar to the proof of \lemref{l:FAL-S3-ortn},
we use the fact that the conical triangulation of the torihedra
(based on any triangulation of the graphs) is geometric
(see \rmkref{r:FAL-triangulation-geometric}).
\end{proof}

\begin{lemma}
\label{l:FAL-T2-shape-param}
Let $L$ be a FAL in $\TT$,
with the same setup as \lemref{l:FAL-T2-pure-imag}.
Let $v₀,v₁,v₂,v₃$ be distinct (ideal) vertices of $F_R$ in cyclic order
(but not necessarily consecutive vertices of $F_R$).
Let $γ',γ,γ''$ be diagonals or edges of $F_R$,
with $∂γ' = \{v₀,v₁\},∂γ = \{v₁,v₂\},∂γ'' = \{v₂,v₃\}$.
Then the shape parameter
$ζ_{γ;γ',γ''} = \geomw{}(γ) / \geomu{ε_{γ γ'}} \geomu{ε_{γ γ''}}$
is real and $0 < ζ_{γ;γ',γ''} < 1$.
\end{lemma}

\begin{proof}
Exactly the same as \lemref{l:FAL-S3-shape-param}.
\end{proof}


\begin{remark}
\label{r:cusp-shapes}
Given a geometric solution $\Omega = (w(-),u(-))$ to the TT method,
it is easy to obtain the cusp shapes of the peripheral tori.
For a peripheral torus $T$, which has a chosen meridian,
consider a sequence of peripheral edges $ε₁,\ldots,ε_k$
whose concatenation forms a closed loop
that is homotopic to a longitude of the peripheral torus.
Then the sum $\sum_{i=1}^k u(ε_i)$ is the cusp shape of $T$.


\end{remark}

\ \\
\subsection{Algebraic Solutions for FALs and Circle Packings}
\label{s:FAL-circ-packing}
\ \\ \indent
As we saw in \secref{s:FAL-circ}, FALs and circle packings are closely related.
More precisely, by \prpref{p:geometric-univalence},
every FAL $L$ in $S³$ (resp. $\TT$) is associated with a univalent circle packing
realizing its region graph $Γ_L$ in $S²$ (resp. $T²$)
that is unique up to conformal maps.
When we drop the univalence condition, we obtain possibly more circle packings
realizing $Γ_L$.
In fact, we will see that they are in 1-to-1 correspondence
with the set of algebraic solutions to the TT method on $L$
that satisfy the additional criteria of \lemref{l:FAL-S3-pure-imag}.
Using this correspondence, we prove the other main theorem of this paper,
which says that the criteria in the other two lemmas,
Lemmas \ref{l:FAL-S3-ortn} and \ref{l:FAL-S3-shape-param},
in addition to those of \lemref{l:FAL-S3-pure-imag},
are sufficient to ensure that the algebraic solution is the geometric labeling.

Let us collect some notation from the previous sections.
First consider the case of a FAL $L$ in $S³$.
As in \secref{s:TT-FAL-S3}, we consider the polyhedral decomposition $τ$ of
$M = S³ - L$  from \defref{d:decomp-S3} and apply the TT method to it.
Fix a basepoint $m₀$ for $M$, and let $\wdtld{M}$ be the universal cover of $M$.
Let $\wdht{M},\hattld{M}$ be the completions of $M,\wdtld{M}$ respectively
by adding ideal points.
Recall from \secref{s:alg-to-geom-reconstruction} that a geometric realization
of an algebraic solution $Ω$ is an assignment $z(-)$ of a parametrized point
$z(\wdtld{x}) ∈ ∂\HH$ for each ideal point $\wdtld{x} ∈ \hattld{M}$,
and $z(-)$ is $π₁(M,m₀)$-equivariant.
Any two geometric realizations are related by an isometry of $\HH$.

Let $P$ be the top polyhedron (above the projection plane) of $τ$.
Recall that the graph on the boundary of $P$ is the bow-tie graph $B_L$.
Let $\wdtld{P}$ be a lift of $P$ in $\wdtld{M}$.
By restricting the geometric realization $z(-)$,
we have an assignment of parametrized points to the vertices of $\wdtld{P}$.
From \rmkref{r:geom-realzn-basepoint}, we may refer to this assignment
as a geometric realization of $Ω$ restricted to $P$,
or more appropriately for this context,
a \emph{geometric realization of $Ω$ on $B_L$}.

The case of a FAL $L$ in $\TT$ is almost the same, with $M = \TT - L$,
except instead of $P$, we consider the top torihedron $\cT$
with its canonical polyhedral decomposition $τ$ from \defref{d:torihedra-conical},
and instead of a lift of $\cT$ in $\wdtld{M}$, we consider copy
$\wdtld{\cT} ⊆ \wdtld{M}$ of the universal cover of $\cT$.
Then as above, we restrict $z(-)$ to $\wdtld{\cT}$,
obtaining a geometric realization of $\wdtld{B_L}$ (universal cover of $B_L$),
which is $\ZZ ⊕ \ZZ$-equivariant.
Applying an isometry of $\HH$ if necessary, we may assume that
the vertex of $\wdtld{\cT}$ (technically the lift of the vertex of $\cT$)
is mapped to $∞ ∈ ∂\HH$.
The $\ZZ ⊕ \ZZ$-action acts by translations,
since it preserves the geometric edge labels around $∞$.

\begin{proposition}
\label{p:crit-01-circ-packing}
Let $L$ be a FAL in $S³$ (resp. $\TT$).
Let $Ω$ be an algebraic solution to the TT method on $L$.
If $Ω$ satisfies the criteria in \lemref{l:FAL-S3-pure-imag}
(resp. \lemref{l:FAL-T2-pure-imag}),
then any geometric realization $z(-)$ of $Ω$ on $B_L$ defines a circle packing
in $S²$ (resp. $T²$) realizing the region graph $Γ_L$, that is,
for every regional (i.e. non-bow-tie) face $R$ of $B_L$,
the vertices of $R$ are assigned by $z(-)$ to points on a circle,
and two circles are tangent when their corresponding regions are adjacent.
\end{proposition}

\begin{proof}
\textbf{Case $L ⊆ S³$}:
We first assume that $L$ has no half-twists.
Let $Ω$ be an algebraic solution that satisfies the criteria in
\lemref{l:FAL-S3-pure-imag}, so in particular,
all labels are either real or pure imaginary.
Following \rmkref{r:meridian-scaling}, consider the modified TT method
where the normalization on the ideal points corresponding to segments of $L$
(i.e. not corresponding to augmentation circles)
is ``meridian $= i$'' instead of 1.
Then the algebraic solution $Ω' = (w'(-),u'(-))$ to this modified TT method
obtained from rescaling $Ω$ satisfies:

\begin{enumerate}[label=(\arabic*)]
\item \label{i:cc1} $u'(e_{i,+}¹) = u'(e_{i,-}¹)
			= u'(e_{i,+}²) = u'(e_{i,-}²) = i/2$,
\item \label{i:cc1p} $w'(γ_i¹) = ±\frac{i}{2} u'(e_{i,+}⁰)$,
	\; $w'(γ_i²) = ±\frac{i}{2} u'(e_{i,+}⁰) ∈ \RR$,
\item \label{i:cc2} $u'(e_j^l) = u'(e_j^r) ∈ \RR$,
\item \label{i:cc3} $w'(γ_i⁰) ∈ \RR$,
\item \label{i:cc5} $u'(e_{i,+}⁰) = u'(e_{i,-}⁰) ∈ i\RR$,
\item \label{i:cc6} $u'(e_i^{μ1}) = u'(e_i^{μ2}) = 1$.
\end{enumerate}
In particular, all the (modified) crossing labels are real,
and all the (modified) edge labels for peripheral edges of type $e_i^{μ*},e_j^*$,
i.e. those in the boundary of regional faces, are real as well.
As discussed in \rmkref{r:meridian-scaling-geom-realz},
the geometry, i.e. relative positions of the images of ideal points
under a geometric realization, is not affected by changing normalizations,
so we can draw conclusions directly from $Ω'$ without having to translate back
to $Ω$ first.

Let $x₁,...,x_n$ be the vertices of $R$ in some cyclic order.
The geometric realization of $Ω'$ on $B_L$
assigns parametrized points $z'(x_k) ∈ ∂\HH$, $k = 1,...,n$.
Note that $z'(-)$ is almost the same as a geometric realization $z$ of $Ω$,
except that for vertices $x_k$ corresponding to segments of $L$,
the parametrization of the horosphere is scaled by $i$;
when the parametrization does not matter, we may write $z$ instead of $z'$.
Since the geometric labels among the $z'(x_k)$'s are all real,
they must lie on a circle.
(This is most clear when they lie on the extended real line $\RR ∪ \{∞\} ⊆ ∂\HH$.)

We denote the circle that passes through all the $z(x_k)$'s by $X_R$;
this implicitly assumes that such a circle is unique,
which we have not proved.
Such non-uniqueness can only happen if the edge labels around $R$ are all 0,
so that the $z(x_k)$'s alternate between two points,
i.e. $z(x₁) = z(x₃) = ... = z(x_{n-1})$ and $z(x₂) = z(x₄) = ... = z(x_n)$.
(Note that adding the additional criteria
from either \lemref{l:FAL-S3-ortn} or \lemref{l:FAL-S3-shape-param}
guarantees that all edge labels are nonzero,
so we would not be in this situation.)

To overcome this, we also need to study the circles circumscribing the bow-tie faces.
Before we continue with the argument,
we point out that while it may be a little long and pedantic,
especially since it is meant to prove an ostensibly improbably edge case,
the argument is not unnatural for someone familiar with FALs:
when we add these circles to the circle packing
from \prpref{p:geometric-univalence}, we get a
\emph{right-angled circle pattern} (see e.g. \cite{purcell}, \cite{kwon2020}).

Let $F = F_{C_i}^{±}$ be one of the bow-tie faces associated to
augmentation circle $C_i$.
The images of the vertices of $F$ under $z(-)$ must be distinct
(since they are connected by internal 1-cells of the polyhedral decomposition),
so there is a unique circle $X_F$ passing through them.
Let $F' ≠ F$ be a bow-tie face that shares a vertex $x$ with $F$.
Let $y₁,y₂,y₃,y₄$ be the vertices adjacent to $x$,
with $y₁,y₂$ in $F$ and $y₃,y₄$ in $F'$.
Let $R$ (resp. $R'$) be the regions meeting $x$ that also meets
$y₂,y₃$ (resp. $y₄,y₁$).
Applying a global isometry if necessary, assume that $z(x) = ∞$,
so that $X_F$ and $X_{F'}$ are straight lines.
Suppose $x$ corresponds to a non-augmentation circle segment $a_j ⊆ L$.
Then by \ref{i:cc1} and \ref{i:cc3}
$z(y₁),z(y₂),z(y₃),z(y₄)$ form a rectangle;
if $x$ instead corresponds to an augmentation circle $C_i ⊆ L$,
then we use \ref{i:cc5} and \ref{i:cc6}.
In particular, $X_F$ and $X_{F'}$ are parallel lines.
If we further have $z(y₂) ≠ z(y₃)$,
then $X_R$ is a straight line perpendicular to $X_F$ and $X_{F'}$.
Moreover, if $z(y₂) = z(y₃)$ but $X_R$ is still well-defined in that
$z(x') ≠ z(y₂),∞$ for some vertex $x'$ of $R$,
we must still have $X_R$ perpendicular to $X_F$ and $X_{F'}$,
since the (modified) edge label for the (added) diagonal from $y₂$ to $x'$
is real.

Thus, it makes sense to take $X_R$ to be the straight line
through $z(y₂)$ perpendicular to $X_F$ when there is no unique choice for $X_R$.
To justify this, we need to check that this gives the same circle $X_R$
regardless of choice of $x$ on $R$.
Let $x₁,...,x_n$ be the vertices of $R$ as before,
and let $F_k$ be the bow-tie face adjacent to $R$ and sharing the vertices
$x_k$ and $x_{k+1}$ with $R$.
The circle $X_{F₁}$ passes through $z(x₁)$ and $z(x₂)$.
The circle $X_{F₂}$ passes through $z(x₂)$ and $z(x₃)$,
and is tangent to $X_{F₁}$.
Recall that we must have the $z(x_k)$'s alternate between two points,
in particular $z(x₃) = z(x₁)$, so in fact we must have $X_{F₂} = X_{F₁}$.
Similarly, the circle $X_{F₃}$ passes through $z(x₃)$ and $z(x₄)$,
and is tangent to $X_{F₂}$, and since $z(x₄) = z(x₂)$,
we must have $X_{F₃} = X_{F₂}$.
Repeating this argument, we find that all the $X_{F_k}$'s are the same.
It follows that $X_R$ can be defined as the circle passing through
$z(x₁)$ and $z(x₂)$ that is orthogonal to $X_{F₁}$.

Since circles around regions $X_R$ and circles around bow-tie faces $X_F$
are orthogonal, it follows that the circles $X_R$, $X_{R'}$ for adjacent regions
are tangent.


Now we consider $L$ with half-twists.
Let $L'$ be the FAL obtained from $L$ by removing half-twists.
From \rmkref{r:half-twist-relation}, we learned that
an algebraic solutions $Ω$ for $L$ and $L'$
that satisfy the criteria in \lemref{l:FAL-S3-pure-imag}
are related by changing normalizations.
From \rmkref{r:meridian-scaling-geom-realz},
different normalizations amounts to changing the
parametrizations of the points $z(x_k)$, but not their relative positions;
in other words, the geometry is invariant with respect to normalization choices.
Thus, the arguments above can be transfered from $L'$ almost verbatim to $L$.

\textbf{Case $L ⊆ \TT$}:
We repeat the arguments above to $\wdtld{B_L}$ in place of $B_L$,
and we get a circle packing realizing $\wdtld{Γ_L}$.
From the discussion before the proposition,
the geometric realization of $\wdtld{B_L}$,
and hence the circle packing realizing of $\wdtld{Γ_L}$, is biperiodic
(invariant under a $\ZZ ⊕ \ZZ$ group of translations),
so the quotient is a circle packing realizing $Γ_L$.
\end{proof}

Next we show that the criteria from the other two lemmas,
Lemma \ref{l:FAL-S3-ortn} and \ref{l:FAL-S3-shape-param},
enforce additional properties on the circle packing,
namely local univalence (\defref{d:univalence})
and local order-preserving (\defref{d:locally-order}).

\begin{proposition}
\label{p:crit-02-circ-packing}
Let $L$ be a FAL in $S³$ (resp. $\TT$).
Let $Ω$ be an algebraic solution satisfying the criteria in
\lemref{l:FAL-S3-pure-imag} (resp. \lemref{l:FAL-T2-pure-imag}),
so that by \prpref{p:crit-01-circ-packing},
we have a circle packing in $S²$ (resp. $T²$) realizing the region graph $Γ_L$
obtained from a geometric realization $z(-)$ of $Ω$ on $B_L$.
If $Ω$ also satisfies the criteria in \lemref{l:FAL-S3-ortn}
(resp. \lemref{l:FAL-T2-ortn}),
then the circle packing is locally univalent.
\end{proposition}
\begin{proof}
Since we only need to work locally,
the proof for the $\TT$ case is essentially the same as for $S³$,
only having to change $B_L$ to $\wdtld{B_L}$.
Thus we only work with the $S³$ case.
As before, we assume that $L$ has no half-twists,
using \rmkref{r:half-twist-relation} to relate to $L$'s with half-twists.
Let the circle packing be denoted $\{X_{•}\}$.
We describe an interior filling locally at each point of tangency,
and show that these choices are globally consistent.

Consider a vertex $x ∈ B_L$ corresponding to an oriented segment $a_j$,
which goes from augmentation circle $C_i$ to $C_{i'}$.
The two peripheral edges $e_j^l$ and $e_j^r$
travel along $a_j$ from the triangular face $F_{C_i}^+$ to $F_{C_{i'}}^+$.
Let us suppose, without loss of generality, that
$e_{i,+}¹$ and $e_{i',+}¹$ (instead of $e_{i,+}²$ and/or $e_{i',+}²$)
run from the ends of $a_j^l$ to $a_j^r$.
Then the four peripheral edges form the boundary of a rectangle
$a_j^l e_{i',+}¹ (a_j^r)^\inv (e_{i,+}¹)^\inv$,
which is non-degenerate by the criterion $u(e_j^l) ∈ i\RR^{<0}$,
and $a_j^l e_{i',+}¹ (a_j^r)^\inv (e_{i,+}¹)^\inv$
goes counterclockwise about the rectangle;
in particular, the rectangle lies to the left of $a_j^l$
and to the right of $a_j^r$.

Let $R$ and $R'$ be the regions to the left and right of $a_j$ respectively.
Let $z = z(x)$ be the point of tangency between $X_R$ and $X_{R'}$.
Let $H$ be the horosphere centered at $z$.
We find segments on $H$ corresponding to the four peripheral edges
$a_j^l, a_j^r, e_{i,+}¹, e_{i',+}¹$, which we also denote with the same symbols.
Let $β^l$ and $β^r$ be the parallel lines running through $a_j^l$ and $a_j^r$
respectively (with the same orientations).
Denote by $Q^l$ (resp. $Q^r$) the half-plane demarcated by
$β^l$ (resp. $β^r$) which does not contain the rectangle.
($Q^l$ is to the right of $β^l$, and $Q^r$ is to the left of $β^r$.)
Projecting $β^l$ onto $∂\HH$ by emanating geodesics from $z$,
we get a circle on $∂\HH$ that coincides with $X_R$;
likewise, $β^r$ projects onto $X_{R'}$.
Similarly, $Q^l$ and $Q^r$ project onto disks in $∂\HH$
whose boundaries are $X_R$ and $X_{R'}$ respectively;
we take these disks to be the interior fillings for $X_R$ and $X_{R'}$.

For a vertex $x ∈ B_L$ corresponding to an augmentation circle $C_i$,
we carry out a similar procedure, mostly swapping out one set of symbols for another.
Replace $a_j^l$ by $e_i^{μ2}$ and $a_j^r$ by $e_i^{μ1}$.
Replace both $e_{i,+}¹$ and $e_{i',+}¹$ by $e_{i,+}⁰$.
Then $e_i^{μ1} e_{i,+}⁰ (e_i^{μ2})^\inv (e_{i,+}⁰)^\inv$ forms a rectangle.
Note that this rectangle is already guaranteed to be non-degenerate
by the criteria in \lemref{l:FAL-S3-pure-imag};
however, we still need the new criteria to ensure that
$e_i^{μ1} e_{i,+}⁰ (e_i^{μ2})^\inv (e_{i,+}⁰)^\inv$ goes around the rectangle
in a counterclockwise direction, and in particular the rectangle lies
to the left of $e_i^{μ2}$ and to the right of $e_i^{μ1}$.
Here we take $β^l$ to be the line passing through $e_i^{μ2}$
and $β^r$ the line passing through $e_i^{μ1}$
(again, with the same orientations).
We take $R$ and $R'$ to be the regions that $C_i$ punctures
($R$ meeting $e_i^{μ2}$ and $R'$ meeting $e_i^{μ1}$).
Then the rest of the procedure is repeated.

It remains to check that these choices of fillings is globally a well-defined
interior filling, since it is clearly locally univalent by construction.
Consider a region $R$ and its corresponding face $F_R$.
Let $a_j$ be a segment that is part of the boundary of $R$.
If the orientation of $a_j$ agrees with the outward orientation on $∂R$,
then the peripheral edge in $F_R$ near $a_j$ is $a_j^l$;
on the other hand, if the orientations disagree,
then the peripheral edge in $F_R$ near $a_j$ is $a_j^r$.
Now let $p$ a point of intersection of some augmentation circle $C_i$ with $R$,
which gives rise to a peripheral edge, $e_i^{μ1}$ or $e_i^{μ2}$, of $F_R$.
In the case it is $e_i^{μ2}$, the orientation of $e_i^{μ2}$
agrees with the outward orientation on $R \backslash p$,
and in the other case, the orientation of $e_i^{μ1}$ disagrees.
To summarize, the orientation agreeing cases are $a_j^l$ and $e_i^{μ2}$,
and the orientation disagreeing cases are $a_j^r$ and $e_i^{μ1}$.
The former are those peripheral edges which lie to the left of their rectangle,
while the latter are those which lie to the right of their rectangle.
In other words, we can choose an orientation on $X_R$ so that
the projection from $β^l$ is orientation-preserving
and the projection from $β^r$ is orientation-reversing,
and moreover the rectangles all lie ``to the left'' of $X_R$,
so the filling for $X_R$ is the disk to its right.
\end{proof}

\begin{proposition}
\label{p:crit-03-circ-packing}
Let $L$ be a FAL in $S³$ (resp. $\TT$).
Let $Ω$ be an algebraic solution satisfies the criteria in
\lemref{l:FAL-S3-pure-imag} (resp. \lemref{l:FAL-T2-pure-imag}),
so that by \prpref{p:crit-01-circ-packing},
we have a circle packing in $S²$ (resp. $T²$) realizing the region graph $Γ_L$
obtained from a geometric realization $z(-)$ of $Ω$ on $B_L$.
If $Ω$ also satisfies the criteria in \lemref{l:FAL-S3-shape-param}
(resp. \lemref{l:FAL-T2-shape-param}),
then the circle packing is locally order-preserving.
\end{proposition}

\begin{proof}
For distinct vertices $v₀,v₁,v₂,v₃$ of $F_R$ in cyclic order
(not necessarily consecutive), the criterion implies that the
geometric shape parameter $ζ_{\wdtld{γ};\wdtld{γ'},\wdtld{γ''}} ∈ (0,1)$
where $\wdtld{γ},\wdtld{γ'},\wdtld{γ''}$ are the geodesics with endpoints
$\{z(v₁),z(v₂)\}, \{z(v₀),z(v₁)\}, \{z(v₂),z(v₃)\}$ respectively.
Thus, as points on $X_R$, $z(v₀),z(v₁),z(v₂),z(v₃)$ are also in some cyclic order,
which means that the circle packing is locally order-preserving.
\end{proof}

Before we state the main theorem of this section,
we note that the criteria of 
\lemref{l:FAL-S3-ortn}/\ref{l:FAL-T2-ortn})
(responsible for local univalence)
are slightly too stringent, in the following sense.
If $Ω$ is an algebraic solution satisfying those criteria
(in addition to those of \lemref{l:FAL-S3-pure-imag}/\ref{l:FAL-T2-pure-imag}),
then the (pointwise) complex conjugate $\ov{Ω}$ doesn't satisfy
the criteria of \lemref{l:FAL-S3-ortn}/\ref{l:FAL-T2-ortn},
but clearly leads to a geometric realization of $B_L$/$\wdtld{B_L}$
that is flipped (under a complex conjugation) relative to that constructed from $Ω$,
and likewise for the circle packing realizing $Γ_L$/$\wdtld{Γ_L}$.
Local univalency of a circle packing is preseved under complex conjugation,
so it is natural to relax the condition defining
the left-hand side of the second map in the theorem.
(We do not need to do that for the first and third map,
since $\ov{Ω}$ would already be in those sets.)

\begin{theorem}
\label{t:FAL-circ-packing}
Let $L$ be a FAL in $S³$ (resp. $\TT$).
\prpref{p:crit-01-circ-packing} defines an assignment
\[
\Set{
\begin{array}{l}
\text{algebraic solution satisfying} \\
\text{criteria in \lemref{l:FAL-S3-pure-imag} (resp. \ref{l:FAL-T2-pure-imag})}
\end{array}
}
\to
\Set{
\begin{array}{l}
\text{circle packing realizing } Γ_L \text{ in }
	S² \text{(resp. }T²\text{)}\\
\text{up to conformal equivalence}
\end{array}
}
\]
and by \prpref{p:crit-02-circ-packing} and \prpref{p:crit-03-circ-packing},
this assignment restricts to
\[
\Set{
\begin{array}{l}
\text{algebraic solution or its complex conjugate} \\
\text{satisfying criteria \lemref{l:FAL-S3-pure-imag} (resp. \lemref{l:FAL-T2-pure-imag})} \\
\text{and \lemref{l:FAL-S3-ortn} (resp. \lemref{l:FAL-T2-ortn})} \\
\end{array}
}
\to
\Set{
\begin{array}{l}
\text{locally univalent} \\
\text{circle packing realizing } Γ_L \text{ in }
	S² \text{(resp. }T²\text{)}\\
\text{up to conformal equivalence}
\end{array}
}
\]
\[
\Set{
\begin{array}{l}
\text{algebraic solution satisfying} \\
\text{criteria in \lemref{l:FAL-S3-pure-imag} (resp. \ref{l:FAL-T2-pure-imag})} \\
\text{and \lemref{l:FAL-S3-shape-param} (resp. \ref{l:FAL-S3-shape-param})}
\end{array}
}
\to
\Set{
\begin{array}{l}
\text{locally order-preserving} \\
\text{circle packing realizing } Γ_L \text{ in }
	S² \text{(resp. }T²\text{)}\\
\text{up to conformal equivalence}
\end{array}
}
\]
Then these maps are bijections.
\end{theorem}

\begin{proof}
We construct the inverse to the first map.
We will assume that $L$ has no half-twists;
if it does, following \rmkref{r:half-twist-relation},
we apply the construction of $Ω$ as below to $L'$,
then convert it to an algebraic solution for $L$.
We also mostly only show work for $L$ in $S³$, as all the constructions are
made with local choices/data; we make a comment concerning the $\TT$ case later.

Let $Ξ = \{X_R\}$ be a circle packing realizing $Γ_L$.
Consider a point of tangency $z$ of $Ξ$ corresponding to a vertex $v$ of $B_L$.
We want to choose a parametrization of $z$ so that
the criteria of \lemref{l:FAL-S3-pure-imag} are satisfied.
Suppose $v$ corresponds to a segment $a_j$, ending at some $C_i$.
Let $γ = γ_i^k, γ' = γ_i⁰$, with $k = 1$ or 2, be the crossing arcs,
with the peripheral edge $e_{i,+}^k$ going from $γ$ to $γ'$.
Let $x,x'$ be the other endpoints of $γ,γ'$ (as edges of $B_L$), and
let $y,y'$ be the points of tangency in $Ξ$ corresponding to $x,x'$, respectively.
Let $γ_\HH, γ_\HH'$ be the geodesics in $\HH$ connecting $z$ to $y,y'$
respectively.
Then there is a unique parametrization of $z$ such that
the geometric edge label from $γ_\HH$ to $γ_\HH'$ is $1/2$,
which is expected of $u(e_{i,+}^k$.
We can consider a similar argument for vertices $v$ corresponding to
an augmentation circle $C_i$, using $e_i^{μ1}$ instead of $e_{i,+}^k$.

Thus do we choose parametrizations for the points of tangency of $Ξ$,
with which we can compute geometric crossing and edge labels.
We construct the labeling $Ω$ on $L$ out of the relevant geometric labels;
it is clear that the circle packing constructed from $Ω$
as in \prpref{p:crit-01-circ-packing} is $Ξ$.
We leave it to the reader to check that $Ω$ satisfies all the criteria
of \lemref{r:half-twist-relation}.
For the $\TT$ case, it suffices to note that since the circle packing for
$\wdtld{B_L}$ realized in $\CC$ is biperiodic, then the geometric labels
are invariant under the biperiodicty, so $Ω$ is well-defined.

For the third map, considering a locally order-preserving $Ξ$,
the construction above of $Ω$ easily implies the criteria of
\lemref{l:FAL-S3-shape-param}/\ref{l:FAL-T2-shape-param}.

For the second map, consider locally univalent $Ξ$,
and the $Ω$ constructed above that gives rise to it.
Again the case of $\TT$ is similar to $S³$, so we only consider the latter.
If the criteria of \lemref{l:FAL-S3-ortn} are violated at every vertex,
that is, $\geomu{}(e_{i,+}⁰) / u(e_i^{μ1}) ∈ i\RR^{>0}$ for all $C_i$
and $\geomu{}(e_j^l) ∈ i\RR^{>0}$ for all $a_j$,
then the complex conjugate labeling $\ov{Ω}$ satisfies
the criteria of \lemref{l:FAL-S3-ortn}.
Suppose otherwise, say the criteria is satisfied at some vertex $v ∈ B_L$.
Let $R$ be a region that has $v$ in its boundary,
and let $v'$ be another vertex in its boundary.
From the proof of \prpref{p:crit-03-circ-packing},
it follows easily that the criteria must be satisfied at $v'$ as well.
Repeatedly applying this argument, this shows that $Ω$ indeed satisfies
the criteria of \lemref{l:FAL-S3-ortn}.
\end{proof}

Toegether with \lemref{l:univalence-is-local}, we have:

\begin{corollary}
\label{c:geometric}
Let $L$ be a FAL in $S³$ (resp. $\TT$),
and let $Ω$ be an algebraic solution to the TT method on $L$.
Then $Ω$ is the geometric labeling if and only if it satisfies the criteria in
Lemmas \ref{l:FAL-S3-pure-imag}, \ref{l:FAL-S3-ortn},
and \ref{l:FAL-S3-shape-param}
(resp. \ref{l:FAL-T2-pure-imag}, \ref{l:FAL-T2-ortn},
and \ref{l:FAL-T2-shape-param}).
\end{corollary}


\section{Examples}
\label{s:examples}

The following example is the TT method on the standard
Square Weave in the thickened torus.
We use slightly different notation from \secref{s:TT-thickened-torus}
to reduce clutter (see \figref{f:square-weave}).
We first deal with labels from the original TT method
(that is, labels not involving $\north$ nor $\south$).
The edge equations of the TT method give us the following set of equations:

\begin{figure}
\centering
\includegraphics[width=10cm]{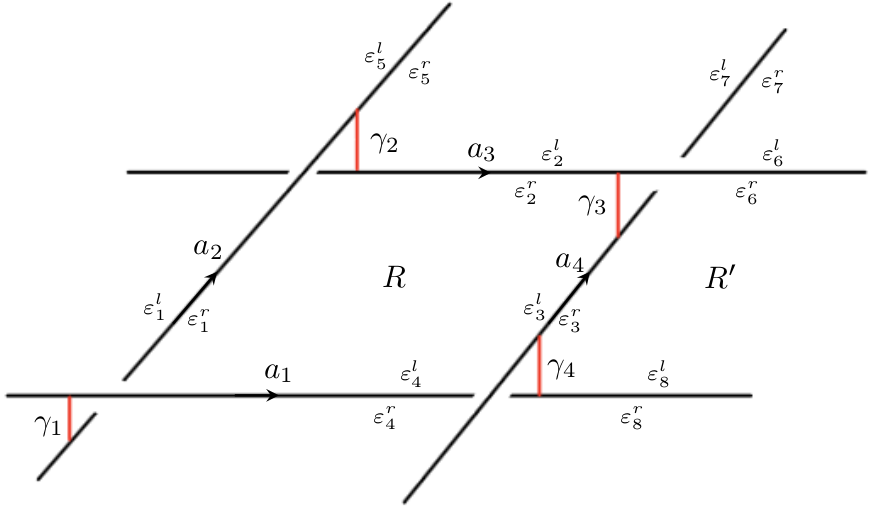}
\includegraphics[width=5cm]{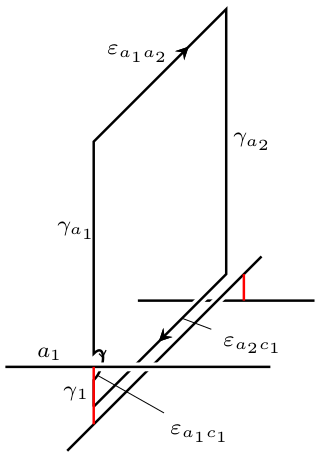}
\caption{The square weave.
}
\label{f:square-weave}
\end{figure}

\begin{equation}
\label{e:square-edge-eqn}
u₁^r - u₁^l =
u₂^r - u₂^l =
-(u₃^r - u₃^l) =
-(u_4^r - u_4^l) =
-(u_5^r - u_5^l) =
-(u_6^r - u_6^l) =
u_7^r - u_7^l =
u_8^r - u_8^l =
1
\end{equation}

Under a symmetry of the left diagram in \figref{f:square-weave},
namely the clockwise rotation by $90^{∘}$ preserving $R$,
the peripheral edges and crossing arcs get sent to each other:

\begin{equation}
\label{e:rot-90-sym-label}
ε₁^r
\mapsto ε₂^r
\mapsto -ε₃^l
\mapsto -ε_4^l
\mapsto ε₁^r
\;\;\;
;
\;\;\;
ε₁^l
\mapsto ε₂^l
\mapsto -ε₃^r
\mapsto -ε_4^r
\mapsto ε₁^l
\;\;\;
;
\;\;\;
γ_{c₁}
\mapsto γ_{c₂}
\mapsto γ_{c₃}
\mapsto γ_{c_4}
\mapsto γ_{c₁}
\end{equation}

Note that the meridians also get sent to meridians,
but with sign changes:
If $μ_{a_i}$ refers to the meridian
on the overpass segment $a_i$,
then

\begin{equation}
\label{e:rot-90-sym-meridian}
μ_{a₁} \mapsto
-μ_{a₂} \mapsto
-μ_{a₃} \mapsto
μ_{a_4} \mapsto
μ_{a₁}
\end{equation}

Thus, we have:

\begin{equation}
\label{e:u-w}
u := u₁^r = u₂^l = u₃^l = u_4^r
\;\;\;
;
\;\;\;
w := w₁ = -w₂ = w₃  = -w_4
\end{equation}

where in the left half, the minus signs in \eqnref{e:rot-90-sym-label}
get canceled out by the minus signs in \eqnref{e:rot-90-sym-meridian}).

We get the following shape parameters for the region $R$:

\begin{equation}
ζ_{1,R} = \frac{+w₁}{u_4^l u₁^r}
= \frac{w}{u²}
\;\; , \;\;
ζ_{2,R} = \frac{-w₂}{u₁^r u₂^r}
= \frac{w}{u²}
\;\; , \;\;
ζ_{3,R} = \frac{+w₃}{u₂^r u₃^l}
= \frac{w}{u²}
\;\; , \;\;
ζ_{4,R} = \frac{-w_4}{u₃^l u_4^l}
= \frac{w}{u²}
\end{equation}

As one would expect from the rotation symmetry,
the shape parameters above are all equal.

The shape parameters satisfy the following relations
(from $f_4 = 0$ in \lemref{l:region-eqn-k4}):
\begin{equation}
ζ_{2,R} + ζ_{3,R} - 1 = 0
\;\;,\;\;
ζ_{3,R} + ζ_{4,R} - 1 = 0
\;\;,\;\;
ζ_{1,R} + ζ_{2,R} - 1 = 0
\end{equation}
Since $ζ_{1,R} = ζ_{2,R} = ζ_{3,R} = ζ_{4,R}
= \frac{w}{u²}$,
each of the above relations reduce to   
$2(\frac{w}{u²}) - 1 = 0
\implies w = \frac{1}{2}u²$.

By applying the same process to $R'$,
we have
$ζ_{1,R'} = ζ_{2,R'} = ζ_{3,R'} = ζ_{4,R'}
= \frac{1}{2}$,
and $ζ_{1,R'} = -\frac{w}{u'^2}$,
where $u' = u₁^l = u_8^l = u₃^r = u_6^r$.
But we also have, by edge equations \eqnref{e:square-edge-eqn}, that
$u' = u₃^r = u₃^l - 1 = u - 1$.
So we have
$w = -\frac{1}{2}(u-1)²$.

Thus we have $-u² = (u-1)²$, from which we obtain
$u = \frac{1}{2} \pm \frac{1}{2}i$ and 
$w = \pm \frac{1}{4}i$


Observe that reflection of the link diagram across a link component, say $a₂$,
is a symmetry of the link diagram,
which induces an orientation-reversing homeomorphism of the link complement
that also fixes each link component,
thus acts on each peripheral torus by an anti-$\CC$-linear automorphism.
For example, this symmetry sends
$μ₂$ to $-μ₂$ and thus acts as $z \mapsto -\ov{z}$;
since it sends $ε₁^l$ to $ε₁^r$,
we have
$u = u₁^r = -\ov{u₁^l} = -\ov{u'} = 1 - \ov{u}$,
which is consistent with $u = \frac{1}{2} \pm \frac{1}{2}i$.

\ \\

Next, we solve the equations involving
new crossing labels and edge labels,
i.e., those involving $\north,\south$ (the Hopf link components).
 
The vertical region in \figref{f:square-weave}
gives rise to a vertical region equation.
It is a truncated ideal triangle
bounded by crossing arcs and peripheral edges
$γ_{a₁}, ε_{a₁ c₁}, γ_{c₁}, ε_{a₂ c₁}, γ_{a₂}, ε_{a₁ a₂}$,
and we have (by \lemref{l:region-eqn-k3}):

\begin{equation}
\label{e:hopf-link-wall}
\frac{w_{a₁}}{u_{a₁ a₂} u_{a₁ c₁}}
=
\frac{w₁}{u_{a₁ c₁} u_{a₂ c₁}}
=
-\frac{w_{a₂}}{u_{a₁ a₂} u_{a₂ c₁}}
= 1
\end{equation}

The symmetry of reflecting across $a₂$
preserves the vertical region in \figref{f:square-weave},
so $u_{a₂ c₁} = -\frac{1}{2}(u₁^l + u₁^r) = -\frac{1}{2}(u' + u)
= \mp\frac{1}{2}i$.
Symmetry across $a₁$ gives $u_{a₁ c₁} = -\frac{1}{2}$.



Note that $u_{a₁ a₂}$ depends on the choice of meridian on $\north$;
we choose the meridian to be homotopic to $a₂$,
so that $u_{a₁ a₂}  = \frac{1}{2}$.

Then from \eqnref{e:hopf-link-wall},
\begin{equation}
\label{e:blue-region-consequences}
w_{a₁} = -\frac{1}{4}
\;\; , \;\;
w₁ = \pm \frac{1}{4} i
\;\; , \;\;
w_{a₂} = \pm \frac{1}{4} i
\;\;\;\;
,
\;\;\;\;
u_{a₂ c₁} = \mp\frac{1}{2} i
\;\; , \;\;
u_{a₁ c₁} = -\frac{1}{2}
\;\; , \;\;
u_{a₁ a₂} = \frac{1}{2}
\end{equation}


The clockwise $90^{∘}$-rotation about $R$ of the diagram
sends $γ_{a₁}$ to $γ_{a₂}$,
so we must have, and indeed we do, $|w_{a₁}| = |w_{a₂}|$.
The rotation sends the meridian $μ_N$ at $\north$ to $-i$,
and further careful consideration shows that we must have
$w_{a₁} = -i w_{a₂}$,
and thus we pick out a sign for $u,w$ of \eqnref{e:u-w}:



\begin{equation}
u = \frac{1}{2} - \frac{1}{2} i
\;\;\;,\;\;\;
w = - \frac{1}{4} i
\end{equation}

Note that choosing the other sign would give the algebraic solution
for the same link complement but with reversed orientation
(since the rotation would send the meridian $μ_N$ at $\north$ to $i$).

All other labels are obtained by symmetry
(e.g. rotating by $180^{∘}$ about a crossing).
The algebraic solution is unique,
and hence is also the geometric solution.




\begin{example}
\label{x:counter-ex}
In \figref{f:counter-xmp-FAL},
we present an example of a FAL in $S³$ with more than one
algebraic solution to the TT method.
By the correspondence between algebraic solutions and circle packings
discussed in \thmref{t:FAL-circ-packing},
we describe two such solutions, one geometric and the other not geometric,
by presenting their corresponding circle packings.

\begin{figure}
\centering
\includegraphics[height=6cm]{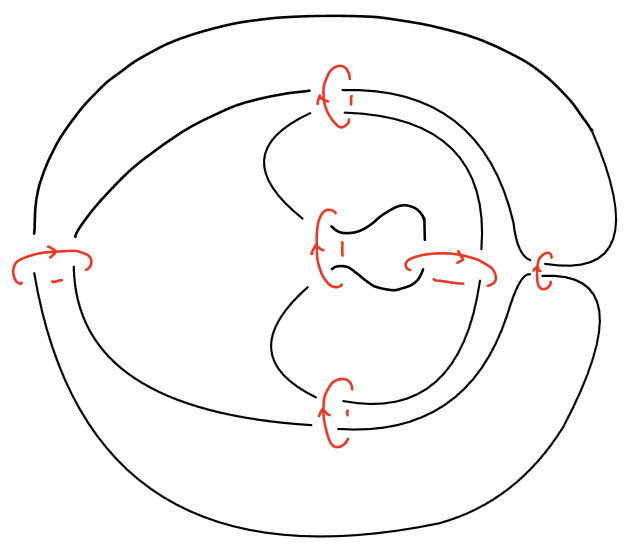}
\includegraphics[height=6cm]{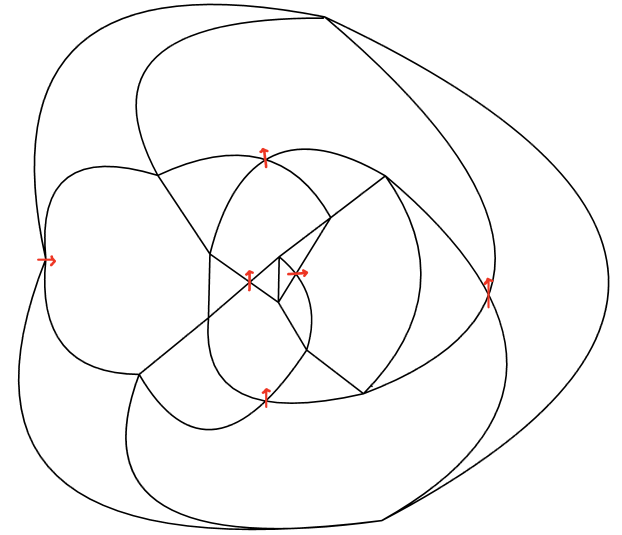}
\caption{
Example of FAL in $S³$ with more than one algebraic solution.
The red arrows on the right diagram are the centers of the bow-tie.
}
\label{f:counter-xmp-FAL}
\end{figure}

\begin{figure}[H]
\centering
\includegraphics[height=6cm]{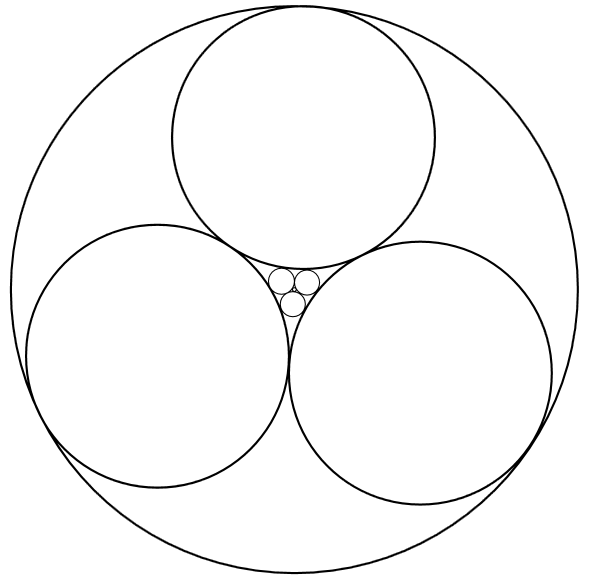}
\includegraphics[height=6cm]{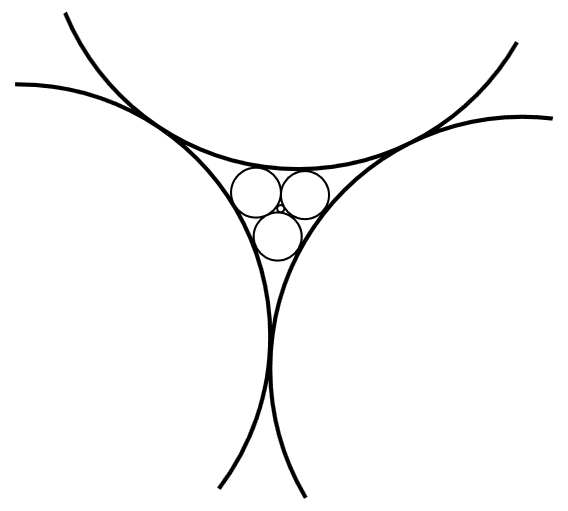}
\caption{
Circle packing corresponding to geometric solution;
the right diagram is a close-up of the center of the left diagram;
note in the very center there is one more tiny circle,
which corresponds to the smallest (non-bow-tie) region in
\figref{f:counter-xmp-FAL}.
}
\label{f:counter-xmp-good}
\end{figure}

\begin{figure}[H]
\centering
\includegraphics[height=6cm]{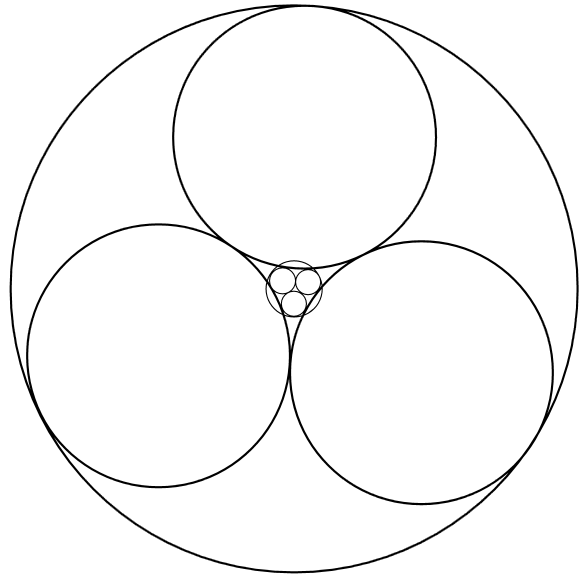}
\includegraphics[height=6cm]{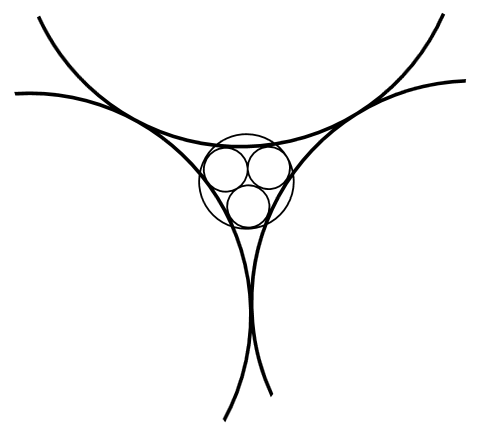}
\caption{
Circle packing corresponding to non-geometric solution.
The tiny circle in \figref{f:counter-xmp-good} is replaced
by a bigger circle which is still tangent to the same three circles.
Local univalence can be satisfied
(by choosing the interior of this circle to be outside),
but univalence clearly cannot be satisfied.
}
\label{f:counter-xmp-bad}
\end{figure}

\end{example}

\bibliographystyle{plain}
\bibliography{references}

\begin{bibdiv}
\begin{biblist}

\bib{BandS}{article}{
      author={Bobenko, Alexander~I.},
      author={Springborn, Boris~A.},
       title={Variational principles for circle patterns and {K}oebe's
  theorem},
        date={2004},
        ISSN={0002-9947},
     journal={Trans. Amer. Math. Soc.},
      volume={356},
      number={2},
       pages={659\ndash 689},
  url={https://doi-org.ezproxy.gc.cuny.edu/10.1090/S0002-9947-03-03239-2},
      review={\MR{2022715}},
}

\bib{CKP}{article}{
      author={Champanerkar, Abhijit},
      author={Kofman, Ilya},
      author={Purcell, Jessica~S},
       title={Geometry of biperiodic alternating links},
        date={2019},
     journal={Journal of the London Mathematical Society},
      volume={99},
      number={3},
       pages={807\ndash 830},
}

\bib{dunfield}{article}{
      author={Dunfield, Nathan~M},
       title={Cyclic surgery, degrees of maps of character curves, and volume
  rigidity for hyperbolic manifolds},
        date={1998},
     journal={arXiv preprint math},
        ISSN={9802022/},
}

\bib{rochyTT}{misc}{
      author={Flint, Rochy},
       title={Intercusp geodesics and cusp shapes of fully augmented links},
        date={2018},
}

\bib{repnvol}{article}{
      author={Francaviglia, Stefano},
       title={Hyperbolic volume of representations of fundamental groups of
  cusped 3-manifolds},
        date={200401},
        ISSN={1073-7928},
     journal={International Mathematics Research Notices},
      volume={2004},
      number={9},
       pages={425\ndash 459},
  eprint={https://academic.oup.com/imrn/article-pdf/2004/9/425/2532216/2004-9-425.pdf},
         url={https://doi.org/10.1155/S1073792804131619},
}

\bib{koebe}{book}{
      author={Koebe, Paul},
       title={Kontaktprobleme der konformen abbildung},
   publisher={Hirzel Stuttgart},
        date={1936},
}

\bib{kwontham}{article}{
      author={Kwon, Alice},
      author={Tham, Ying~Hong},
       title={Hyperbolicity of augmented links in the thickened torus},
        date={2022},
     journal={Journal of Knot Theory and Its Ramifications},
      volume={31},
      number={04},
       pages={2250025},
      eprint={https://doi.org/10.1142/S0218216522500250},
         url={https://doi.org/10.1142/S0218216522500250},
}

\bib{kwon2020}{misc}{
      author={Kwon, Alice},
       title={Fully augmented links in the thickened torus},
        date={2020},
        note={accepted for publication in Algebraic \& Geometric Topology},
}

\bib{lackenby}{article}{
      author={Lackenby, Marc},
       title={The volume of hyperbolic alternating link complements},
        date={2004},
        ISSN={0024-6115},
     journal={Proc. London Math. Soc. (3)},
      volume={88},
      number={1},
       pages={204\ndash 224},
         url={https://doi-org.ezproxy.gc.cuny.edu/10.1112/S0024611503014291},
        note={With an appendix by Ian Agol and Dylan Thurston},
      review={\MR{2018964}},
}

\bib{matveev}{article}{
      author={Matveev, Sergei~V},
       title={Transformations of special spines and the zeeman conjecture},
        date={1988},
     journal={Mathematics of the USSR-Izvestiya},
      volume={31},
      number={2},
       pages={423},
}

\bib{intercusp}{article}{
      author={Neumann, Walter},
      author={Tsvietkova, Anastasiia},
       title={Intercusp geodesics and the invariant trace field of hyperbolic
  3-manifolds},
        date={2016},
     journal={Proceedings of the American Mathematical Society},
      volume={144},
      number={2},
       pages={887\ndash 896},
}

\bib{piergallini}{article}{
      author={Piergallini, Riccardo},
       title={Standard moves for standard polyhedra and spines},
        date={1988},
     journal={Rend. Circ. Mat. Palermo (2) Suppl},
      volume={18},
       pages={391\ndash 414},
}

\bib{negative-triangulation}{article}{
      author={Petronio, C.},
      author={Porti, J.},
       title={Negatively oriented ideal triangulations and a proof of hurston's
  hyperbolic dehn filling},
        date={1999jun},
        ISSN={0723-0869},
     journal={Exposition. Math.},
      volume={18},
      number={1},
       pages={1\ndash 35},
}

\bib{purcell}{article}{
      author={Purcell, Jessica~S},
       title={An introduction to fully augmented links},
        date={2011},
     journal={Interactions between hyperbolic geometry, quantum topology and
  number theory},
      volume={541},
       pages={205\ndash 220},
}

\bib{Stephenson}{article}{
      author={Stephenson, Kenneth},
       title={Circle packing: a mathematical tale},
        date={2003},
     journal={Notices of the AMS},
      volume={50},
      number={11},
       pages={1376\ndash 1388},
}

\bib{TTmethod}{article}{
      author={Thistlethwaite, Morwen},
      author={Tsvietkova, Anastasiia},
       title={An alternative approach to hyperbolic structures on link
  complements},
        date={2014},
     journal={Algebraic \& Geometric Topology},
      volume={14},
      number={3},
       pages={1307\ndash 1337},
}

\end{biblist}
\end{bibdiv}



\end{document}